\documentclass[11pt,reqno]{amsart}
\usepackage{fullpage}
\usepackage{times}
\usepackage[T1]{fontenc}
\usepackage[utf8]{inputenc}
\usepackage{amsmath,amsfonts,amssymb,amsxtra,setspace,xspace,graphicx,lmodern,fourier,mathrsfs}
\usepackage[colorlinks=true]{hyperref}
\hypersetup{urlcolor=blue, citecolor=red, linkcolor=blue}
\usepackage[usenames,dvipsnames]{color}
\usepackage[misc]{ifsym}
\usepackage{cleveref}

\newcommand{\R}{{\mathbb R}}
\newcommand{\N}{{\mathbb N}}

\newcommand{\T}{{\mathbb T}}

\newcommand{\be}[1]{\begin{equation*}\label{#1}}

\newtheorem{thm}{Theorem}[section]
\newtheorem{cor}[thm]{Corollary}
\newtheorem{lem}[thm]{Lemma}
\newtheorem{rmk}[thm]{Remark}

{ \theoremstyle{remark} }

\begin{document}

\title[De Giorgi Argument for non-cutoff Boltzmann equation with soft potentials ]
{ De Giorgi Argument for non-cutoff Boltzmann equation with soft potentials }

\author[C. Cao]{Chuqi Cao}
\address[Chuqi Cao]{Yau Mathematical Science Center and Beijing Institute of Mathematical Sciences and Applications, Tsinghua University\\
Beijing 100084,  P. R.  China.} \email{chuqicao@gmail.com}

\begin{abstract}
In this paper, we consider the global well-posedness to the non-cutoff Boltzmann equation with soft potential in the $L^\infty$ setting. We show that when the initial data is close to equilibrium and the perturbation is small in $L^2 \cap L^\infty$ polynomial weighted space, the Boltzmann equation has a global solution in the weighted $L^2 \cap L^\infty$ space. The ingredients of the proof lie in  strong averaging lemma, new polynomial weighted estimate for the non-cutoff Boltzmann equation and the $L^2$ level set Di Giorgi iteration method developed in \cite{AMSY2}.  The convergence to the equilibrium state in both $L^2$ and $L^\infty$ spaces is also proved.

\end{abstract}

\maketitle

\setcounter{tocdepth}{1}
\tableofcontents

\section{Introduction}
The Boltzmann equation reads
\begin{equation}
\label{Boltzmann equation}
\partial_t F +v \cdot \nabla_x F =Q(F, F), \quad F(0, x, v) =F_0(x, v),
\end{equation}
where $F(t, x, v) \ge 0$ is a distributional functions of colliding particles which, at time $t>0$ and position $x \in \T^3$, move with velocity $v \in \R^3$. We remark that the Boltzmann equation is one of the fundamental equations of mathematical physics and is a cornerstone of statistical physics. The Boltzmann collision operator $Q$ is a bilinear operator which acts only on the velocity variable $v$, that is
\[
 Q(G,F)(v)=\int_{\R^3}\int_{\mathbb{S}^2}B(v-v_*,\sigma)(G'_*F'-G_*F)d\sigma dv_*.
\]
Let us give some explanations on the collision operator.
\begin{enumerate}
\item  We use the standard shorthand $F=F(v),G_*=G(v_*),F'=F(v'),G'_*=G(v'_*)$, where $v',v'_*$ are given by
\[
v'=\frac{v+v_*}{2}+\frac{|v-v_*|}{2}\sigma, \quad v_*'=\frac{v+v_*}{2}-\frac{|v-v_*|}{2}\sigma, \quad \sigma\in\mathbb{S}^2.
\]
 This representation follows the parametrization of set of solutions of the physical law of elastic collision:
\[
v+v_*=v'+v'_*,  \quad  |v|^2+|v_*|^2=|v'|^2+|v'_*|^2.
\]

\item  The nonnegative function $B(v-v_*,\sigma)$ in the collision operator is called the Boltzmann collision kernel. It is always assumed to depend only on $|v-v_*|$ and the deviation angle $\theta$ through $\cos\theta:= \frac{v-v_*}{|v-v_*|}\cdot\sigma$.
\item  In the present work,  our {\bf basic assumptions on the kernel $B$}  can be concluded as follows:
\begin{itemize}
\item[$\mathbf{(A1).}$] The Boltzmann kernel $B$ takes the form: $B(v-v_*,\sigma)=|v-v_*|^\gamma b(\frac{v-v_*}{|v-v_*|}\cdot\sigma)$, where   $b$ is a nonnegative function.

\item[$\mathbf{(A2).}$] The angular function $b(\cos \theta)$ is not locally integrable and it satisfies
\[
\mathcal{K}\theta^{-1-2s}\leq \sin\theta b(\cos\theta)\leq \mathcal{K}^{-1}\theta^{-1-2s},~\mbox{with}~0<s<1,~\mathcal{K}>0.
\]

\item[$\mathbf{(A3).}$]
The parameter $\gamma$ and $s$ satisfy the condition $-3<\gamma  \le 0, s \in (0, 1)$ and $\gamma+2s>-1.$

\item[$\mathbf{(A4).}$]  Without lose of generality, we may assume that $B(v-v_*,\sigma)$ is supported in the set $0\le \theta \le \pi/2$, i.e.$\frac{v-v_*}{|v-v_*|}\cdot\sigma \ge 0$, for otherwise $B$ can be replaced by its symmetrized form:
\[
\overline{B}(v-v_*,\sigma )=|v-v_*|^\gamma\big(b(\frac{v-v_*}{|v-v_*|}\cdot\sigma )+b(\frac{v-v_*}{|v-v_*|}\cdot(-\sigma))\big) \mathrm{1}_{\frac{v-v_*}{|v-v_*|}\cdot\sigma \ge0},
\]
where $\mathrm{1}_A$ is the characteristic function of the set $A$.
\end{itemize}
\end{enumerate}

\begin{rmk} For inverse repulsive potential, it holds that $\gamma = \frac {p-5} {p-1}$ and $s = \frac 1 {p-1}$ with $p > 2$. It is easy to check that $\gamma + 4s = 1$ which means that assumption $\gamma+2s > -1$  is satisfied for the full range of the inverse power law model. Generally, the case $\gamma > 0$,  $\gamma = 0$, and  $\gamma < 0$ correspond to so-called hard, Maxwellian, and soft potentials respectively. Assumption  $\mathbf{(A3)}$ corresponds to soft potential and Maxwellian molecule case. 
\end{rmk}

\subsection{Basic properties and the perturbation equation} We recall some basic facts on the Boltzmann equation.
 \smallskip

\noindent$\bullet$ {\bf Conservation Law.}  Formally if $F$ is the solution to the Boltzmann equation \eqref{Boltzmann equation} with the initial data $F_0$, then it enjoys the conservation of mass, momentum and the energy, that is,
\begin{equation}
\label{conservation law}
\frac {d}{dt}\int_{\T^3\times{\R}^3} F(t, x, v)\varphi(v)dvdx= 0,\quad   \varphi(v)=1, v, |v|^2.
\end{equation}
For simplicity, we introduce   the normalization identities on the initial data $F_0$ which satisfies
\[
\int_{\T^3\times\R^3} F_0(x, v)\, dvdx= 1, \quad \int_{\T^3\times\R^3} F_0(x, v)\, v\, dvdx =0,
\quad \int_{\T^3\times\R^3} F_0(x, v)\, |v|^2 \, dvdx =3.
\]
This means that the  equilibrium associated to \eqref{Boltzmann equation} will be the standard Gaussian function, i.e.
\[
\mu(v) := (2\pi)^{-3/2} e^{-|v|^2/2}, 
\] 
which enjoys the same mass, momentum and energy as $F_0$.
\smallskip

\noindent$\bullet$ {\bf Perturbation Equation.} In the perturbation framework, let $f$ be the perturbation such that
\[
F=\mu+f.
\]
The Boltzmann equation (\ref{Boltzmann equation}) becomes
\[
\partial_t f +v \cdot \nabla_x f= Q(\mu, f)+Q(f,\mu)+Q(f, f):=Lf +Q(f, f)
\]
with the linearized operator is defined  by $L: =Q(\mu,\cdot)+Q(\cdot,\mu) - v \cdot \nabla_x f$.

\subsection{Brief review of previous results}

In what follows we recall some known results on the Landau and Boltzmann equations with a focus on the topics under consideration in this paper, particularly on global existence and large-time behavior of solutions to the spatially inhomogeneous equations in the perturbation framework. For global solutions to the renormalized equation with large initial data, we mention the classical works  \cite{DL, DL2, L, V2, V3, DV, AV}. We mention \cite{CH, CH2} for the best regularity results available for the Boltzmann equation without cut-off. For the stability of vacuum, see \cite{L2, G5, C2}  for the Landau, cutoff  and non-cutoff Boltzmann equation with moderate soft potential  respectively.

In the near Maxwellian framework, global existence and large-time behavior of solutions to the spatially inhomogeneous equations is proved   in \cite{G2, G3, SG, SG2} for the cutoff Boltzmann equation and  in \cite{G} for the Landau equation.  For the non-cutoff Boltzmann equation it is first proved in  \cite{GS, GS2, AMUXY, AMUXY2, AMUXY3, AMUXY4}, see also \cite{DLSS} for a recent work on such topic. 
We also refer to \cite{G6, G7, G8, DL3, XXZ, DLYZ, DLYZ2} for the former works on the Vlasov-Poisson/Maxwell-Boltzmann/Landau equation near Maxwellian.    We remark here all these works above are base on the following decomposition 
\[
\partial_t f  + v \cdot \nabla_x f =L_{\mu} f + \Gamma(f, f), \quad L_{\mu} f  = \frac {1} {\sqrt{\mu}}  (Q(\mu, \sqrt{\mu } f) + Q(\sqrt{\mu} f, \mu)  ) ,\quad \Gamma(f, f) = \frac 1 { \sqrt{\mu}} Q(\sqrt{\mu} f, \sqrt{\mu}  f), 
\]
which means the result are in $\mu^{-1/2}$ weighted space.

For the inhomogeneous equations  with  polynomial weight perturbation near Maxwellian, in Gualdani-Mischler-Mouhot \cite{GMM} the authors first  prove the global existence and large time behavior of solutions with polynomial velocity weight for the cutoff Boltzmann equation with hard potential.  This  method is generalized to the Landau equation in \cite{CTW, CM} and the cutoff Boltzmann with soft potential in \cite{C}. The non-cutoff Boltzmann equation with hard potential is proved in \cite{HTT, AMSY}, and the soft potential case is proved in \cite{CHJ}.

For the non-cutoff Boltzmann equation,  \cite{IMS, IMS2, IS, IS2, IS3, S}  obtains global regularity and long time behavior by  assuming a uniformly bound in $t, x$ such that
\[
0<m_0 \le M(t, x) \le M_0, \quad E(t, x) \le E_0, \quad  H(t, x) \le H_0,
\]
for some constant $m_0, M_0, E_0, H_0$, where
\[
M(t, x) =\int_{\R^3} f(t, x, v) dv, \quad E(t, x) = \int_{\R^3} f(t, x, v) |v|^2 dv, \quad H(t, x) = \int_{\R^3} f(t, x, v) \ln f(t, x ,v) dv.
\]
For the Landau equation the local H\"older estimate is proved in \cite{GIMV} for  and higher regularity of solutions is studied in \cite{HS} by applying a kinetic version of Schauder estimates.  These papers can be seen as under conditional regularity. More specifically, solutions with properties which remain to be justified in general are shown to have quantitative regularization. In this respect, these works are mainly concerned with the regularization mechanism of kinetic equations while our goal is to establish a self-contained well-posedness result.

In the near Maxwellian framework, former mentioned works mainly focus on the $L^2$ well-posedness, there are also several results for $L^\infty_{x, v}$ well-posedness result near Maxwellian. For the cutoff Boltzmann equation near equilibrium, the $L^2-L^\infty$ approach has been introduced in  \cite{UY, G4} and apply to various contexts, see  \cite{K, GKTT} and the reference therein for example. Also see \cite{DHWY} for the solution with large amplitude initial data under the assumption of small  entropy. For the Landau equation, in \cite{KGH}  the authors proved a global $L^\infty_{x, v}$ well-posedness result near Maxwellian by using strategies inspired by \cite{GIMV}. However, it is not clear to us how to extend the argument in \cite{KGH} to the non-cutoff Boltzmann equation since the Landau operator  is closer to classical nonlinear parabolic operators. For the non-cutoff Boltzmann equation, \cite{AMSY2} proves a $L^\infty_{x, v}$ well-posedness result for the hard potential $\gamma>0$  by using  Di Giorgi iteration developed in \cite{A}.  Another  $L^\infty_{x, v}$ well-posedness result is obtained in \cite{SS} by using the result in \cite{IS, IS2, IS3}  for moderate soft potential $\gamma+2s \ge 0$. In this paper we will use the Di Giorgi iteration method developed in \cite{A, AMSY2} together with polynomial weighted estimate for the non-cutoff Boltzmann equation  developed in \cite{CHJ} to prove the well-posedness for the case $-3 < \gamma \le 0$. 

Recently in \cite{HST}, the authors proved independently the global well-posedness $L^\infty_{x, v}$ near Maxwellian by using the result in \cite{IS, IS2, IS3}, but their proof requires the lower bound assumption on the initial data
\[
f_0(x, v) \ge \delta \quad \hbox{for} \quad x \in B_r(x, v)
\]
for some constant $(x, v)$ and $\delta, r>0$. And in our paper we also obtained the rate of convergence to the equilibrium. 
\subsection{ Main results and notations}  Let us first introduce the function spaces and notations.

  \noindent $\bullet$  For any $p \in [1, +\infty)$, $q \in \R$ the $L^p_{q}$ norm is defined by
\[
\| f \|_{L^p_{q}}^p :  = \int_{\R^3} |f(v)|^p   \langle v \rangle^{pq} dv,
\]
where the Japanese bracket $\langle v \rangle$ is defined as $\langle v \rangle := (1+|v|^2)^{1/2}$.

\noindent $\bullet$  For real numbers $m, l$, we define the weighted Sobolev space $H_l^m$ by
\[
H^m_l:= \{ f(v) |  |f|_{ H^m_l}=|\langle \cdot\rangle^l \langle D\rangle^mf|_{L^2}< +\infty \},
\]
where  $a(D)$ is a pseudo-differential operator with the symbol $a(\xi)$ and it is defined as
\[
(a(D)f)(v):= \frac{1}{(2\pi)^3}  \int_{\R^3}\int_{\R^3}  e^{i(v-u)\xi}  a(\xi)  f(u)  du d\xi.
\]
and we denote $H^m := H^m_0$.

\noindent $\bullet$ For function $f(x, v),x\in\T^3,v\in\R^3$, the norm $\|\cdot\|_{H^\alpha_xH^m_l}$ is defined as
\[
\|f\|_{H^\alpha_xH^m_l}:=\left(\int_{\T^3}\|\langle D_x\rangle^\alpha f(x,\cdot)\|^2_{H^m_l}dx\right)^{1/2}.
\]
If $\alpha=0$, $H^0_xH^m_l=L^2_xH^m_l$.

\noindent $\bullet$ The $L\log L$ space is defined as
\[
L\log L:=\Big\{f(v)|\|f\|_{L\log L}=\int_{\R^3}|f|\log(1+|f|)dv\Big\}.
\]
The $L\log L$ norm is defined by   
\[
\|f\|_{L\log L} := \int_{\R^3} |f| \log (|f|+1)dv.
\]

\noindent $\bullet$ We write $a \lesssim b$ indicate that there is a uniform constant $C$, which may be different on different lines, such that $a\le C b $. We use the notation $a\sim b$ whenever $a\lesssim b$ and $b\lesssim a$.  We denote $C_{a_1,a_2, \cdots, a_n}$ by a constant depending on parameters $a_1,a_2,\cdots, a_n$. Moreover, we use parameter $\epsilon$ to represent different positive numbers much less than 1 and determined in different cases.

\noindent $\bullet$ We use $(f, g)$ to denote the inner product of $f, g$ in the $v$ variable $(f, g)_{L^2_v}$ for short and we use $(f, g)_{L^2_k}$ to denote $(f, g  \langle v \rangle^{2k})$.

\noindent $\bullet$  For any function $f$ we define
\[
\|f(\theta)\|_{L^1_\theta} : =\int_{\S^2}f(\theta)d\sigma = 2\pi\int_0^{\pi}f(\theta)\sin\theta d\theta.
\]

\noindent $\bullet$  For $ p \in [1, \infty)$ and $\beta \in \R$ we define  the Bessel potential space as
\[
H^{\beta, p}(\R^d):=\{f(v) | |\langle  D\rangle^\beta   f|_{L^p(\R^d)}< + \infty\},
\]
and the associated norm is defined by
\[
\Vert  u \Vert_{H^{\beta, p}(\R^d)} : = \Vert \langle D \rangle^\beta u \Vert_{L^p(\R^d)}.
\]

\noindent $\bullet$ For any $ p \in [1.\infty), \beta \in \R$, Sobolev-Slobodeckij space is  defined as
\[
W^{\beta, p}(\R^d) : = \left  \{ f \in L^p(\R^d) | \int_{\R^d}  \int_{\R^d }  \frac {|  f(x)  - f (y)|^p } {|x-y|^{d+\beta p}}  dx dy   \right \},
\]
and the associated norm is defined by
\[
\Vert  u \Vert_{W^{\beta, p}(\R^d)} : = \left(\int_{\R^d} | f (x) |^p dx + \int_{\R^3}  \int_{\R^3 }  \frac {|f (x)  - f (y) |^p } {|x-y|^{d+\beta p}}  dx dy  \right)^{\frac 1 p} .
\]
The Bessel potential space and Sobolev-Slobodeckij space agree for $p = 2$. More generally, the following condition holds: \\
(i)For all $p \in (1, 2]$, $\beta \in (0, 1)$  it holds that 
\[
\Vert  u \Vert_{H^{\beta, p}(\R^d)} \le C  \Vert  u \Vert_{W^{\beta, p}(\R^d)}.
\]
(ii)For all $p \in [2, + \infty)$, $\beta \in (0, 1)$  it holds that 
\[
\Vert  u \Vert_{W^{\beta, p}(\R^d)} \le C  \Vert  u \Vert_{H^{\beta, p}(\R^d)}.
\]
The proof can be found in \cite{S2}, Chapter V. 

\noindent $\bullet$ For the linearized operator $L$ we have
\[
\ker(L) = \text{span} \{ \mu, v_1\mu, v_2 \mu, v_3 \mu, |v|^2 \mu \}.
\]
We define the projection onto $\ker(L)$ by
\begin{equation}
\label{projection}
P f : = \left(  \int_{\T^3}   \int_{\R^3} f dv dx \right) \mu + \sum_{i=1}^3 \left(   \int_{\T^3}  \int_{\R^3} v_i f dv dx \right) v_i \mu + \left(   \int_{\T^3}   \int_{\R^3} \frac {|v|^2-3}{\sqrt{6} } f dv dx \right) \frac  {|v|^2 -3} {\sqrt{6}} \mu.
\end{equation}
For any space $X$, define the  subspace $\Pi X$ by
\begin{equation}
\label{Pi X}
\Pi X : =\{  f \in X|  P f=0\}. 
\end{equation}

\noindent $\bullet$  For any $ k \in  \R, \gamma \in (-3, 1]$, we define
\[
\Vert f \Vert_{L^2_{k+\gamma/2, *}}^2 :=  \int_{\R^3} \int_{\R^3} \mu (v_*) |v-v_*|^\gamma |f(v)|^2 \langle v \rangle^{2k} dv dv_*.
\]
It is easily seen that $\Vert f \Vert_{L^2_{k+\gamma/2, *}} \sim \Vert f \Vert_{L^2_{k+\gamma/2}}$. 

\noindent $\bullet$ For any $l \ge 0, K \ge 0$, define the level set function
\begin{equation}
\label{level set function}
f_{K}^l := f \langle v \rangle^l - K, \quad f_{K, +}^l := f_{K}^l 1_{ \{f_K^l \ge 0\} }, \quad  f_{K, -}^l := f_{K}^{l} 1_{ \{f_K^l < 0\} }.
\end{equation}

\noindent $\bullet$. For some $\alpha\ge 0$, we introduce a regularizing linear operator defined by
\begin{equation}
\label{definition L alpha}
L_\alpha \phi(v) = -(  \langle v \rangle^{2\alpha} \phi -\nabla_v \cdot (\langle v \rangle^{2\alpha} \nabla_v \phi  )) ,\quad \alpha \ge 0.
\end{equation}

\noindent $\bullet$ In the whole paper, we define the cutoff function $\chi \in C^\infty$ satisfies $0 \le  \chi \le 1$, $|\nabla\chi| \le 4/\delta_0$  and 
\begin{align}
\label{cutoff function}
\mathcal \chi(x):=\left\{
\begin{aligned}
  & 1,  \quad |x| \le \delta_0\\
 & 0, \quad   |x| \ge 2\delta_0,
\end{aligned}
\right.
\end{align}
for some small constant $\delta_0 >0$.
\medskip

\subsection{Main results} We may now state our main results. 
\begin{thm}\label{T11}Assume that $\gamma \in (-3, 0], s \in (0, 1), \gamma+2s>-1$, for any smooth function $f$, consider the Cauchy problem
\[
\partial_t f = L f +Q(f, f), \quad \mu+f\ge 0,\quad f(0) =f_0, \quad P  f_0 = 0.
\]
 Suppose the kernel $B$ verifying \big($\mathbf{(A1)}$-$\mathbf{(A4)}$\big). Then for any $k_0 \ge 14$ large, there exist   $k>k_0 \ge 14$ large and  $\epsilon>0$ small such that if
\[
\Vert \langle v \rangle^{k_0} f_0 \Vert_{L^2_{x, v} \cap L^\infty_{x, v}} \le \epsilon, \quad \Vert \langle v \rangle^{k} f_0 \Vert_{L^2_{x, v}} < +\infty,
\]
then there exist a nonnegative solution $F = \mu +f \ge 0$, $F \in L^\infty( [0,\infty), L^2_xL^2_k(\T^3 \times \R^3))$ to the Boltzmann equation \eqref{Boltzmann equation}. Moreover, if $\gamma =0$, we have
\[
\Vert \langle v \rangle^{k_0} f(t) \Vert_{ L^\infty_{x, v}} \le \delta_1, \quad \Vert \langle v \rangle^{k} f(t) \Vert_{L^2_{x, v} } \le Ce^{-\lambda t}, \quad  \Vert \langle v \rangle^{k_0}  f(t) \Vert_{ L^\infty_{x, v}} \le  Ce^{-\lambda_1 t},
\]
for some constant $\delta_1, C, \lambda, \lambda_1>0$. If $\gamma <0$, for any $14 \le k_1 <k$ we have
\[
\Vert \langle v \rangle^{k_0} f(t) \Vert_{ L^\infty_{x, v}} \le \delta_1, \quad \Vert \langle v \rangle^{k_1} f(t) \Vert_{L^2_{x, v} } \le C (1+t)^{- \frac { k -k_1}  {|\gamma|}}, \quad \Vert \langle v \rangle^{k_0}  f(t)  \Vert_{ L^\infty_{x, v}} \le C (1+t)^{-\alpha_1},
\]
for some constant $C, \delta_1, \alpha >0$.
\end{thm}

\begin{rmk}
We only focus on the  Maxwellian molecule and soft potential case $-3<\gamma \le 0$ since the hard potential case is already proved in \cite{AMSY2}.
\end{rmk}

%

\noindent$\bullet$ \underline{\it Comment on the solutions.} 

The spectral gap  for the polynomial weight $\langle v \rangle^k$ is $\gamma \ge 0$ instead of $\gamma+2s \ge 0$. which is  proved in  \cite{CHJ}. In fact it is proved in \cite{CHJ} such that the spectral gap for $m =e^{k\langle  v \rangle^\beta},  k >0, \beta \in (0, 2] $ is $\gamma+\beta s \ge 0$.

 \subsection{Strategies and ideas of the proof} In this subsection, we will explain main strategies and ideas of the proof for our results. First observe the fact that if we assume the smallness of
\[
\sup_{t, x} \Vert f \Vert_{L^\infty_{|\gamma|+9}} \le \delta_0,
\]
for some $\delta_0>0$ small, then we can obtain $L^2_{x} H^s_v$ energy estimates with general weights. Using classical velocity averaging lemma we can transfer the regularity from the velocity to the spatial variable to obtain regularization in the spatial variable $H^{s'}_xL^2_x$ for any $s' \in (0. \frac {s} {2(s+3)})$. This Hypoellipticity allows us to apply the De Giorgi argument through embeddings of Sobolev spaces to $L^p$ spaces. 

Since the time averaging lemma requires $p >1$, when applying the time averaging lemma to $(f_{K, +}^l)^2$, we need the estimate  of $ \Vert f_{K, +}^l \Vert_{L^{2p}_{x, v}}$, we take $p$ very close to 1 such that the $L^{2p}$ norm can be estimated by $H^{s'}_xL^2_x$  and $L^2_{x} H^s_v$ norm using Sobolev embedding.

More precisely, we construct the following energy functional:
\begin{align}
\nonumber
\mathcal{E}_{p, s''} (K, T_1, T_2) :=& \sup_{t \in [T_1, T_2] } \Vert f_{K, +}^l \Vert_{L^2_{x, v}}^2   + \int_{T_1}^{T_2} \int_{\T^3} \Vert \langle  v \rangle^{\gamma/2} f_{K, +}^l \Vert_{H^s_v}^2 dx d\tau 
\\ \nonumber
&+\frac 1 {C_0} \left( \int_{T_1}^{T_2} \Vert (1-\Delta_x)^{s''/2} ( \langle v \rangle^{-2+ \gamma/2} f_{K, +}^l)^2\Vert_{L^p_{x, v}}^p d\tau \right)^{1/p}.
\end{align}
The main step in the De Giorgi argument is to  prove that there exist a constant $K_0 \ge 0$ such that
\[
M_k : = K_0 (1-\frac 1 {2^k}), \quad \mathcal{E}_k := \mathcal{E}_{p, s''}(M_k, 0, T), \quad \mathcal{E}_k \le C \sum_{i=1}^4 \frac {2^{k(a_i +1)}   \mathcal{E}_{k-1}^{\beta_i }  } { K_0^{a_i} }.
\]
Using Di Giorgi iteration we can prove there exist a constant $K_0'$ such that if $K_0 \ge K_0'$, then  $\mathcal{E}_k \to 0$, which implies
\[
\sup_{t \in [0, T] } \Vert f_{K_0, +}^l(t, \cdot, \cdot ) \Vert_{L^2_{x, v}}  = 0. 
\]
similar results is proved for $f_-$. Thus instead of directly proving $\Vert \langle  v \rangle^l f \Vert_{L^\infty_{x, v}} \le \delta $ for some $\delta >0$, we prove
\[
\sup_{t \in [0, T] } \Vert f_{K_0, +}^l(t, \cdot, \cdot ) \Vert_{L^2_{x, v}}  = 0, \quad \sup_{t \in [0, T] } \Vert f_{K_0, -}^l(t, \cdot, \cdot ) \Vert_{L^2_{x, v}}  = 0,
\]
for some constant $K_0 >0$, which could implies 
\[
\sup_{t \in [0, T]} \Vert \langle v \rangle^l f (t, \cdot ,\cdot)\Vert_{L^\infty_{x, v}} \le K_0,
\] 
See also Figure 1 in \cite{AMSY2} for a better understanding for the structure of the proof.

The strategy described above is applied first to the linearized equation and then to the nonlinear equation to obtain local solutions with $L^\infty$-bounds to the original Boltzmann equation. Although obtaining a solution to the linearized equation is fairly straight-forward, significant effort has been carried out to show the $L^\infty$-bounds of the solution.  More precisely, we first prove the local existence of solutions for the linearized equation
\[
\partial_t f +v \cdot \nabla_x f = \epsilon L_\alpha(\mu + f) + Q(\mu + g \chi(\langle v \rangle^{k_0}  g), u+ f),
\]
where $L_\alpha$ is a regularizing term defined in \eqref{definition L alpha} and $\chi$ is the cutoff function  defined in \eqref{cutoff function}. Using fixed point theorem we obtained a local solution for the nonlinear equation
\[
\partial_t f+ v \cdot \nabla_x f = \epsilon L_\alpha(\mu +f) + Q(\mu + f \chi(\langle v \rangle^{k_0} f), \mu + f),
\]
then we prove that the  solution satisfies
\[
\Vert f \Vert_{L^\infty ([0, T] \times \T^3 \times \R^3  )} \le \delta_0,
\]
hence the solution becomes a solution to 
\begin{equation}
\label{Boltzmann equation epsilon}
\partial_t f+ v \cdot \nabla_x f =\epsilon L_\alpha(\mu +f)  + Q(\mu + f , \mu + f).
\end{equation}
Finally we prove that the priori estimates of \eqref{Boltzmann equation epsilon} are independent of $\epsilon$, thus we can pass the limit in $\epsilon$ to prove the local existence of the Boltzmann equation \eqref{Boltzmann equation}. Finally, combining the local existence with the global estimate obtained in \cite{CHJ} we obtain a global solution to the original Boltzmann equation.

 \subsection{Organization of the paper}

After introducing a technical toolbox in Section \ref{section 2}, Section \ref{section 3} is devoted to the upper bounds  and coercivity estimate on collision operator $Q$. In Section \ref{section 4} we establish the estimates for the $L^2$ level set function $f_{K, +}^l$ and we establish the $L^\infty$ bound for the linearized equation. These estimates are used in Section \ref{section 5} to show the existence of solution to the linear equation. In Section \ref{section 6} we establish the nonlinear counterparts of the estimates of those in Section \ref{section 3} and Section \ref{section 4}, then apply them to establish the local well-posedness of the Boltzmann equation. In Section \ref{section 7} we combine the results in Section \ref{section 6} and global estimate of the  Boltzmann operator to establish the global well-posedness of the nonlinear Boltzmann equation. The well-posedness proved in Section \ref{section 7} is only for weak singularity kernels, we extend the result to the strong singularity case in Section \ref{section 8}.

\section{Preliminaries}\label{section 2}

In this section we recall several lemmas which is useful in later proof, some of them may be elementary  but we still write it for completeness. 

\begin{lem} (\cite{ADVW}) For any smooth function $f, g, b $, we have \\
(1) (Regular change of variables)
\[
\int_{\R^3} \int_{\mathbb{S}^2} b(\cos \theta) |v-v_*|^\gamma f(v') d \sigma dv= \int_{\R^3} \int_{\mathbb{S}^2} b(\cos \theta)\frac 1 {\cos^{3+\gamma} (\theta/2)} |v-v_*|^\gamma f(v) d\sigma dv.
\]
(2) (Singular change of variables)
\[
\int_{\R^3} \int_{\mathbb{S}^2} b(\cos \theta) |v-v_*|^\gamma f(v') d \sigma dv_* = \int_{\R^3} \int_{\mathbb{S}^2} b(\cos \theta)\frac 1 {\sin^{3+\gamma} (\theta/2)} |v-v_*|^\gamma f(v_*) d\sigma dv_* .
\]
\end{lem}

\begin{lem}\label{L22} For any smooth function $f, g, h, b$, for any $\gamma \in\R$ we have
\begin{equation*}
\begin{aligned}
&\left(\int_{\R^3}\int_{\R^3} \int_{\mathbb{S}^2}b(\cos \theta) |v-v_*|^\gamma f_* g h' dv dv_* d\sigma  \right)^2
\\
\le& \left( \int_{\mathbb{S}^2} b( \cos \theta) \sin^{-\frac 3 2- \frac  \gamma 2 } \frac \theta 2   d\sigma \right)^2 \int_{\R^3}\int_{\R^3}  |v-v_*|^\gamma |f_*|^2 |g| dv dv_* \int_{\R^3}\int_{\R^3}  |v-v'|^\gamma |g| |h'|^2dv dv'  .
\end{aligned}
\end{equation*}
and 
\begin{equation*}
\begin{aligned}
&\left(\int_{\R^3}\int_{\R^3} \int_{\mathbb{S}^2}b(\cos \theta) |v-v_*|^\gamma f_* g h' dv dv_* d\sigma  \right)^2
\\
\le&\left( \int_{\mathbb{S}^2} b(\cos \theta) \cos^{-\frac 3 2- \frac  \gamma 2} \frac \theta 2   d\sigma  \right)^2  \int_{\R^3}\int_{\R^3}  |v-v_*|^\gamma |f_*| |g|^2 dv dv_*  \int_{\R^3}\int_{\R^3}  |v_*-v'|^\gamma |f_*| |h'|^2dv_* dv'.
\end{aligned}
\end{equation*}
\end{lem}
\begin{proof}
 The proof is  the combination of Cauchy-Schwarz inequality and regular/singular  change of variable.
\end{proof}

Then we recall the upper and lower bound for the Boltzmann operator

\begin{lem}\label{L23}(\cite{H} Theorem 1.1) 
Let $w_1, w_2 \in \R$, $a, b \in[0, 2s]$ with $w_1+w_2 =\gamma+2s$  and $a+b =2s$. Then for any smooth functions $g, h, f$ we have \\
(1) if $\gamma + 2s > 0$ we have
\[
|(Q(g, h), f)_{L^2_v}| \lesssim \Vert(\Vert g \Vert_{L^1_{\gamma+2s +(-w_1)^+(-w_2)^+}}  +\Vert g \Vert_{L^2} ) \Vert h \Vert_{H^a_{w_1}}   \Vert f \Vert_{H^b_{w_2}}  .
\]
(2) if $\gamma + 2s = 0$ we have
\[
|(Q(g, h), f)_{L^2_v}|\lesssim (\Vert g \Vert_{L^1_{w_3}} + \Vert g \Vert_{L^2}) \Vert h \Vert_{H^a_{w_1}}   \Vert f \Vert_{H^b_{w_2}} ,
\]
where $w_3 = \max\{\delta,(-w_1)^+ +(-w_2)^+ \}$, with $\delta>0$ sufficiently small.\\
(3) if $\gamma + 2s < 0$ we have
\[
|(Q(g, h), f)_{L^2_v}|\lesssim  (\Vert g \Vert_{L^1_{w_4}} + \Vert g \Vert_{L^2_{-(\gamma+2s)}}) \Vert h \Vert_{H^a_{w_1}}   \Vert f \Vert_{H^b_{w_2}}  ,
\]
where $w_3 = \max\{-(\gamma+2s), \gamma+2s +(-w_1)^+ +(-w_2)^+ \}$.
\end{lem}

\begin{lem}\label{L24}(\cite{H}, Theorem 1.2)
Suppose that $g$ is a non-negative and smooth function verifying that
\[
\Vert g \Vert_{L^1} \ge \delta, \quad \Vert g \Vert_{L^1_2} + \Vert g \Vert_{L \log L} < \lambda,
\]
Then there exist $C_1(\lambda, \delta)$ and $C_2(\lambda, \delta)$ such that\\
(1) If $\gamma +2s \ge 0$, then we have
\[
( -Q(g, f) , f)_{L^2_v} \ge C_1(\lambda, \delta)\Vert f \Vert_{H^s_{\gamma/2}}^2 - C_2(\lambda, \delta) \Vert f \Vert_{L^2_{\gamma/2}}^2.
\]
(2) If $-1-2s < \gamma < -2s$
\[
( -Q(g, f) , f)_{L^2_v}  \ge C_1(\lambda, \delta)\Vert f \Vert_{H^s_{\gamma/2}}^2 - C_2(\lambda, \delta)(1+\Vert g \Vert_{L^p_{|\gamma|}}^{\frac {(\gamma+2s+3)p } {(\gamma+2s+3)p -3 } }) \Vert f \Vert_{L^2_{\gamma/2}}^2,
\]
with $p>\frac {3} {3+\gamma+2s}$.
\end{lem}

\begin{lem} (Cancellation Lemma) (\cite{ADVW}, Lemma 1) For any smooth function $f$ we have
\[
\int_{\R^3} \int_{\mathbb{S}^{2}} B(v-v_*, \sigma) (f'-f) dv d\sigma= (f*S )(v_*),
\]
where
\[
S(z) = |\mathbb{S}^1| \int_{0}^{\frac \pi 2} \sin \theta \left[ \frac 1 {\cos^3(\theta/2)} B\left(\frac {|z|}  {\cos (\theta/2)} , \cos \theta \right) - B(|z|, cos \theta) \right].
\]
\end{lem}

\begin{lem}(\cite{V}, Section 1.4)(Pre-post collisional change of variable) For any function $F$ smooth enough we have
\[
\int_{\R^3}\int_{\R^3} \int_{\mathbb{S}^2 }F(v, v_*, v', v_*') B(|v-v_*|, \cos \theta) dv dv_* d \sigma =\int_{\R^3}\int_{\R^3} \int_{\mathbb{S}^2 }F(v', v_*', v, v_*) B(|v-v_*|, \cos \theta) dv dv_* d \sigma.
\]
\end{lem}

\begin{lem} (Hardy-Littlewood-Sobolev inequality) (\cite{LL}, Chapter 4)
Let $p, r>1$ and $0<\lambda<N$ with $1/p+\lambda/n+1/r = 2$. Let $f \in L^p(\R^n)$ and $h \in L^r(\R^N)$. Then there exist a constant $C(n,\lambda, p)$, independent of $f$ and $h$, such that
\[
\int_{\R^N} \int_{\R^N} f(x) |x-y|^{-\lambda} h(y) dx dy \le C(n,\lambda, p) \Vert f \Vert_{L^p} \Vert h \Vert_{L^r}.
\]
\end{lem}

In this paper we will use the following representation of $v'$ which can be proved directly. We have
\[
\langle v' \rangle^2 = \langle v \rangle^2 \cos^2 \frac \theta 2 + \langle v_*\rangle^2 \sin^2 \frac \theta 2+ 2 \cos\frac \theta 2 \sin \frac \theta 2 |v-v_*| v \cdot \omega,   \quad \omega \perp (v-v_*), \quad v \cdot w = v_* \cdot w,
\]
where $\omega = \frac {\sigma - (\sigma \cdot k)k } {|\sigma - (\sigma \cdot k)k |}$ with $k = \frac {v-v_*} {|v-v_*|}$. We can further decompose $\omega$ by
\begin{equation}
\label{decomposition omega}
\omega = \tilde{\omega} \cos \frac \theta 2  + \frac {v'-v_*} {|v'-v_*|} \sin \frac \theta 2,\quad \tilde{\omega} = \frac {v'-v} {|v'-v|}, \quad \tilde{\omega} \perp (v'-v_*). 
\end{equation}
For the $\langle v '\rangle^k$ we have

\begin{lem}\label{L28} (\cite{C}, Lemma 2.7)
For any constant $k \ge 4$ we have
\[
\langle v' \rangle^{k} -\langle v \rangle^{k} \cos^{k} \frac \theta 2 =  k  \langle v \rangle^{k-2} \cos^{k-1} \frac \theta 2 \sin \frac \theta 2 |v-v_*| (v_* \cdot \omega) + L_1 +L_2,
\]
with 
\[
|L_1 |  \le C_k\sin^{k-2} \frac \theta 2 \langle  v_* \rangle^{k} \langle v \rangle^2 ,\quad | L_2 | \le C_k \langle v \rangle^{k-2}  \langle v_* \rangle^4 \sin^2 \frac \theta 2,
\]
for some constant $C_k>0$. 
\end{lem}

By symmetry we have

\begin{lem}\label{L29}
 For smooth function $f, g$ and for any constant $k$, we have
\[
\Gamma:=\int_{\R^3}\int_{\R^3} \int_{\mathbb{S}^2} b(\cos \theta) |v-v_*|^{1+\gamma} (v_*\cdot \tilde{\omega}) \cos^k \frac \theta 2 \sin \frac \theta 2 f_* g' dv dv_* d\sigma =0.
\] 
\end{lem}
\begin{proof}
By the regular change of variable $v \to v'$ and take the new $v'-v_*$ as the north pole, recall $\tilde{\omega} \perp (v'-v_*)$. Then we have 
\[
 \Gamma=\int_{\R^3}\int_{\R^3} \int_{\mathbb{S}^2} b(\cos \theta) \frac 1 {\cos^{4+\gamma} \frac \theta 2} |v'-v_*|^{1+\gamma} (v_*\cdot \tilde{\omega}) \cos^k \frac \theta 2 \sin \frac \theta 2 f_* g' dv' dv_* \sin \theta d\theta d\phi,
\]
where $\tilde{\omega} = (\cos \phi, \sin \phi, 0)$. It's easily seen that the integration in $\phi$ gives that $\Gamma =0$.
\end{proof}

We introduce some $L^p$ inequalities related to the singularity of the kernel.

\begin{lem}\label{L210}(\cite{CHJ}, Lemma 2.6)
Suppose $\gamma \in (-3, 1], s \in (0, 1), \gamma+2s >-1$. For any smooth function $g$ and $f$ we have
\[
\mathcal{R}  := \int_{\R^3} \int_{\R^3} |v-v_*|^\gamma g_* f^2 dv_* dv\lesssim \Vert g \Vert_{L^2_{|\gamma|+2 }} \Vert f \Vert_{H^s_{ \gamma/2}}^2 .
\]
\end{lem}

\begin{lem}\label{L211} (\cite{C}, Lemma 2.5) Suppose $\gamma \in(-3, 0)$.  For any smooth function $f$  we have
\[
\sup_{v \in \R^3}\int_{\R^3} |v-v_*|^{\gamma} |f| (v_*) dv_* \lesssim \Vert f \Vert_{L^1}^{1 + \frac \gamma 3} \Vert f \Vert_{L^\infty}^{ - \frac \gamma 3} \lesssim \Vert f \Vert_{L^\infty_4}.
\]
\end{lem}

\begin{lem}\label{L212}
If $-3 <\gamma \le 1$, then for any smooth function $g$ and $f$, we have
\[
\mathcal{R}  := \int_{\R^3} \int_{\R^3} |v-v_*|^\gamma g_* |f|^2 dv_* dv \lesssim \Vert g \Vert_{L^\infty_{|\gamma|+4 }} \Vert f \Vert_{L^2_{ \gamma/2}}^2. 
\]
\end{lem}
\begin{proof}
The case $\gamma \ge 0 $ is obvious so we focus on the case $\gamma<0$. Since $ \langle v \rangle^{|\gamma|} \lesssim  \langle v_* \rangle^{|\gamma|}  \langle v-v_* \rangle^{|\gamma|}$, together with Lemma \ref{L211} we have
\begin{equation*}
\begin{aligned} 
\mathcal{R} \lesssim& \int_{\R^3} \int_{\R^3} \frac {\langle v-v_*\rangle^{|\gamma|}} {|v-v_*|^{|\gamma|}} g_* \langle v_* \rangle^{|\gamma|}  |f|^2\langle v \rangle^{\gamma} dv_* dv
\\
\lesssim&\int_{\R^3} \int_{\R^3}  (1+|v-v_*|^\gamma) (g_* \langle v_* \rangle^{|\gamma|})  (f\langle v \rangle^{\gamma/2})^2 dv_* dv
\\
\lesssim& (\Vert g\Vert_{L^1_{|\gamma|}} + \Vert g \Vert_{L^\infty_{|\gamma|}}) \Vert f\Vert_{L^2_{\gamma/2}}\lesssim \Vert g \Vert_{L^\infty_{|\gamma|+4 }} \Vert f \Vert_{L^2_{ \gamma/2}}^2,
\end{aligned}
\end{equation*}
so the lemma is thus proved.
\end{proof}

\begin{lem}\label{L213}
Suppose $\gamma+2s>-1, \gamma >-3, s \in (0, 1)$. For any smooth function $g, f$ and $h$, let
\[
\mathcal{R}  := \int_{\R^3} \int_{\R^3} \int_{\mathbb{S}^2}b(\cos \theta) \sin^{k} \frac \theta 2 |v-v_*|^\gamma g_* f h' dv_* dv d \sigma,
\]
then we have
\begin{equation}
\label{estimate regular}
\mathcal{R} \lesssim  \Vert b (\cos \theta) \sin^k \frac  \theta 2\Vert_{L^1_\theta} \min \{  \Vert g \Vert_{L^2_{|\gamma| +2}}\Vert f \Vert_{H^s_{ \gamma/2}}  \Vert h \Vert_{H^s_{ \gamma/2}} , \Vert g \Vert_{L^\infty_{|\gamma| + 4}}\Vert f \Vert_{L^2_{ \gamma/2}}  \Vert h \Vert_{L^2_{ \gamma/2}}  \},
\end{equation}
and
\begin{equation}
\label{estimate singular}
\mathcal{R} \lesssim \Vert  b (\cos \theta) \sin^{k-\gamma/2-3/2} \frac  \theta 2 \Vert_{L^1_\theta}  \min \{ \Vert f \Vert_{L^2_{|\gamma| +2}}\Vert g \Vert_{H^s_{ \gamma/2}}  \Vert h \Vert_{H^s_{ \gamma/2}} ,  \Vert f  \Vert_{L^\infty_{|\gamma| + 4 }}\Vert g \Vert_{L^2_{ \gamma/2}}  \Vert h \Vert_{L^2_{ \gamma/2}}  \}.
\end{equation}
\end{lem}
\begin{proof}The proof is just the combination of Lemma \ref{L22},  Lemma \ref{L210} and Lemma {L212}.
\end{proof}

\begin{lem}\label{L214}(\cite{AMSY2}, Lemma 2.2)  Suppose $\alpha \in (0, 1)$ and $f \in H_v^\alpha(\R^3)$ smooth, then we have
\[
\Vert (-\Delta_v)^\alpha (  \langle v \rangle^{-2} f ) \Vert_{L^2_v( \R^3) }  \le C \Vert (-\Delta_v)^\alpha f  \Vert_{L^2_v( \R^3) }.
\]
\end{lem}

\begin{lem} (\cite{CCL}, Lemma 2.1)\label{L215} Suppose $H \in W^{2 , \infty} (\R^3)$. Then for any $s \in (0, 1)$, it holds that
\[
\int_{\mathbb{S}^2} (H' -H) b (\cos \theta) d \sigma \le C \left(  \sup_{|u| \le |v_*| + |v|} |\nabla H(u)|  +  \sup_{|u| \le |v_*| + |v|} |\nabla^2 H(u)| \right) |v-v_*|^2.
\]
\end{lem}

\begin{lem}(\cite{AMSY2}, Proposition 2.11)\label{L216} Let $\eta, \eta' \in (0, 1)$, then for some $r = r(\eta, \eta', d ) > 2$ and $\alpha = \alpha(\eta, \eta', d ) \in (0, 1)$, for any smooth function $f$ we have
\[
\Vert f  \Vert_{L^r_{x, v}}  \le C \left( \int_{\T^d}   \Vert (-\Delta_v )^{\frac {\eta} 2} f \Vert_{L^2_v}^2 dx  \right)^\frac {\alpha} 2 \left( \int_{\R^d}   \Vert (1 -\Delta_x )^{\frac {\eta'} 2} f \Vert_{L^2_x}^2 dv \right)^\frac {1-\alpha} 2,
\]
where the constant $C$ is independent of $f$.
\end{lem}

\begin{lem}(\cite{AMSY2}, Proposition 2.12)\label{L217} Let $\eta, \eta' \in (0, 1), m \ge 1$, then for some $r = \tilde{r}(\eta, \eta', m,  d ) >2$ and $\alpha = \tilde{\alpha}(\eta, \eta', m,  d ) \in (0, 1)$, for any smooth function $f$ we have
\[
\Vert f  \Vert_{L^r_{x, v}}  \le C \left( \int_{\T^d}   \Vert (-\Delta_v )^{\frac {\eta} 2} f \Vert_{L^2_v}^2 dx  \right)^\frac {\alpha} 2 \left( \int_{\R^d}   \Vert (1 -\Delta_x )^{\frac {\eta'} 2} f^2 \Vert_{L^m_x}^2 dv  \right)^\frac {1-\alpha} 2,
\]
where the constant $C$ is independent of $f$ and $\alpha$ is continuous with $m$.
\end{lem}

\begin{lem}(\cite{AMSY2}, Proposition 2.13)\label{L218} Let $p \in (1, 2)$, $0 \le \beta' < \beta \in (0, 1)$, $p' = \frac p {2-p}$, then for any smooth function $f$ we have
\[
\Vert (-\Delta)^{\frac {\beta'} 2} f^2 \Vert_{L^p(\R^d)} \le C (\Vert f \Vert_{H^\beta(\R^d) }  \Vert f^2 \Vert_{L^{p'}(\R^d)}^{\frac 1 2 } + \Vert f^2 \Vert_{L^p(\R^d)} ),
\]
where the constant C is independent of $f$. 
\end{lem}

We then introduce the time averaging lemma. 

\begin{lem}(\cite{AMSY2}, Proposition 2.14)\label{L219} Fix $0 \le T_1 \le T_2, p \in (1, \infty), \beta > 0$ and assume that $f \in C([T_1, T_2], L^p_{x, v})$ with $\Delta_v^{\beta/2} \in L^{p}_{t, x, v}$ satisfies 
\[
\partial_t f + v \cdot \nabla_x f =  F, \quad  t \in (0, +\infty),
\]
Then for any $r \in [0, \frac 1 p], m \in \N, \beta_- \in [0, \beta)$ if we define
\[
s^b =\frac {(1 - r p)\beta_-} {p (1+ m +\beta)}, \quad \tilde{f} = f 1_{(T_1, T_2) }(t) , \quad \tilde{F} = F 1_{(T_1, T_2) }(t),
\] 
then it follows that
\begin{equation*}
\begin{aligned} 
\Vert (-\Delta_x)^{\frac {s_b} 2} \tilde{f} \Vert_{L^p_{t, x, v}} \le& C (\Vert \langle v \rangle^{1+m}  (1-\Delta_x )^{-\frac r 2}  (1-\Delta_v)^{-\frac m 2} f (T_1) \Vert_{L^p_{x, v}} + \Vert \langle v \rangle^{1+m}  (1-\Delta_x)^{-\frac r 2}  (1-\Delta_v)^{-\frac m 2} f (T_2) \Vert_{L^p_{x, v}} 
\\
&+ \Vert \langle v \rangle^{1+m}  (1-\Delta_x- \partial_t^2 )^{-\frac r 2}  (1-\Delta_v  )^{-\frac m 2} \tilde{F}  \Vert_{L^p_{t, x, v}}  + \Vert  (-\Delta_v)^{\frac \beta 2}  \tilde{f} \Vert_{L^p_{t, x, v}}  +  \Vert   \tilde{f} \Vert_{L^p_{t, x, v}} ,
\end{aligned}
\end{equation*}
where the constant $C$ is independent of $f$ and $F$. 
\end{lem}

\begin{lem}\label{L220}(\cite{AMSY2}, Lemma 3.5)
For any smooth function $f, g$, denote $G= \mu +g, F= \mu +f$, suppose $f$ satisfies the linearized Boltzmann equation
\[
\partial_t f + v \cdot \nabla_x f = 2 \tilde{Q}(G, F).
\]
Then for any $j, l \ge 0, \tau>0, K>0, 0 \le T_1 <T_2 \le T$, it follows that
\begin{equation*}
\begin{aligned} 
&\int_{T_1}^{T_2} \int_{\T^3} \left| \langle v \rangle^j (1-\Delta_v)^{-\tau/2}  \tilde{Q}(G, F )  \langle v \rangle^l f_{K, +}^l( \cdot, \cdot, v)               \right| dx dt
\\
\le& \frac 1 2 \int_{\T^3} \langle v \rangle^j  (1-\Delta_v)^{-\tau/2}  (f_{K, +}^l)^2(T_1, \cdot, v )+2 \int_{T_1}^{T_2} \int_{\T^3} \left[\langle v \rangle^j (1-\Delta_v)^{-\tau/2}  \tilde{Q}(G, F )  \langle v \rangle^l f_{K, +}^l( \cdot, \cdot, v)               \right]^+ dx dt.
\end{aligned}
\end{equation*}
where $[\cdot]^+$ denotes the positive part of the term and $f_{K, +}^l$ is defined in \eqref{level set function}.
\end{lem}
Similarly we have
\begin{lem}(\cite{AMSY2} Lemma 3.6)
For any smooth function $ g, h$, denote $G= \mu +g$, suppose $h$ satisfies the linearized Boltzmann equation
\[
\partial_t h + v \cdot \nabla_x h = \tilde{Q} (G, -\mu + h ),
\]
then for any $j, l \ge 0, \tau>0, K>0, 0 \le T_1 <T_2 \le T$, it follows that
\begin{equation*}
\begin{aligned} 
&\int_{T_1}^{T_2} \int_{\T^3} \left| \langle v \rangle^j (1-\Delta_v)^{-\tau/2}  \tilde{Q}(G, -\mu +h )  \langle v \rangle^l h_{K, +}^l( \cdot, \cdot, v)               \right| dx dt
\\
\le& \frac 1 2 \int_{\T^3} \langle v \rangle^j  (1-\Delta_v)^{-\tau/2}  (h_{K, +}^l)^2(T_1, \cdot, v )+2 \int_{T_1}^{T_2} \int_{\T^3} \left[\langle v \rangle^j (1-\Delta_v)^{-\tau/2}  \tilde{Q}(G, -\mu+h )  \langle v \rangle^l h_{K, +}^l( \cdot, \cdot, v)               \right]^+ dx dt,
\end{aligned}
\end{equation*}
where $[\cdot]^+$ denotes the positive part of the term.
\end{lem}

%
%
%

\section{Estimates for the collision operator}\label{section 3}

In this section we focus on the estimates for the collisional operator $Q$. We first recall 

\begin{lem}\label{L31}(\cite{CHJ}, Lemma 3.1)
Suppose $\gamma>-3,  s \in(0, 1)$, for any $l \ge 8$ large, $h, g$ smooth, we have
\begin{equation*}
\begin{aligned} 
|(Q (h, \mu), g \langle v \rangle^{2l}) | \le &  \Vert b(\cos \theta) \sin^{l-\frac {3+\gamma} 2} \frac \theta 2 \Vert_{L^1_\theta}\Vert h \Vert_{L^2_{l +\gamma/2, *}}\Vert g \Vert_{L^2_{l + \gamma/2, *}} + C_l \Vert h \Vert_{L^2_{l + \gamma/2-1/2}}\Vert g \Vert_{L^2_{ l + \gamma/2 - 1/2}}
\\ 
\le & \Vert b(\cos \theta) \sin^{l - 2} \frac \theta 2 \Vert_{L^1_\theta}\Vert h \Vert_{L^2_{l + \gamma/2, *}}\Vert g \Vert_{L^2_{l + \gamma/2, *}} + C_l \Vert h \Vert_{L^2_{l + \gamma/2-1/2}}\Vert g \Vert_{L^2_{l + \gamma/2-1/2}},
\end{aligned}
\end{equation*}
for some constant $C_l > 0$.
\end{lem}

Then we introduce several upper bounds for the weighted commutator which is very important in the whole paper. Various upper bounds handle the moments required in various estimates.

\begin{lem}\label{L32}
Suppose $\gamma \in (-3, 0], s\in (0,1)$, $\gamma+2s>-1, l \ge 8$, for any smooth function $g, f, h$ we have
\begin{align} 
\label{estimate Lambda 1}
\nonumber
\Lambda =& \int_{\R^3}\int_{\R^3}\int_{\mathbb{S}^2}  | v-v_*|^\gamma b(\cos \theta) g(v_*) f(v) h'(v) \langle v' \rangle^l (\langle v' \rangle^l - \cos^l \frac \theta 2 \langle v \rangle^l ) dv dv_* d \sigma 
\\ \nonumber
\lesssim &  \min \{ \Vert f \Vert_{L^2_{|\gamma| +7 }} \Vert g \Vert_{H^s_{ l +\gamma/2 }}\Vert h \Vert_{H^s_{ l + \gamma/2}}, \Vert g\Vert_{L^\infty_{ l + 5 +|\gamma|}}  \Vert f \Vert_{L^2_{ l + \gamma/2  -1 }}\Vert h \Vert_{L^2_{ l + \gamma/2-1 }} \}
\\
&+ \min \{ \Vert g \Vert_{L^2_{|\gamma| +7 }} \Vert f \Vert_{H^s_{ l + \gamma/2}}\Vert h \Vert_{H^s_{ l + \gamma/2}} ,   \Vert g \Vert_{L^\infty_{|\gamma| +9}} \Vert f \Vert_{H^s_{ l+ \gamma/2  }}  \Vert h \Vert_{L^2_{ l + \gamma/2  -s'} }  \},
\end{align}
with $s' = \min \{\frac 1 2 , 1 - s \} >0$.  We also have another estimate of $\Lambda$
\begin{equation}
\label{estimate Lambda 2}
|\Lambda| \lesssim  \Vert g \Vert_{L^\infty_{|\gamma| + 9  }} \Vert f \Vert_{H^s_{ l + \gamma/2 }}\Vert h \Vert_{L^2_{ l + \gamma/2}}  +\Vert g \Vert_{L^\infty_l }  \Vert f \Vert_{L^2_4} \Vert h \Vert_{L^2_{l + 2 }}. 
\end{equation}
Similarly we have
\begin{align} 
\label{estimate Gamma 1}
\nonumber
\Gamma : =& \int_{\R^3}\int_{\R^3}\int_{\mathbb{S}^2}   | v-v_*|^\gamma  b(\cos \theta) g(v_*) f(v) h'(v) \langle v' \rangle^l (\langle v' \rangle^l - \langle v \rangle^l ) dv dv_* d \sigma 
\\
\lesssim &  \min \{ \Vert f \Vert_{L^2_{|\gamma| +7 }} \Vert g \Vert_{H^s_{ l + \gamma/2 }}\Vert h \Vert_{H^s_{ l  + \gamma/2}}, \Vert g\Vert_{L^\infty_{ l + 5 +|\gamma|}}  \Vert f \Vert_{L^2_{ l + \gamma/2 }}\Vert h \Vert_{L^2_{ l + \gamma/2  }} \}+ \Vert g \Vert_{L^2_{|\gamma| +7 }} \Vert f \Vert_{H^s_{ l + \gamma/2}}\Vert h \Vert_{H^s_{ l + \gamma/2}},
\end{align}
we also have two other estimates of $\Gamma$
\begin{equation}
\label{estimate Gamma 2}
|\Gamma|\lesssim  \Vert g \Vert_{L^\infty_{|\gamma| + 9  }} \Vert f \Vert_{H^s_{ l + \gamma/2 }}\Vert h \Vert_{L^2_{ l + \gamma/2}}  +\Vert g \Vert_{L^\infty_l   }  \Vert f \Vert_{L^2_4} \Vert h \Vert_{L^2_{l + 2}} ,
\end{equation}
and
\begin{equation}
\label{estimate Gamma 3}
|\Gamma|\lesssim  \Vert g\Vert_{L^\infty_l } ( \Vert  f \Vert_{L^\infty_l }  +  \Vert \nabla( f\langle v \rangle^{l - 2} )  \Vert_{L^\infty} )  \Vert  h \Vert_{L^1_{ l +2 } } .
\end{equation}
\end{lem}
\begin{proof}
By \eqref{estimate regular}
\begin{equation*}
\begin{aligned} 
|\Gamma - \Lambda| := |\Lambda_4| = &\int_{\R^3}\int_{\R^3}\int_{\mathbb{S}^2} | v-v_*|^\gamma   b(\cos \theta) g(v_*) f(v) h'(v) \langle v' \rangle^l  \langle v \rangle^l ( 1-\cos^l  \frac \theta 2) dv dv_* d \sigma 
\\
 \lesssim&   \int_{\R^3} \int_{\R^3} \int_{\mathbb{S}^2} b(\cos \theta) \sin^{2} \frac \theta 2   |v-v_*|^\gamma \langle v_* \rangle^{4}  g_* |f| \langle v \rangle^{l }   |h'|  \langle v' \rangle^l  dvdv_* d\sigma
\\
\lesssim & \min \{ \Vert g \Vert_{L^2_{|\gamma| +7}} \Vert f \Vert_{H^s_{ l + \gamma/2  }} \Vert h \Vert_{H^s_{ l +  \gamma/2 }}, \Vert g\Vert_{L^\infty_{9+|\gamma|}}  \Vert f \Vert_{L^2_{ l  + \gamma/2}}\Vert h \Vert_{L^2_{ l + \gamma/2}} \}.
\end{aligned}
\end{equation*}
By regular change of variable and Lemma \ref{L211} we have
\begin{equation*}
\begin{aligned} 
 |\Lambda_4|  \lesssim &\int_{\R^3} \int_{\R^3} \int_{\mathbb{S}^2} b(\cos \theta) \sin^{2} \frac \theta 2   |v-v_*|^\gamma \langle v_* \rangle^{4}  g_* |f| \langle v \rangle^{l}   |h'|  \langle v' \rangle^l dvdv_* d\sigma
\\
\lesssim & \Vert f \Vert_{L^\infty_l} \int_{\R^3} \int_{\R^3} \int_{\mathbb{S}^2} b(\cos \theta) \sin^{2} \frac \theta 2   |v-v_*|^\gamma \langle v_* \rangle^{4}  g_*    |h'|  \langle v' \rangle^l dvdv_* d\sigma
\\
\lesssim &\Vert f \Vert_{L^\infty_l}  \int_{\R^3} \int_{\R^3}    |v-v_*|^\gamma \langle v_* \rangle^{4}  g_*    |h'|  \langle v' \rangle^l dvdv_*\lesssim \Vert g \Vert_{L^\infty_8} \Vert f \Vert_{L^\infty_l} \Vert h \Vert_{L^1_l} ,
\end{aligned}
\end{equation*}
Gathering the two estimates we have  
\[
|\Lambda_4 | \lesssim  \min \{ \Vert g \Vert_{L^2_{|\gamma| +7}} \Vert f \Vert_{H^s_{ l + \gamma/2  }} \Vert h \Vert_{H^s_{ l + \gamma/2  }}, \Vert g\Vert_{L^\infty_{9+|\gamma|}}  \Vert f \Vert_{L^2_{ l + \gamma/2}}\Vert h \Vert_{L^2_{ l + \gamma/2}}, \Vert g \Vert_{L^\infty_8} \Vert f \Vert_{L^\infty_l} \Vert h \Vert_{L^1_l}  \},
\]
We focus on the $\Lambda$ term. By Lemma \ref{L28} we have
\begin{equation*}
\begin{aligned} 
\Lambda=& \int_{\R^3} \int_{\R^3} \int_{\mathbb{S}^2} b(\cos \theta) |v-v_*|^\gamma (l \langle v \rangle^{l - 2} |v-v_*| (v_*\cdot \omega) \cos^{l-1} \frac \theta 2 \sin \frac \theta 2)   g_* f  h'   \langle v' \rangle^l dvdv_* d\sigma
\\
& + \int_{\R^3} \int_{\R^3} \int_{\mathbb{S}^2} b(\cos \theta) |v-v_*|^\gamma  g_* f   h' \langle v' \rangle^l \sum_{i=2}^3 L_i dvdv_* d\sigma := \sum_{i=1}^3 \Lambda_i,
\end{aligned}
\end{equation*}
with 
\[
|L_1 |\le C_l\sin^{l - 2} \frac \theta 2 \langle  v_* \rangle^{l} \langle v \rangle^2 ,\quad |L_2| \le C_l \langle v \rangle^{l - 2}  \langle v_* \rangle^4 \sin^2 \frac \theta 2 .
\]
We first estimate the $\Lambda_2$ term, for the $\Lambda_2$ term, we have several different estimates. Since $l \ge 5$, by \eqref{estimate singular} we have
\begin{equation*}
\begin{aligned} 
|\Lambda_2| \lesssim   \int_{\R^3} \int_{\R^3} \int_{\mathbb{S}^2} b(\cos \theta) |v-v_*|^\gamma\sin^{ l  - 2} \frac \theta 2  |g_*| \langle v_* \rangle^{l} |f| \langle v \rangle^2   |h'|  \langle v' \rangle^l  dvdv_* d\sigma
\lesssim    \Vert f \Vert_{L^2_{|\gamma| + 4}} \Vert g \Vert_{H^s_{ l + \gamma/2 }}   \Vert h \Vert_{H^s_{ l  + \gamma/2}},
\end{aligned}
\end{equation*}
Using \eqref{estimate regular} we have
\begin{equation*}
\begin{aligned} 
 |\Lambda_2| \lesssim&   \int_{\R^3} \int_{\R^3} \int_{\mathbb{S}^2} b(\cos \theta) |v-v_*|^\gamma\sin^{l - 2 } \frac \theta 2  |g_*| \langle v_* \rangle^{l}  |f| \langle v \rangle^2   |h'|  \langle v' \rangle^l dvdv_* d\sigma
\\
\lesssim&   \int_{\R^3} \int_{\R^3} \int_{\mathbb{S}^2} b(\cos \theta) |v-v_*|^\gamma\sin^{l-2} \frac \theta 2  |g_*| \langle v_* \rangle^{l + 1} |f| \langle v \rangle^3  |h'|  \langle v' \rangle^{l - 1} dvdv_* d\sigma
\\
\lesssim &   \Vert g \Vert_{L^\infty_{ l  + |\gamma| + 5 }} \Vert f \Vert_{L^2_{ l + \gamma/2-1 }}   \Vert h \Vert_{L^2_{ l  + \gamma/2-1}}.
\end{aligned}
\end{equation*}
By singular change of variable we have
\begin{equation*}
\begin{aligned} 
 |\Lambda_2| \lesssim&   \int_{\R^3} \int_{\R^3} \int_{\mathbb{S}^2} b(\cos \theta) |v-v_*|^\gamma\sin^{l - 2} \frac \theta 2  |g_*| \langle v_* \rangle^{ l } |f| \langle v \rangle^2   |h'|  \langle v' \rangle^l dvdv_* d\sigma
\\
\lesssim &   \Vert g \Vert_{L^\infty_l}    \int_{\R^3} \int_{\R^3} \int_{\mathbb{S}^2} b(\cos \theta) |v-v_*|^\gamma\sin^{l - 2-3-\gamma} \frac \theta 2  |f| \langle v \rangle^2   |h_*|  \langle v_* \rangle^l dvdv_* d\sigma
\\
\lesssim &\Vert g \Vert_{L^\infty_l  } \int_{\R^3} \int_{\R^3} |v-v_*|^\gamma  |f| \langle v \rangle^2 |h_*|  \langle  v_*\rangle^l    dv dv_*  ,
\end{aligned}
\end{equation*}
by Hardy-Littlewood-Sobolev inequality we have
\[
|\Lambda_2|   \lesssim \Vert g \Vert_{L^\infty_l }  \Vert f \Vert_{L^p_2} \Vert h \Vert_{L^p_l }     \lesssim  \Vert g \Vert_{L^\infty_l  }  \Vert f \Vert_{L^2_4} \Vert h \Vert_{L^2_{ l + 2 }} ,
\]
where $p=1$ if $\gamma=0$ and  $\frac 1 p =\frac 1 2(2 + \frac \gamma 3) \in (\frac 1 2, 1)$ if $-3< \gamma <0$.  By Lemma \ref{L211} we have
\[
|\Lambda_2|  \lesssim \Vert g \Vert_{L^\infty_l  } \Vert f  \Vert_{L^\infty_6} \Vert h \Vert_{L^1_l  },
\]
Gathering the terms we have 
\begin{equation*}
\begin{aligned} 
|\Lambda_2| \lesssim & \min \{ \Vert f \Vert_{L^2_{|\gamma| + 4 }} \Vert g \Vert_{H^s_{ l +\gamma/2 }}   \Vert h \Vert_{H^s_{ l + \gamma/2}} ,    \Vert g \Vert_{L^\infty_{l + |\gamma| +5}} \Vert f \Vert_{L^2_{ l +\gamma/2-1 }}   \Vert h \Vert_{L^2_{ l + \gamma/2-1}},  \Vert g \Vert_{L^\infty_l}  \Vert f \Vert_{L^2_4} \Vert h \Vert_{L^2_{l+2}} , \Vert g \Vert_{L^\infty_l} \Vert f  \Vert_{L^\infty_6} \Vert h \Vert_{L^1_l}   \}.
\end{aligned}
\end{equation*}
For the $\Lambda_3$  term, by Lemma \ref{L213}  we have
\begin{equation*}
\begin{aligned} 
 |\Lambda_3| \lesssim&   \int_{\R^3} \int_{\R^3} \int_{\mathbb{S}^2} b(\cos \theta) \sin^{2} \frac \theta 2   |v-v_*|^\gamma \langle v_* \rangle^{4}  g_* |f| \langle v \rangle^{l - 2}   |h'|  \langle v' \rangle^l  dvdv_* d\sigma
\\
\lesssim&   \int_{\R^3} \int_{\R^3} \int_{\mathbb{S}^2} b(\cos \theta) \sin^{2} \frac \theta 2   |v-v_*|^\gamma \langle v_* \rangle^{5}  g_* |f| \langle v \rangle^{l - 1}   |h'|  \langle v' \rangle^{l - 1} dvdv_* d\sigma
\\
\lesssim & \min \{   \Vert g \Vert_{L^2_{|\gamma| +7}} \Vert f \Vert_{H^s_{ l  + \gamma/2  }}  \Vert h \Vert_{H^s_{ l + \gamma/2}} , \Vert g \Vert_{L^\infty_{|\gamma| +9}} \Vert f \Vert_{L^2_{ l + \gamma/2 -1 }}  \Vert h \Vert_{L^2_{ l + \gamma/2 -1 }}  \}.
\end{aligned}
\end{equation*}
By regular change of variable and Lemma \ref{L211} we have
\begin{equation*}
\begin{aligned} 
|\Lambda_3|  \lesssim& \int_{\R^3} \int_{\R^3} \int_{\mathbb{S}^2} b(\cos \theta) \sin^2 \frac \theta 2 |v-v_*|^\gamma   |g_*| \langle  v_* \rangle^{4}  |f| \langle v \rangle^{l} |h'| \langle v' \rangle^l   dvdv_* d\sigma
\\
\lesssim&\Vert f \Vert_{L^\infty_l} \int_{\R^3} \int_{\R^3} \int_{\mathbb{S}^2} b(\cos \theta) \sin^2 \frac \theta 2 |v-v_*|^\gamma   |g_*| \langle  v_* \rangle^{4} |h'|  \langle v' \rangle^l  dvdv_* d\sigma
\\
\lesssim& \Vert f \Vert_{L^\infty_l}\int_{\R^3} \int_{\R^3}  |v-v_*|^\gamma   |g_*| \langle v_* \rangle^4  |h'|    \langle v' \rangle^l    dv dv_* 
\lesssim \Vert g \Vert_{L^\infty_8} \Vert f \Vert_{L^\infty_l} \Vert h\Vert_{L^1_l}.
\end{aligned}
\end{equation*}
Gathering the two estimates we have
\begin{equation*}
\begin{aligned} 
|\Lambda_3| \lesssim \min \{  \Vert g \Vert_{L^2_{|\gamma| +7}} \Vert f \Vert_{H^s_{ l + \gamma/2 }}  \Vert h \Vert_{H^s_{ l + \gamma/2}} ,   \Vert g \Vert_{L^\infty_{|\gamma| +9}} \Vert f \Vert_{L^2_{ l + \gamma/2  -1 }}  \Vert h \Vert_{L^2_{ l + \gamma/2 -1 }} ,   \Vert g \Vert_{L^\infty_8} \Vert f \Vert_{L^\infty_l} \Vert h\Vert_{L^1_l} \},
\end{aligned}
\end{equation*}
For the $\Lambda_1$ term,  by \eqref{decomposition omega} we have
\begin{equation*}
\begin{aligned} 
\Lambda_1=& \int_{\R^3} \int_{\R^3} \int_{\mathbb{S}^2} b(\cos \theta) |v-v_*|^\gamma (l \langle v \rangle^{l-2} |v-v_*| (v_*\cdot \tilde{\omega}) \cos^{l} \frac \theta 2 \sin \frac \theta 2) g_* f  h'\langle v' \rangle^l dvdv_* d\sigma
\\
& +\int_{\R^3} \int_{\R^3} \int_{\mathbb{S}^2} b(\cos \theta) |v-v_*|^\gamma (l \langle v \rangle^{l - 2} |v-v_*| (v_*\cdot \frac {v'-v_*} {|v'-v_*|}) \cos^{ l - 1} \frac \theta 2 \sin^2 \frac \theta 2) g_*  f h' \langle v' \rangle^l dvdv_* d\sigma
\\
:= &\Lambda_{1, 1} +\Lambda_{1, 2}.
\end{aligned}
\end{equation*}
For the $\Lambda_{1, 2}$ term, using \eqref{estimate regular}  we have
\begin{equation*}
\begin{aligned} 
 |\Lambda_{1, 2}| \lesssim&   \int_{\R^3} \int_{\R^3} \int_{\mathbb{S}^2} b(\cos \theta) \sin^{2} \frac \theta 2 |v-v_*|^{1+\gamma} \langle v_* \rangle  |g_*| |f| \langle v \rangle^{l -2}   |h'|  \langle v' \rangle^l  dvdv_* d\sigma
\\
\lesssim&   \int_{\R^3} \int_{\R^3} \int_{\mathbb{S}^2} b(\cos \theta) \sin^{2} \frac \theta 2 |v-v_*|^{\gamma} \langle v_* \rangle^2  |g_*| |f| \langle v \rangle^{ l -1}   |h' |  \langle v' \rangle^{l  -1} dvdv_* d\sigma
\\
\lesssim &   \min \{ \Vert g \Vert_{L^2_{|\gamma| +7}} \Vert f \Vert_{H^s_{ l + \gamma/2  }}  \Vert h \Vert_{H^s_{ l  + \gamma/2}} ,   \Vert g \Vert_{L^\infty_{|\gamma| +7}} \Vert f \Vert_{L^2_{ l  + \gamma/2  -1 }}  \Vert h \Vert_{L^2_{ l  + \gamma/2 -1 }}  \}.
\end{aligned}
\end{equation*}
By regular change of variable and Lemma \ref{L211} we have
\begin{equation*}
\begin{aligned} 
 |\Lambda_{1, 2}| \lesssim&   \int_{\R^3} \int_{\R^3} \int_{\mathbb{S}^2} b(\cos \theta) \sin^{2} \frac \theta 2 |v-v_*|^{1+\gamma} \langle v_* \rangle  |g_*| |f| \langle v \rangle^{l - 2}   |h'|  \langle v' \rangle^{l }  dvdv_* d\sigma
\\
\lesssim&  \Vert f \Vert_{L^\infty_l} \int_{\R^3} \int_{\R^3} \int_{\mathbb{S}^2} b(\cos \theta) \sin^{2} \frac \theta 2 |v-v_*|^{\gamma} \langle v_* \rangle^3  |g_*|  |h'| \langle v' \rangle^{l } dvdv_* d\sigma
\\
\lesssim &   \Vert g \Vert_{L^\infty_{|\gamma| +7}} \Vert f \Vert_{L^\infty_{l }}  \Vert h \Vert_{L^1_l }.
\end{aligned}
\end{equation*}
Gathering the two estimates we have
\[
|\Lambda_{1, 2}| \lesssim  \min \{   \Vert g \Vert_{L^2_{|\gamma| +7}} \Vert f \Vert_{H^s_{ l + \gamma/2  }}  \Vert h \Vert_{H^s_{ l  + \gamma/2}} ,    \Vert g \Vert_{L^\infty_{|\gamma| +7}} \Vert f \Vert_{L^2_{ l + \gamma/2  -1 }}  \Vert h \Vert_{L^2_{ l + \gamma/2 -1 }} ,     \Vert g \Vert_{L^\infty_{|\gamma| +7}} \Vert f \Vert_{L^\infty_{l }}  \Vert h \Vert_{L^1_l}      \}.
\]
For the $\Lambda_{1, 1}$ term,  by Lemma \ref{L29} we have
\begin{equation}
\label{Lambda 11}
\Lambda_{1, 1} =  l \int_{\R^3} \int_{\R^3} \int_{\mathbb{S}^2} b(\cos \theta) |v-v_*|^{1+ \gamma}  (v_*\cdot \tilde{\omega}) \cos^{l} \frac \theta 2 \sin \frac \theta 2 g_* (f \langle v \rangle^{l - 2} -f'  \langle v' \rangle^{l - 2})  h'\langle v' \rangle^l  dvdv_* d\sigma. 
\end{equation}
Using 
\[
f \langle v \rangle^{l-2} -f' \langle v' \rangle^{l-2}= \frac 1 {\langle v \rangle^{2-s} }( f \langle v \rangle^{l-s} - f' \langle v' \rangle^{l-s})+ f' \langle v' \rangle^{l - s} \left ( \frac 1 {\langle v \rangle^{2-s}}- \frac 1 {\langle v' \rangle^{2-s}  } \right ),
\]
we split $\Lambda_{1, 1}$ into two parts
\begin{equation*}
\begin{aligned}
\Lambda_{1, 1} =& l \int_{\R^3} \int_{\R^3} \int_{\mathbb{S}^2} b(\cos \theta) |v-v_*|^{1+\gamma} (v_* \cdot \tilde{\omega}) \cos^l \frac \theta 2 \sin \frac \theta 2 g_* h' \langle v' \rangle^{l} \frac 1 {\langle v \rangle^{2-s} }(f \langle v \rangle^{l - s} -f' \langle v' \rangle^{ l - s })dv dv_* d\sigma
\\
&+ l \int_{\R^3} \int_{\R^3} \int_{\mathbb{S}^2} b(\cos \theta) |v-v_*|^{1+\gamma} (v_* \cdot \tilde{\omega}) \cos^l \frac \theta 2 \sin \frac \theta 2 g_*  f' h' \langle v' \rangle^{2l-s } \left (\frac 1 {\langle v \rangle^{2-s}}- \frac 1 {\langle v' \rangle^{2 - s} } \right) dv dv_* d\sigma
\\
:=& \Lambda_{1,1,1} + \Lambda_{1,1,2}.
\end{aligned}
\end{equation*}
For the  $\Lambda_{1,1,2}$ term, by
\[
\left| \frac 1 {\langle v \rangle^{2-s}}- \frac 1 {\langle v' \rangle^{2-s}} \right|= \frac {|\langle v' \rangle^{2-s} - \langle v \rangle^{2-s}|} {\langle v' \rangle^{2 - s} \langle v\rangle^{2-s}} \lesssim  \frac {|v'-v| (\langle v \rangle^{1-s} +\langle v' \rangle^{1-s} )} {\langle v' \rangle^{2-s} \langle v\rangle^{2-s}} \lesssim  \frac {|v'-v| \langle v \rangle^{1-s} \langle v_* \rangle^{1-s} } {\langle v' \rangle^{2-s} \langle v\rangle^{2-s}} \lesssim \frac {|v-v_*| \sin \frac \theta 2 \langle v_* \rangle } {\langle v' \rangle^{2-s} \langle v\rangle} ,
\]
together with regular change of variable and  \eqref{estimate regular}  we have
\begin{equation*}
\begin{aligned}
|\Lambda_{1,1,2} |\lesssim &\int_{\R^3} \int_{\R^3} \int_{\mathbb{S}^2} b(\cos \theta) |v-v_*|^{2+\gamma}  \cos^l \frac \theta 2 \sin^2 \frac \theta 2 \langle v_* \rangle^2 |g_*| \frac 1 {\langle v \rangle} |f'| |h'| \langle v' \rangle^{2l-2}  dv dv_* d\sigma
\\
\lesssim &\int_{\R^3} \int_{\R^3} \int_{\mathbb{S}^2} b(\cos \theta) |v-v_*|^{2+\gamma}  \cos^l \frac \theta 2 \sin^2 \frac \theta 2 \langle v_* \rangle^3 |g_*|  |f'| |h'| \langle v' \rangle^{2l-3}  dv dv_* d\sigma
\\
\lesssim &\int_{\R^3} \int_{\R^3} \int_{\mathbb{S}^2} b(\cos \theta) |v-v_*|^{2+\gamma}  \cos^{l-3-\gamma} \frac \theta 2 \sin^2 \frac \theta 2 \langle v_* \rangle^3 |g_*|  |f| |h| \langle v \rangle^{2l-3}  dv dv_* d\sigma
\\
\lesssim &\int_{\R^3} \int_{\R^3} \int_{\mathbb{S}^2} b(\cos \theta) |v-v_*|^{\gamma}  \cos^{l-3-\gamma} \frac \theta 2 \sin^2 \frac \theta 2 \langle v_* \rangle^5 |g_*|  |f| |h| \langle v \rangle^{2l-1}  dv dv_* d\sigma
\\
\lesssim &  \min \{ \Vert g \Vert_{L^\infty_{|\gamma| +9}} \Vert f \Vert_{L^2_{  l + \gamma/2 -1/2  }}  \Vert h \Vert_{L^2_{ l + \gamma/2  -1/2 }}  , \Vert g \Vert_{L^2_{|\gamma| +7}} \Vert f \Vert_{H^s_{ l + \gamma/2 -1/2 }}  \Vert h \Vert_{H^s_{ l + \gamma/2 -1/2 }}   \},
\end{aligned}
\end{equation*}
For the $\Lambda_{1, 1, 1}$ term, by Cauchy-Schwarz inequality we have
\begin{equation*}
\begin{aligned}
|\Lambda_{1,1,1}| \le & \left( \int_{\R^3} \int_{\R^3} \int_{\mathbb{S}^2} b(\cos \theta) |v-v_*|^{\gamma}  \langle v_*\rangle^2  |g_*| (f \langle v \rangle^{l-s} -f' \langle v' \rangle^{l -s})^2  dv dv_* d\sigma  \right)^{1/2}
\\
&\left( \int_{\R^3} \int_{\R^3} \int_{\mathbb{S}^2} b(\cos \theta) \frac{ |v-v_*|^{\gamma+2} } {\langle v \rangle^{4-2s}} \cos^{2l} \frac \theta 2\sin^2 \frac \theta 2 |g_*| |h'|^2\langle v' \rangle^{2l}   dv dv_* d\sigma  \right)^{1/2}.
\end{aligned}
\end{equation*}
By
\[
\frac{ |v-v_*|^2 } {\langle v \rangle^{4-2s}}  \lesssim \frac{ \langle v \rangle^2+ \langle v_* \rangle^2} {\langle v \rangle^{4-2s}}  \lesssim \frac { \langle v_* \rangle^2 } {\langle v \rangle^{2-2s}}  ,
\]
together with  regular change of variable and Lemma \ref{L212}  we have
\begin{equation*}
\begin{aligned}
&\int_{\R^3} \int_{\R^3} \int_{\mathbb{S}^2} b(\cos \theta) \frac{ |v-v_*|^{\gamma+2} } {\langle v \rangle^{4}} \cos^{2l} \frac \theta 2\sin^2 \frac \theta 2 |g_*| |h'|^2\langle v' \rangle^{2l}   dv dv_* d\sigma
\\
\lesssim &\int_{\R^3} \int_{\R^3} \int_{\mathbb{S}^2} b(\cos \theta) \cos^{2l} \frac \theta 2\sin^2 \frac \theta 2     |v-v_*|^{\gamma} \langle v_*\rangle^2  |g_*| \frac {1} {\langle v \rangle^{2-2s} }  |h'|^2\langle v' \rangle^{2l}   dv dv_* d\sigma
\\
\lesssim &\int_{\R^3} \int_{\R^3} \int_{\mathbb{S}^2} b(\cos \theta) \cos^{2l} \frac \theta 2\sin^2 \frac \theta 2     |v-v_*|^{\gamma} \langle v_*\rangle^4  |g_*|  |h'|^2\langle v' \rangle^{2l - 2 + 2s }   dv dv_* d\sigma
\\
\lesssim&\int_{\mathbb{S}^2} b(\cos \theta) \cos^{2l  -3-\gamma} \frac \theta 2\sin^2 \frac \theta 2 d\sigma \int_{\R^3} \int_{\R^3}      |v'-v_*|^{\gamma} \langle v_*\rangle^4  |g_*| |h'|^2\langle v' \rangle^{2l - 2 + 2s}   dv' dv_* 
\\
\lesssim &  \min \{ \Vert g \Vert_{L^2_{|\gamma| +7}} \Vert h \Vert_{H^s_{l+\gamma/2 }}^2  , \Vert g \Vert_{L^\infty_{|\gamma| +8}} \Vert h \Vert_{L^2_{l +\gamma/2  -1 +s }}^2   \}.
\end{aligned}
\end{equation*}
Since $(a-b)^2 =-2a(b-a)+ (b^2 -a^2)$, we have
\begin{equation*}
\begin{aligned}
 &\int_{\R^3} \int_{\R^3} \int_{\mathbb{S}^2} b(\cos \theta) |v-v_*|^{\gamma}  |g_*| \langle v_* \rangle^2  (f \langle v \rangle^{l-s} -f' \langle v' \rangle^{l-s})^2  dv dv_* d\sigma
\\
=&-2\int_{\R^3} \int_{\R^3} \int_{\mathbb{S}^2} b(\cos \theta) |v-v_*|^{\gamma}   |g_*|  \langle v_* \rangle^2 f \langle v \rangle^{l-s} ( f' \langle v' \rangle^{l - s} - f \langle v \rangle^{l - s} )  dv dv_* d\sigma
\\
&+\int_{\R^3} \int_{\R^3} \int_{\mathbb{S}^2} b(\cos \theta) |v-v_*|^{\gamma}  |g_*| \langle v_* \rangle^2  (|f'|^2 \langle v' \rangle^{2l - 2s}  - |f|^2 \langle v \rangle^{2l - 2s} )  dv dv_* d\sigma
\\
=&-2(Q(|g|\langle \cdot \rangle^2, f \langle \cdot \rangle^{l -s}) , f \langle \cdot \rangle^{l -s} )+\int_{\R^3} \int_{\R^3} \int_{\mathbb{S}^2} b(\cos \theta) |v-v_*|^{\gamma}  |g_*| \langle v_* \rangle^2(|f'|^2 \langle v' \rangle^{2l - 2s}  - |f|^2 \langle v \rangle^{2l - 2s} )  dv dv_* d\sigma.
\end{aligned}
\end{equation*}
By cancellation lemma and Lemma \ref{L212}  we have
\begin{equation*}
\begin{aligned}
&\int_{\R^3} \int_{\R^3} \int_{\mathbb{S}^2} b(\cos \theta) |v-v_*|^{\gamma}  |g_*|  \langle v_* \rangle^2  (|f'|^2 \langle v' \rangle^{2l - 2s}  - |f|^2 \langle v \rangle^{2l - 2s} )  dv dv_* d\sigma
\\
\lesssim&  \int_{\R^3} \int_{\R^3}  |v-v_*|^{\gamma}  |g_*| \langle v_* \rangle^2  |f|^2 \langle v \rangle^{2l - 2s}   dv dv_* \lesssim   \min \{ \Vert g \Vert_{L^2_{|\gamma| +7}} \Vert f \Vert_{H^s_{l + \gamma/2-s}}^2  , \Vert g \Vert_{L^\infty_{|\gamma| +8}} \Vert h \Vert_{L^2_{l + \gamma/2-s}}^2   \},
\end{aligned}
\end{equation*}
and by Lemma \ref{L23} we have
\[
|(Q(|g|  \langle \cdot \rangle^2 , f \langle \cdot \rangle^{l - s}) , f \langle \cdot \rangle^{l - s} )| \lesssim( \Vert g \Vert_{L^1_{|\gamma| +5}}+ \Vert g \Vert_{L^2_{|\gamma| +5 }})  \Vert f \Vert_{H^s_{ l+  \gamma/2 }}^2\lesssim \Vert g \Vert_{L^2_{|\gamma| +7 }} \Vert f \Vert_{H^s_{l + \gamma/2}}^2  \lesssim \Vert g \Vert_{L^\infty_{|\gamma| +9 }} \Vert f \Vert_{H^s_{l + \gamma/2}}^2 .
\]
For the term $\Lambda_{1, 1}$ we have another estimate, recall \eqref{Lambda 11}, by mean value theorem 
\[
|f \langle v \rangle^{l - 2} -f' \langle v' \rangle^{l - 2}| \le \Vert \nabla( f\langle v \rangle^{l - 2} ) \Vert_{L^\infty} |v-v'| \le  \Vert \nabla( f\langle v \rangle^{l -2} ) \Vert_{L^\infty} |v-v_*| \sin \frac \theta 2,
\]
together with regular change of variable and Lemma \ref{L211} we have
\begin{equation*}
\begin{aligned} 
\Lambda_{1, 1} \lesssim&  \Vert \nabla( f\langle v \rangle^{l - 2} ) \Vert_{L^\infty}\int_{\R^3} \int_{\R^3} \int_{\mathbb{S}^2} b(\cos \theta) |v-v_*|^{2+\gamma} \cos^l \frac \theta 2 \sin^2 \frac \theta 2 g_* \langle  v_* \rangle  h' \langle v' \rangle^{l}  dv dv_* d\sigma
\\
\lesssim&  \Vert \nabla( f\langle v \rangle^{l-2} ) \Vert_{L^\infty}\int_{\R^3} \int_{\R^3} \int_{\mathbb{S}^2} b(\cos \theta) |v-v_*|^{2+\gamma} \cos^{l-3-\gamma} \frac \theta 2 \sin^2 \frac \theta 2 g_* \langle  v_* \rangle  h \langle v \rangle^{l}  dv dv_* d\sigma
\\
\lesssim&  \Vert \nabla( f\langle v \rangle^{l-2} ) \Vert_{L^\infty}\int_{\R^3} \int_{\R^3}  |v-v_*|^{\gamma} g_* \langle  v_* \rangle^3  h\langle v \rangle^{l + 2} dv dv_* d\sigma \lesssim \Vert g \Vert_{L^\infty_7} \Vert \nabla( f\langle v \rangle^{l - 2} )  \Vert_{L^\infty} \Vert h \Vert_{L^1_{l + 2}},
\end{aligned}
\end{equation*}
Gathering the estimates for  $\Lambda_{1, 1, 1}$ and $\Lambda_{1, 1, 2 }$ terms we have
\begin{equation*}
\begin{aligned} 
|\Lambda_{1, 1}| \lesssim & \min \{   \Vert g \Vert_{L^2_{|\gamma| +7}} \Vert f \Vert_{H^s_{ l+ \gamma/2  }}  \Vert h \Vert_{H^s_{ l + \gamma/2}} ,    \Vert g \Vert_{L^\infty_{|\gamma| +9}} \Vert f \Vert_{H^s_{ l + \gamma/2  }}  \Vert h \Vert_{L^2_{ l + \gamma/2  -s'}} ,    \Vert g \Vert_{L^\infty_7} \Vert \nabla( f\langle v \rangle^{l -2} )  \Vert_{L^\infty} \Vert h \Vert_{L^1_{l  + 2}}  \},
\end{aligned}
\end{equation*}
with $s' = \min\{ \frac 1 2 , 1-s \} > 0$. The lemma is proved by gathering all the terms together.
\end{proof}

\begin{cor}\label{C33}
Suppose $\gamma \in (-3, 0],  s\in (0, 1), \gamma +2s>-1$. For any smooth function $f, g$, for any small constant $\epsilon >0$ we have
\begin{equation*}
\begin{aligned} 
\Lambda =& \int_{\R^3}\int_{\R^3}\int_{\mathbb{S}^2} | v-v_*|^\gamma  b(\cos \theta)  \mu (v_*) g(v) f'(v) \langle v' \rangle^l (\langle v' \rangle^l  - \cos^l \frac \theta 2 \langle v \rangle^l ) dv dv_* d \sigma 
\\
\le &  C_l \Vert \mu \Vert_{L^\infty_{l + |\gamma| + 5}} \Vert g \Vert_{H^s_{ l +  \gamma/2  }}  \Vert f\Vert_{L^2_{ l  + \gamma/2  -s'} }  \le C_l \Vert g \Vert_{H^s_{ l +  \gamma/2  }}  \Vert f\Vert_{L^2_{ l  + \gamma/2  -s'} }  \le \epsilon \Vert g \Vert_{H^s_{ l +  \gamma/2  }}^2 +C_{l, \epsilon}  \Vert f\Vert_{L^2_{ l  + \gamma/2  -s'} }^2,
\end{aligned}
\end{equation*}
for some constant $C_{l, \epsilon}>0$, where $s' = \min \{ \frac 1 2, 1-s\}$.
\end{cor}

\begin{thm}\label{T34}
Suppose $\gamma+2s>-1, s \in (0, 1), \gamma \in (-3, 0]$. For any smooth function $f, g$, suppose $G = \mu + g$ satisfies 
\[
G \ge 0,\quad  \Vert G \Vert_{L^1} \ge A, \quad \Vert G \Vert_{L^1_2} +\Vert G \Vert_{L \log L} \le B,
\]
for some constant $A, B>0$. For any $l  \ge 10$ we have
\begin{equation*}
\begin{aligned}
(Q(G, f), f \langle v \rangle^{2l} )\le& - \gamma_2 \Vert f \Vert_{H^s_{l + \gamma/2}}^2  - \frac {1} {4} \Vert  b(\cos \theta) \sin^2 \frac \theta 2  \Vert_{L^1_\theta}\Vert f \Vert_{L^2_{l + \gamma/2, *}}^2 + C_l  \Vert f \Vert_{L^2_{l + \gamma/2-s' }}^2
\\
&+C_l \min \{ \Vert f \Vert_{L^2_{|\gamma| +7 }} \Vert g \Vert_{H^s_{ l + \gamma/2 }}\Vert f \Vert_{H^s_{ l  + \gamma/2}}, \Vert g\Vert_{L^\infty_{l + 5+|\gamma|}}  \Vert f \Vert_{L^2_{ l  + \gamma/2}}^2 \}  +C_l \Vert g \Vert_{L^2_{|\gamma| +7 }} \Vert f \Vert_{H^s_{ l  + \gamma/2}}^2,
\end{aligned}
\end{equation*}
for some constant $\gamma_2, C_l > 0$, where $s' = \min \{\frac 1 2 , 1 - s \} >0$. 
\end{thm}

\begin{proof} By Cauchy-Schwarz inequality 
\[
|f| |f'| \langle v \rangle^{l} \langle v' \rangle^l \cos^l \frac \theta 2 -|f|^2 \langle v \rangle^{2l} \le \frac 1 2 (|f'|^2\langle v' \rangle^{2l} \cos^{2l} \frac \theta 2-|f|^2 \langle v \rangle^{2l} ),
\]
so by cancellation lemma we have
\begin{equation*}
\begin{aligned}
\int_{\R^3} Q(G, f) f \langle v \rangle^{2l} dv =& \int_{\R^3} \int_{\R^3} \int_{\mathbb{S}^2} b(\cos \theta) |v-v_*|^r  G_* (ff' \langle v '\rangle^{2l}  -|f|^2 \langle v \rangle^{2l}    ) dv dv_* d\sigma
\\
\le & \int_{\R^3} \int_{\R^3} \int_{\mathbb{S}^2} b(\cos \theta) |v-v_*|^r  G_* (|f| |f'| \langle v '\rangle^{l} \langle v '\rangle^{l} -|f|^2 \langle v \rangle^{2l}    ) dv dv_* d\sigma
\\
\le& \int_{\R^3} \int_{\R^3} \int_{\mathbb{S}^2} b(\cos \theta) |v-v_*|^r  G_* (|f| |f'| \langle v' \rangle^l \langle v \rangle^l \cos^l \frac \theta 2 -|f|^2 \langle v \rangle^{2l }    ) dv dv_* d\sigma 
\\
&+ \int_{\R^3} \int_{\R^3} \int_{\mathbb{S}^2} b(\cos \theta) |v-v_*|^r  G_* |f||f'| \langle v' \rangle^l (\langle v' \rangle^{l} -   \langle v \rangle^l \cos^l \frac \theta 2 ) dv dv_* d\sigma 
\\
\le&\frac 1 2\int_{\R^3} \int_{\R^3} \int_{\mathbb{S}^2} b(\cos \theta) |v-v_*|^r  G_* |f|^2  \langle v \rangle^{2l } ( \cos^{2 l - 3 - \gamma} \frac \theta 2 -1    ) dv dv_* d\sigma 
\\
&+ \int_{\R^3} \int_{\R^3} \int_{\mathbb{S}^2} b(\cos \theta) |v-v_*|^r  G_* |f||f'| \langle v' \rangle^l (\langle v' \rangle^{l} -   \langle v \rangle^l \cos^l \frac \theta 2 ) dv dv_* d\sigma  :=T_1+T_2.
\end{aligned}
\end{equation*}
First we talk on $T_1$ term, by $G= \mu +g \ge 0$, by Lemma \ref{L212} we have
\[
T_1 \le -\gamma_0 \Vert f \Vert_{L^2_{l + \gamma/2, *}}^2 +C_l \Vert g \Vert_{L^2_{|\gamma| +7 }} \Vert f \Vert_{H^s_{ l  + \gamma/2}}^2,
\]
for some constant $\gamma_0, C_l>0$, where $\gamma_0$ is defined by
\begin{equation}
\label{gamma 0}
 \gamma_0  : = \frac 1 2 \int_{\R^3} b(\cos \theta) \sin^2 \frac \theta 2    d \sigma  \le \frac  {1} {2} \int_{\R^3} b(\cos \theta) ( 1-\cos^{2l-3-\gamma} \frac \theta 2   )  d \sigma .
\end{equation}
For the $T_2$ term, by \eqref{estimate Lambda 1} and Corollary \ref{C33} we have
\[
T_{2} \le \epsilon \Vert f \Vert_{H^s_{l + \gamma/2}}^2 +  C_{l ,\epsilon}  \Vert f \Vert_{L^2_{l + \gamma/2-s'}}^2 +C_l  \Vert g \Vert_{L^2_{|\gamma| +7 }} \Vert f \Vert_{H^s_{ l  + \gamma/2}}^2 + C_l  \min \{ \Vert f \Vert_{L^2_{|\gamma| +7 }} \Vert g \Vert_{H^s_{ l + \gamma/2 }}\Vert f \Vert_{H^s_{ l + \gamma/2}}, \Vert g\Vert_{L^\infty_{ l + 5 +  |\gamma|  }}  \Vert f \Vert_{L^2_{ l  + \gamma/2}}^2 \}.
\]
which implies
\begin{equation*}
\begin{aligned}
T_1 + T_{2}\le& -\gamma_0   \Vert f \Vert_{L^2_{l + \gamma/2, *}}^2 + \epsilon \Vert f \Vert_{H^s_{l +\gamma/2}}^2+C_{l,\epsilon}  \Vert f \Vert_{L^2_{l +\gamma/2-s'}}^2+C_l \Vert g \Vert_{L^2_{|\gamma| +7 }} \Vert f \Vert_{H^s_{ l  + \gamma/2}}^2
\\
&+ C_l  \min  \{ \Vert f \Vert_{L^2_{|\gamma| +7 }} \Vert g \Vert_{H^s_{ l + \gamma/2 }}\Vert f \Vert_{H^s_{ l  + \gamma/2}}, \Vert g\Vert_{L^\infty_{ l  + 5+|\gamma| }}  \Vert f \Vert_{L^2_{ l + \gamma/2}}^2 \} .
\end{aligned}
\end{equation*}
We introduce another decomposition
\begin{equation*}
\begin{aligned}
\int_{\R^3} Q(G, f) f \langle v \rangle^{2l } dv =&\int_{\R^3} Q(G, f \langle v \rangle^l ) f \langle v \rangle^{l} dv +\int_{\R^3}  \left( \langle v \rangle^l    Q(G, f ) -  Q(G, f \langle v \rangle^l) \right) f \langle v \rangle^l dv :=T_3+T_4.
\end{aligned}
\end{equation*}
For the $T_3$ term, by Lemma \ref{L24} we have
\[
T_3 \le -\gamma_1 \Vert f \Vert_{H^s_{l +\gamma/2}}^2 +C_l   \Vert f \Vert_{L^2_{l +\gamma/2}}^2 .
\]
for some constant $\gamma_1 >0$. For $T_4$ term, we split it into three parts
\begin{equation*}
\begin{aligned} 
T_4 =& \int_{\R^3} \int_{\R^3} \int_{\mathbb{S}^2} b(\cos \theta) |v-v_*|^r  G_* f f' \langle v' \rangle^l  (\langle v' \rangle^{l} -   \langle v \rangle^l   ) dv dv_* d\sigma 
\\
=& \int_{\R^3} \int_{\R^3} \int_{\mathbb{S}^2} b(\cos \theta) |v-v_*|^r  \mu_* f f' \langle v' \rangle^l  (\langle v' \rangle^{l } -   \langle v \rangle^l \cos^l \frac \theta 2 ) dv dv_* d\sigma 
\\
&+\int_{\R^3} \int_{\R^3} \int_{\mathbb{S}^2} b(\cos \theta) |v-v_*|^r  \mu_* f \langle v \rangle^l f' \langle v' \rangle^l  (\cos^l \frac \theta 2 -1) dv dv_* d\sigma
\\
&+ \int_{\R^3} \int_{\R^3} \int_{\mathbb{S}^2} b(\cos \theta) |v-v_*|^r  g_* f f' \langle v' \rangle^l (\langle v' \rangle^{l} -   \langle v \rangle^l ) dv dv_* d\sigma  :=T_{4,1} +T_{4,2} +T_{4,3} .
\end{aligned}
\end{equation*}
For the $T_{4, 1}, T_{4, 2}$ term, by Corollary \ref{C33} and Lemma \ref{L210}  we have
\[
|T_{4, 1}| \le \epsilon \Vert f \Vert_{H^s_{l  +   \gamma/2}}^2+C_{l  , \epsilon}  \Vert f \Vert_{L^2_{l + \gamma/2-s'}}^2, \quad |T_{4,2}| \lesssim  \Vert  f\Vert_{L^2_{l + \gamma/2}}^2.
\]
For $T_{4, 3}$ term by \eqref{estimate Gamma 1} we have
\[
|T_{4, 3}| \le  C_l  \min \{ \Vert f \Vert_{L^2_{|\gamma| +7 }} \Vert g \Vert_{H^s_{ l + \gamma/2 }}\Vert f \Vert_{H^s_{ l + \gamma/2}}, \Vert g\Vert_{L^\infty_{ l + 5+|\gamma|}}  \Vert f \Vert_{L^2_{ l  + \gamma/2}}^2 \}   +C_l \Vert g \Vert_{L^2_{|\gamma| +7 }} \Vert f \Vert_{H^s_{ l + \gamma/2}}^2.
\]
which implies 
\[
T_{4} \le \epsilon \Vert f \Vert_{H^s_{l  + \gamma/2}}^2+C_{l , \epsilon}  \Vert f \Vert_{L^2_{l + \gamma/2}}^2 +  C_l \Vert g \Vert_{L^2_{|\gamma| +7 }} \Vert f \Vert_{H^s_{ l + \gamma/2}}^2 + C_l  \min \{ \Vert f \Vert_{L^2_{|\gamma| +7 }} \Vert g \Vert_{H^s_{ l + \gamma/2 }}\Vert f \Vert_{H^s_{ l  + \gamma/2}}, \Vert g\Vert_{L^\infty_{ l  + 5 + |\gamma|}}  \Vert f \Vert_{L^2_{ l  + \gamma/2}}^2 \} .
\]
Taking $\epsilon =\frac {\gamma_1} 2$, we have
\begin{equation*}
\begin{aligned}
T_{3}+ T_{4} \le& -\frac {\gamma_1} 2 \Vert f \Vert_{H^s_{l + \gamma/2}}^2+C_l  \Vert f \Vert_{L^2_{l + \gamma/2}}^2   +C_l  \Vert g \Vert_{L^2_{|\gamma| +7 }} \Vert f \Vert_{H^s_{ l  + \gamma/2}}^2
\\
&+ C_l  \min \{ \Vert f \Vert_{L^2_{|\gamma| +7 }} \Vert g \Vert_{H^s_{ l + \gamma/2 }}\Vert f \Vert_{H^s_{ l  + \gamma/2}}, \Vert g\Vert_{L^\infty_{ l + 5 + |\gamma|  }}  \Vert f \Vert_{L^2_{ l  + \gamma/2}}^2 \} .
\end{aligned}
\end{equation*}
Gathering the two ways of expansion linearly, take $\delta_2>0$ small we have
\begin{equation*}
\begin{aligned}
\int_{\R^3} Q(G, f) f \langle v \rangle^{2l} dv 
\le& \frac{1} {1+\delta_2} (T_1 +T_2) + \frac{\delta_2} {1+\delta_2} (T_3 +T_4)
\\
\le& \frac{1} {1+\delta_2} (\epsilon- \delta_2\frac {\gamma_1} 2) \Vert f \Vert_{H^s_{l +\gamma/2}}^2 +  \frac{1} {1+\delta_2} (-\gamma_0+ C_l \delta_2) \Vert f \Vert_{L^2_{l +\gamma/2, *}}^2 + C_{l ,\epsilon}  \Vert f \Vert_{L^2_{l + \gamma/2-s'}}^2
\\
&+ C_l  \min \{ \Vert f \Vert_{L^2_{|\gamma| +7 }} \Vert g \Vert_{H^s_{ l+\gamma/2 }}\Vert f \Vert_{H^s_{ l  + \gamma/2}}, \Vert g\Vert_{L^\infty_{ l + 5+|\gamma|}}  \Vert f \Vert_{L^2_{ l + \gamma/2}}^2 \}  +C_l \Vert g \Vert_{L^2_{|\gamma| +7 }} \Vert f \Vert_{H^s_{ l + \gamma/2}}^2
\\
\le& - \delta_2\frac {\gamma_1} {8} \Vert f \Vert_{H^s_{l +\gamma/2}}^2  - \frac {\gamma_0} {2}\Vert f \Vert_{L^2_{l +\gamma/2, *}}^2 + C_{l }  \Vert f \Vert_{L^2_{l+\gamma/2 - s'  }}^2
\\
&+ C_l   \min \{ \Vert f \Vert_{L^2_{|\gamma| +7 }} \Vert g \Vert_{H^s_{ l + \gamma/2 }}\Vert f \Vert_{H^s_{ l  + \gamma/2}}, \Vert g\Vert_{L^\infty_{l + 5+|\gamma|}}  \Vert f \Vert_{L^2_{ l + \gamma/2}}^2 \}    +  C_l \Vert g \Vert_{L^2_{|\gamma| +7 }} \Vert f \Vert_{H^s_{ l + \gamma/2}}^2,
\end{aligned}
\end{equation*}
by taking $\delta_2 = \min\{ \frac {\gamma_0} {4 C_l}, \frac 1 2\}$, $\epsilon = \frac {\gamma_1 \delta_2} {4}$. The proof is thus finished. 
\end{proof}

\begin{rmk}\label{R35}
For any $0 \le \theta \le \frac \pi 2$, $l > 10$,  we have
\[
 \frac 1 4 \sin^2\frac \theta 2 -  \sin^{l-2 } \frac \theta 2 \ge \sin^2 \frac \theta 2\left( \frac 1 4 -\frac {1} {   2^{\frac {l-2} {2} } } \right)>0,
\]
\end{rmk}
Combine Lemma \ref{L31} and Theorem \ref{T34}, by Remark \ref{R35} we have 
\begin{thm}\label{T36} Suppose $\gamma+2s>-1, s \in (0, 1), \gamma \in (-3, 0]$. For any smooth function $f $, suppose $F= \mu +f$ satisfies 
\[
F \ge 0,\quad  \Vert F \Vert_{L^1} \ge A, \quad \Vert F \Vert_{L^1_2} +\Vert F \Vert_{L \log L} \le B,
\]
for some constant $A, B>0$. For any  $l > 10$, there exist constants $c_0, C_l>0$ such that
\begin{equation*}
\begin{aligned}
( Q(\mu+f, \mu+f) , f \langle v \rangle^{2l} ) &\le -4c_0 \Vert f \Vert_{H^s_{l +\gamma/2}}^2 + C_{l}  \Vert f \Vert_{L^2_{l +\gamma/2 - s' }}^2 + C_l \Vert f \Vert_{L^2_{|\gamma| +7 }} \Vert f \Vert_{H^s_{ l  + \gamma/2}}^2
\\
&\le - 2 c_0 \Vert f \Vert_{H^s_{l +\gamma/2}}^2 + C_{l }  \Vert f \Vert_{L^2}^2 + C_l \Vert f \Vert_{L^2_{|\gamma| +7 }} \Vert f \Vert_{H^s_{ l  + \gamma/2}}^2.
\end{aligned}
\end{equation*}
Moreover, if  $\chi$ is the cutoff function defined in \eqref{cutoff function}, we still have
\begin{equation*}
\begin{aligned}
( Q(\mu+f\chi, \mu+f) , f \langle v \rangle^{2l} ) &\le -4c_0 \Vert f \Vert_{H^s_{l+\gamma/2}}^2 + C_{l}  \Vert f \Vert_{L^2_{l + \gamma/2-s'}}^2 + C_l\Vert f \Vert_{L^2_{|\gamma| +7 }} \Vert f \Vert_{H^s_{ l  + \gamma/2}}^2
\\
&\le -2c_0 \Vert f \Vert_{H^s_{l + \gamma/2}}^2 + C_l  \Vert f \Vert_{L^2}^2 + C_l \Vert f \Vert_{L^2_{|\gamma| +7 }} \Vert f \Vert_{H^s_{ l  + \gamma/2}}^2.
\end{aligned}
\end{equation*}
\end{thm}

Combine Lemma \ref{L31} and Theorem \ref{T34}, we also have

\begin{thm}\label{T37} Suppose $\gamma+2s>-1, s \in (0, 1), \gamma \in (-3, 0]$. For any smooth function $g$, suppose $G = \mu + g$ satisfies 
\[
G \ge 0,\quad  \Vert G \Vert_{L^1} \ge A, \quad \Vert G \Vert_{L^1_2} +\Vert G \Vert_{L \log L} \le B,
\]
for some constant $A, B>0$. For any  $l > 10$, there exist constants $c_0, C_l>0$ such that
\begin{equation*}
\begin{aligned}
( Q(\mu+g, \mu+f) , f \langle v \rangle^{2l} ) \le &-4c_0 \Vert f \Vert_{H^s_{l  +  \gamma/2}}^2 + C_{l}  \Vert f \Vert_{L^2_{l + \gamma/2 - s'}}^2 + C_l \Vert g \Vert_{L^2_{|\gamma| +7 }} \Vert f \Vert_{H^s_{ l + \gamma/2}}^2 
\\
&+C_l \Vert g\Vert_{L^\infty_{l + 5 + |\gamma|}}  \Vert f \Vert_{L^2_{ l + \gamma/2}}^2+ C_l\Vert g\Vert_{L^2_{l + \gamma/2}}  \Vert f \Vert_{L^2_{ l  + \gamma/2}}
\\
\le& -2c_0 \Vert f \Vert_{H^s_{l + \gamma/2}}^2 + C_{l}  \Vert f \Vert_{L^2}^2 + C_l\Vert g \Vert_{L^2_{|\gamma| +7 }} \Vert f \Vert_{H^s_{ l + \gamma/2}}^2 
\\
&+ C_l\Vert g\Vert_{L^\infty_{l + 5 + |\gamma|}}  \Vert f \Vert_{L^2_{ l + \gamma/2}}^2 + C_l\Vert g\Vert_{L^2_{ l + \gamma/2}}  \Vert f \Vert_{L^2_{ l + \gamma/2}}.
\end{aligned}
\end{equation*}
\end{thm}

For the regularizing linear operator defined in \eqref{definition L alpha} we have
\begin{lem}\label{L38} (\cite{AMSY2}, Proposition 3.1 ) For any $\alpha \ge 0, l \ge 8$, for any function $f$ smooth, we have 
\[
\int_{\T^3} \int_{\R^3  }  L_\alpha(\mu + f) f \langle v \rangle^{2l} dv dx  \le   - \frac 1 2 \Vert f \Vert_{L^2_x L^2_{l+\alpha}}^2  - \Vert  f \Vert_{L^2_x H^1_{ l +\alpha}}^2 + C_{l, \alpha} \Vert f\Vert_{L^2_xL^2_v}^2 + C_{l, \alpha}\Vert f \Vert_{L^2_x L^2_v}.
\]
\end{lem}

\begin{lem}\label{L39} Suppose $\gamma+2s>-1, s \in (0, 1), \gamma \in (-3, 0]$. Fo any smooth function $f, g$, suppose $G= \mu +g$ satisfies 
\[
G \ge 0,\quad \inf_{t, x} \Vert G \Vert_{L^1_v} \ge A, \quad  \sup_{t, x}(\Vert G \Vert_{L^1_2} +\Vert G \Vert_{L \log L} )\le B,
\]
for some constant $A, B>0$. If $f$ is the solution to 
\[
\partial_t f + v \cdot \nabla_x f= \tilde{Q}(\mu+g , \mu +f)  = Q(\mu +g , \mu + f) +\epsilon L_\alpha(\mu + f) ,\quad \epsilon \in [0, 1],
\]
 Then for any $\epsilon \in [0, 1] $, $\alpha \ge 0$, $12 \le l \le k_0 - 8$,  suppose 
\[
\sup_{t, x} \Vert g \Vert_{L^2_{|\gamma| + 7}} <  \delta_0,  \quad \sup_{t, x} \Vert g \Vert_{L^\infty_{k_0}} <+\infty,
\]
for some $\delta_0 >0$ small. Then for any $t \ge 0$ have
\begin{equation}
\label{estimate v l 1}
\Vert \langle v \rangle^l f(t)\Vert_{L^2_{x, v}}^2 + c_0\int_0^t \Vert \langle v \rangle^l f(\tau) \Vert_{L^2_xH^s_{\gamma/2}}^2 d\tau  + \le Ce^{C(g) t} (\Vert \langle v \rangle^l f_0\Vert_{L^2_{x, v}}^2  +\sup_{t, x}\Vert g\Vert_{L^\infty_{k_0}}^2 t  +\epsilon^2 t ) ,
\end{equation}
for some constant $c_0>0$, where 
\[
C(g) = 1 + \sup_{t, x}\Vert g\Vert_{L^\infty_{k_0}}^2.
\] 
If in addition we assume $l \ge 3+2\alpha $, then for any $0 < s' < \frac {s }  {2(s+3)}$, we have
\begin{equation}
\label{estimate v l 2}
\int_{0}^t \Vert (I-\Delta_x)^{s'/2} f (\tau) \Vert_{L^2_{x, v}} d \tau \le Ce^{C(g) t} (\Vert \langle v \rangle^l f_0\Vert_{L^2_{x, v}}^2  +\sup_{t, x}\Vert g\Vert_{L^\infty_{k_0}}^2 t  +\epsilon^2 t) ,
\end{equation}
where the constant $C$ is independent of $\epsilon$. 
\end{lem}
\begin{proof} Since $\gamma \le 0$, by Theorem \ref{T37} and Lemma \ref{L38} we have
\begin{align}\label{computation estimate g}
\nonumber
\frac {d} {dt} \frac 1 2 \Vert \langle v \rangle^l f\Vert_{L^2_{x, v}}^2 = & (Q(\mu +g, \mu+f ), f \langle v \rangle^{2l})_{L^2_{x, v}} +  \epsilon (L_{\alpha}(\mu +f) , f \langle v \rangle^{2l})_{L^2_{x, v}}
\\ \nonumber
\le &-2c_0 \Vert f \Vert_{L^2_x H^s_{l +\gamma/2}}^2 + C_{l}  \Vert f \Vert_{L^2_{x, v}}^2 + C_l\sup_x \Vert g \Vert_{L^2_{|\gamma| +7 }} \Vert f \Vert_{L^2_x H^s_{ l + \gamma/2}}^2 
\\ \nonumber
&+ C_l \sup_x \Vert g\Vert_{L^\infty_{l + 5 + |\gamma| }}  \Vert f \Vert_{L^2_x L^2_{ l + \gamma/2}}^2 + C_l \Vert g\Vert_{L^2_x L^2_{ l  + \gamma/2}}  \Vert f \Vert_{L^2_x L^2_{ l + \gamma/2}} + \epsilon C_{l} \Vert f\Vert_{L^2_x  L^2_v}^2 + \epsilon C_{l}\Vert f \Vert_{L^2_x L^2_v},
\\ \nonumber
\le &  -( 2c_0 - C_l \sup_{x} \Vert g\Vert_{L^2_{|\gamma| + 7}} ) \Vert f \Vert_{L^2_x H^s_{l+\gamma/2}}^2 + C_l (1+ \sup_{x}\Vert g\Vert_{L^\infty_{k_0}} +\epsilon ) \Vert f \Vert_{L^2_x L^2_{ l + \gamma/2}}^2  
\\
&+ C_l (\sup_{x}\Vert g\Vert_{L^\infty_{k_0}} +\epsilon )  \Vert f \Vert_{L^2_x L^2_{ l + \gamma/2}}
\\ 
\le&  -c_0 \Vert f \Vert_{L^2_x H^s_{l+\gamma/2}}^2 + C_l (1+ \sup_{x}\Vert g\Vert_{L^\infty_{k_0}}  ) \Vert f \Vert_{L^2_x L^2_{ l}}^2   + \sup_{x}\Vert g\Vert_{L^\infty_{k_0}}^2 +\epsilon^2 ,
\end{align}
by Gr\"onwall's inequality we deduce \eqref{estimate v l 1}. For \eqref{estimate v l 2}, by Lemma \ref{L23} we have
\begin{equation*}
\begin{aligned}
\Vert\langle v \rangle^3 (I-\Delta_v)^{-1} ({Q}(G, F)) \Vert_{L^2_{v}}  \lesssim \Vert\langle v \rangle^3 (I-\Delta_v)^{-s} ({Q}(G, f) +Q(f, \mu) ) \Vert_{L^2_{v}}  \lesssim \Vert g\Vert_{L^2_{7}} \Vert f\Vert_{L^2_{7}} + \Vert g\Vert_{L^2_{7}} + \Vert f\Vert_{L^2_{7}},
\end{aligned} 
\end{equation*}
and
\[
\Vert\langle v \rangle^3 (I-\Delta_v)^{-1} L_\alpha (\mu +f)   \Vert_{L^2_{v}}    \le \epsilon \Vert f \Vert_{L^2_{3+2\alpha}}  + C\epsilon. 
\]
By Lemma \ref{L219} with 
\[
\beta =s, \quad m=2, \quad r =0, \quad p =2, \quad s^b = \frac {s_-} {2(s+3) } :=s',\quad T_1 =0, \quad T_2 =T,
\]
where $s_-$ is any constant satisfies $0< s_- <s$, we have
\begin{align}
\label{computation s'}
\nonumber
\int_{0}^t \Vert (I-\Delta_x)^{s'/2} f(\tau) \Vert_{L^2_{x, v}}^2 d\tau \lesssim& \Vert \langle v \rangle^3 (1-\Delta_v)^{-1} f(0)  \Vert_{L^2_{x, v}}^2 + \Vert \langle v \rangle^3 (1-\Delta_v)^{-1} f(t)  \Vert_{L^2_{x, v}}^2 
\\ \nonumber
&+ \int_{0}^t \Vert (I-\Delta_v)^{s/2} f(\tau )  \Vert_{L^2_{x, v}}^2 d\tau   + \int_{0}^t \Vert\langle v \rangle^3 (I-\Delta_v)^{-1} (\tilde{Q}(G, F)  (\tau  ) ) \Vert_{L^2_{x, v}}^2 d\tau  
\\ \nonumber
\lesssim& \Vert \langle v \rangle^3 f(0)  \Vert_{L^2_{x, v}}^2 + \Vert \langle v \rangle^3  f(t)  \Vert_{L^2_{x, v}}^2 + \int_{0}^t \Vert (I-\Delta_v)^{s/2} f (\tau  ) \Vert_{L^2_{x, v}}^2 d\tau 
\\
& + \int_{0}^t \Vert\langle v \rangle^3 (I-\Delta_v)^{-1} (\tilde{Q}(G, F) (\tau)) \Vert_{L^2_{x, v}}^2 d\tau,
\end{align} 
the proof is thus finished if we assume $l \ge 3 + 2  \alpha$.
\end{proof}

\begin{cor}\label{C310} Suppose $\gamma+2s>-1, s \in (0, 1), \gamma \in (-3, 0]$. Fo any smooth function $f$,  suppose $f$  is a solution to the modified Boltzmann equation
\[
\partial_t f+ v \cdot \nabla_x f = \epsilon L_\alpha(\mu +f) + Q(\mu + f \chi, \mu + f),
\]
where $\chi$ is the cutoff function defined in \eqref{cutoff function}. Suppose $F = \mu +f \chi$ satisfies
\[
F \ge 0,\quad \inf_{t, x} \Vert F \Vert_{L^1_v} \ge A, \quad  \sup_{t, x}(\Vert F \Vert_{L^1_2} +\Vert F \Vert_{L \log L} )\le B,
\]
For some constant $A, B >0$. Then for any $\epsilon \in [0, 1] $,  $12 \le l$,  suppose 
\[
\sup_{t, x} \Vert f \Vert_{L^2_{|\gamma|+7}} <  \delta_0,  
\]
for some $\delta_0 >0$ small. Then for any $t \ge 0$ have
\[
\Vert \langle v \rangle^l f(t)  \Vert_{L^2_{x, v}}^2 + c_0 \int_{0}^t \int_{\T^3} \Vert \langle v \rangle^l  f(\tau) \Vert_{H^s_{\gamma/2}}^2 dx d\tau \le C_l e^{C_l t} (\Vert \langle v \rangle^l f_0 \Vert_{L^2_{x, v}}^2  + \epsilon^2 t) ,
\]
for some constant $c_0, C_l>0$.  If in addition $l \ge 3+2\alpha $, then for any $0 < s' < \frac {s }  {2(s+3)}$, we have
\[
\int_{0}^t \Vert (I-\Delta_x)^{s'/2} f \Vert_{L^2_{x, v}} dt \le C_l e^{C_l  t} (\Vert \langle v \rangle^l f_0\Vert_{L^2_{x, v}}^2 +  \epsilon^2 t ). 
\]
\end{cor}

\begin{proof} Since $\gamma \le 0$, by Theorem \ref{T37} and Lemma \ref{L38} we have
\begin{equation*}
\begin{aligned}
\frac {d} {dt} \frac 1 2 \Vert \langle v \rangle^l f\Vert_{L^2_{x, v}}^2 = &(Q(\mu +f \chi, \mu+f ), f \langle v \rangle^{2l})_{L^2_{x, v}} +  \epsilon (L_{\alpha}(\mu +f) , f \langle v \rangle^{2l})_{L^2_{x, v}}
\\
\le &-2c_0 \Vert f \Vert_{L^2_x H^s_{l +\gamma/2}}^2 + C_{l}  \Vert f \Vert_{L^2_{x, v}}^2 + C_l\sup_x \Vert g \Vert_{L^2_{|\gamma| +7 }} \Vert f \Vert_{L^2_x H^s_{ l + \gamma/2}}^2  + \epsilon C_{l} \Vert f\Vert_{L^2_x  L^2_v}^2 + \epsilon C_{l}\Vert f \Vert_{L^2_x L^2_v},
\\
\le   & -( 2c_0 - C_l\sup_{x}  \Vert f \Vert_{L^2_{7+|\gamma|}}) \Vert f \Vert_{L^2_xH^s_{l + \gamma/2}}^2 + (C_l +\epsilon  ) \Vert f \Vert_{L^2_{x, v}}^2 +\epsilon \Vert f \Vert_{L^2_{x, v}}
\\
\le& -c_0 \Vert f \Vert_{L^2_xH^s_{l + \gamma/2}}^2 + C_l  \Vert f \Vert_{L^2_{x, v}}^2 +\epsilon ^2,
\end{aligned}
\end{equation*}
the remaining proof follows a  similar line  as for  Lemma \ref{L39} and thus omitted.
\end{proof}

\section{Estimate for the level set function}\label{section 4}

In this section, we focus on the linearized equation
\begin{equation}
\label{linearized equation}
\partial_t f + v \cdot \nabla_x f= \tilde{Q}(\mu+g , \mu +f)  = Q(\mu +g , \mu + f) +\epsilon L_\alpha(\mu + f) ,\quad \epsilon \in [0, 1],\quad f(0, x ,v)=f_0(x, v), \quad \mu+f \ge 0. 
\end{equation}
In this section we come to compute the level set estimate for the Boltzmann equation. It is easily seen that if $f$ is a solution to \eqref{linearized equation}, the level set function $f_{k, +}^l$ defined in \eqref{level set function} satisfies
\begin{equation}
\label{level set equation}
\partial_t (f_{K, +}^l)^2  + v \cdot \nabla_x (f_{K, +}^l)^2 = 2 \tilde{Q}(G, F) \langle v \rangle^l f_{K, +}^l. 
\end{equation}
We first prove a bound for the level set function

\begin{lem}\label{L41}
Suppose $G = \mu+g, F =\mu +f$ smooth and $\gamma \in (-3, 0], s \in (0, 1), \gamma +2s >-1$. Suppose in addition G satisfies that
\[
G \ge 0, \quad \inf_{t, x} \Vert G \Vert_{L^1_v} \ge A >0, \quad \sup_{t, x} (\Vert G \Vert_{L^1_2}  +\Vert G \Vert_{L \log L} ) < B< +\infty. 
\]
For some constant $A, B >0$. \\
(1) For any constant $l > 10, K \ge 0$ we have
\begin{align}
\label{estimate f l k}
\nonumber
\int_{\T^3} \int_{\R^3} Q(G, F) f_{K, +}^l \langle v \rangle^l dv dx \le& -(2c_1 -C_l \sup_x \Vert  g \Vert_{L^\infty_{|\gamma|+9}} )\Vert f_{K, +}^l \Vert_{L^2_x H^s_{\gamma/2}}^2 + C_{l}(1 + \sup_x \Vert g \Vert_{L^\infty_l}  ) \Vert f_{K,+}^l \Vert_{L^2_x L^2_2}^2 
\\
&+ C_l (1 + K)(1 + \sup_x \Vert g \Vert_{L^\infty_l})  \Vert f_{K, +}^l \Vert_{L^1_x L^1_2},
\end{align}
for some constant $c_1, C_l>0$. \\
(2) For the regularizing term, for any $l \ge 10, \alpha, K \ge 0$ we have
\[
\int_{\T^3} \int_{\R^3  }  L_\alpha(\mu + f) f_{K, +}^l  \langle v \rangle^{l} dv dx  \le  - \frac 1 2 \Vert \langle v \rangle^{\alpha} f_{K, +}^l \Vert_{L^2_x H^1_v}^2  +C_{l ,\alpha} \Vert  f_{K, +}^l \Vert_{L^2_{x, v}}^2 + C_{l, \alpha} (1 + K)  \Vert f_{K, +}^l \Vert_{L^1_{x, v}},
\]
for some constant $C_{l, \alpha} >0$. 
\end{lem}
\begin{proof}
The regularizing term is proved in \cite{AMSY2} Proposition 3.3. We focus on the collision term later. First make the decomposition
\begin{equation*}
\begin{aligned} 
\int_{\R^3} Q(G, F) f_{K, +}^l \langle v \rangle^l dv  =& \int_{\R^3} Q (G, f-\frac K {\langle v \rangle^l} ) f_{K, +}^l \langle v \rangle^l dv  +\int_{\R^3} Q (G, \frac {K} {\langle v \rangle^l}) f_{K, +}^l \langle v \rangle^l dv 
\\
&+\int_{\R^3} Q(G, \mu) f_{K, +}^l \langle v \rangle^l dv    :=T_1+T_2+T_3.
\end{aligned}
\end{equation*}
Recall the fact that
\[
f \langle v \rangle^l - K \le f_{K, +}^l(v)  ,\quad  (f \langle v \rangle^l - K )f_{K, +}^l(v)  = |f_{K, +}^l(v)|^2   , \quad     f_{K, +}^l(v') \ge 0, 
\]
since $G \ge 0$, we have
\begin{equation*}
\begin{aligned} 
T_1 =& \int_{\R^3} \int_{\R^3} \int_{\mathbb{S}^2} b(\cos \theta) |v-v_*|^\gamma G_* ( f - \frac {K} {\langle v \rangle^l}  )(f_{K, +}^l(v') \langle v' \rangle^l  -f_{K, +}^l(v) \langle v \rangle^l )dv dv_* d \sigma
\\
\le &  \int_{\R^3} \int_{\R^3} \int_{\mathbb{S}^2} b(\cos \theta) |v-v_*|^\gamma G_*  f_{K, +}^l (v) \frac  1 {\langle v \rangle^l}  (f_{K, +}^l(v') \langle v' \rangle^l  -f_{K, +}^l(v) \langle v \rangle^l )  dv dv_* d \sigma.
\end{aligned}
\end{equation*}
By Cauchy-Schwarz inequality we have
\begin{equation*}
\begin{aligned} 
&f_{K, +}^l(v)  \frac 1 {\langle  v \rangle^l} (f_{K, +}^l (v') \langle v' \rangle^l  - f_{K, +}^l (v) \langle v \rangle^l  ) 
\\
= & \frac 1 {\langle v \rangle^l} (f_{K,+}^l (v ) f_{K,+}^l (v') \langle v \rangle^l  \cos^l \frac \theta 2 - |f_{K,+}^l (v )|^2  \langle v \rangle^l) + \frac 1 {\langle v \rangle^l}  f_{K,+}^l (v ) f_{K,+}^l (v') (\langle v' \rangle^l  - \langle v \rangle^l  \cos^l \frac \theta 2 )
\\
\le& \frac 1 2 \left( |f_{K,+}^l (v') |^2 \cos^{2l} \frac \theta 2-|f_{K,+}^l (v )|^2 \right) + \frac 1 {\langle v \rangle^l}  f_{K,+}^l (v ) f_{K,+}^l (v') (\langle v' \rangle^l  - \langle v \rangle^l  \cos^l \frac \theta 2 ),
\end{aligned}
\end{equation*}
together with  cancellation lemma we have
\begin{equation*}
\begin{aligned} 
T_1 \le& \frac 1 2 \int_{\R^3} \int_{\R^3} \int_{\mathbb{S}^2} b(\cos \theta) |v-v_*|^\gamma G_*  \left( |f_{K,+}^l (v') | ^2 \cos^{2l} \frac \theta 2 -  |f_{K,+}^l (v )|^2 \right)  dv dv_* d\sigma
\\
&+\int_{\R^3} \int_{\R^3} \int_{\mathbb{S}^2}b(\cos \theta) |v-v_*|^\gamma G_*  \frac 1 {\langle v \rangle^l}  f_{K,+}^l (v ) f_{K,+}^l (v') (\langle v' \rangle^l  - \langle v \rangle^l  \cos^l \frac \theta 2 ) dv dv_* d\sigma
\\
\le& \frac 1 2 \int_{\R^3} \int_{\R^3} \int_{\mathbb{S}^2} b(\cos \theta) |v-v_*|^\gamma \mu_*   |f_{K,+}^l (v) |^2  (\cos^{2l-3-\gamma} \frac \theta 2-1)   dv dv_* d\sigma
\\
&+ \frac 1 2 \int_{\R^3} \int_{\R^3} \int_{\mathbb{S}^2} b(\cos \theta) |v-v_*|^\gamma g_*| f_{K,+}^l (v) |^2  (\cos^{2l-3-\gamma} \frac \theta 2-1)   dv dv_* d\sigma
\\
&+\int_{\R^3} \int_{\R^3} \int_{\mathbb{S}^2}b(\cos \theta) |v-v_*|^\gamma \mu_*  \frac 1 {\langle v \rangle^l}  f_{K,+}^l (v ) f_{K,+}^l (v') (\langle v' \rangle^l  - \langle v \rangle^l  \cos^l \frac \theta 2 ) dv dv_* d\sigma
\\
&+\int_{\R^3} \int_{\R^3} \int_{\mathbb{S}^2}b(\cos \theta) |v-v_*|^\gamma g_*  \frac 1 {\langle v \rangle^l}  f_{K,+}^l (v ) f_{K,+}^l (v') (\langle v' \rangle^l  - \langle v \rangle^l  \cos^l \frac \theta 2 ) dv dv_* d\sigma
\\
:=& T_{1,1} + T_{1, 2}+T_{1,3} +T_{1,4}.
\end{aligned}
\end{equation*}
By cancellation lemma and   Lemma \ref{L211} we have
\[
T_{1, 1} \le -\gamma_0 \Vert f_{K, +}^l \Vert^2_{L^2_{\gamma/2}}, \quad  T_{1, 2} \lesssim \Vert g \Vert_{L^\infty_{|\gamma|+4}} \Vert f_{K, +}^l \Vert^2_{L^2_{\gamma/2}},
\]
where $\gamma_0$ is defined in \eqref{gamma 0}. For the term $T_{1,3}$, by Lemma \ref{L23} apply for $f= f_{K,+}^l \frac 1 {\langle v \rangle^l}$ we have
\[
T_{1, 3} \lesssim \epsilon \Vert f_{K,+}^l \Vert_{H^s_{\gamma/2}}^2+ C_{l, \epsilon} \Vert f_{K,+}^l \Vert_{L^2_{\gamma/2 -s'}}^2.
\]
For the term $T_{1, 4}$, by \eqref{estimate Lambda 2} apply for $f=h= f_{K,+}^l \frac 1 {\langle v \rangle^l}$  we have
\[
T_{1,4} \lesssim \Vert g \Vert_{L^\infty_{|\gamma| + 9  }} \Vert f_{K, +}^l \Vert_{H^s_{ \gamma/2 }}\Vert  f_{K, +}^l  \Vert_{L^2_{  \gamma/2}}  +\Vert g \Vert_{L^\infty_l}  \Vert  f_{K, +}^l  \Vert_{L^2_{3-l}} \Vert  f_{K, +}^l  \Vert_{L^2_{2}}. 
\]
We also have another way on estimate $T_1$  term, by 
\begin{equation}
\label{f K l decomposition}
f_{K, +}^l (v) \frac 1 {\langle  v \rangle^l} (f_{K, +}^l (v') \langle v' \rangle^l  - f_{K, +}^l (v) \langle v \rangle^l  ) 
=  f_{K,+}^l (v )( f_{K,+}^l (v')  - f_{K,+}^l (v )) + \frac 1 {\langle v \rangle^l}  f_{K,+}^l (v ) f_{K,+}^l (v') (\langle v' \rangle^l  - \langle v \rangle^l   ),
\end{equation}
we have
\begin{equation*}
\begin{aligned} 
T_1 \le&  \int_{\R^3} \int_{\R^3} \int_{\mathbb{S}^2} b(\cos \theta) |v-v_*|^\gamma G_* f_{K,+}^l (v )( f_{K,+}^l (v')  - f_{K,+}^l (v ))  dv dv_* d\sigma
\\
&+\int_{\R^3} \int_{\R^3} \int_{\mathbb{S}^2}b(\cos \theta) |v-v_*|^\gamma G_*  \frac 1 {\langle v \rangle^l}  f_{K,+}^l (v ) f_{K,+}^l (v') (\langle v' \rangle^l  - \langle v \rangle^l   )dv dv_* d\sigma := T_{1,5}  +T_{1,6}.
\end{aligned}
\end{equation*}
By Lemma \ref{L24} we have
\[
T_{1,5} = (Q(G, f_{K,+}^l) ,f_{K, +}^l) \le -\gamma_1 \Vert f_{K, +}^l \Vert_{H^s_{\gamma/2}}^2 + C_l \Vert f_{K, +}^l \Vert_{L^2_{\gamma/2}}^2, 
\]
for some constant  $\gamma_1 >0$. For the term $T_{1, 6}$,  by \eqref{estimate Gamma 2} apply for $f=h= f_{K,+}^l \frac 1 {\langle v \rangle^l}$  we have
\[
T_{1,6} \lesssim \Vert g \Vert_{L^\infty_{|\gamma| + 9  }} \Vert f_{K, +}^l \Vert_{H^s_{ \gamma/2 }}\Vert  f_{K, +}^l  \Vert_{L^2_{  \gamma/2}}  +\Vert g \Vert_{L^\infty_l}  \Vert  f_{K, +}^l  \Vert_{L^2_{3-l}} \Vert  f_{K, +}^l  \Vert_{L^2_{2}} .
\]
Gathering the two estimates, for some  $\delta_2 \in (0, 1)$ small we compute
\begin{equation*}
\begin{aligned}
T_1 =& \frac{1} {1+\delta_2} (T_{1, 1} +T_{1, 2}+ T_{1, 3} +T_{1, 4}) + \frac{\delta_2} {1+\delta_2} (T_{1, 5} +T_{1, 6} )
\\
\le &\frac {-\gamma_0 +C_l \delta_2 }  {1+\delta_2}   \Vert f_{K, +}^l \Vert_{L^2_{\gamma/2}}^2 + \frac {-\gamma_1 \delta_2 + \epsilon}  {1+\delta_2} \Vert f_{K, +}^l \Vert_{H^s_{\gamma/2}}^2 + C_{l, \epsilon} \Vert f_{K,+}^l \Vert_{L^2_{\gamma/2 -s'}}^2
\\
&+ C_l \Vert g \Vert_{L^\infty_{|\gamma| + 9  }} \Vert f_{K, +}^l \Vert_{H^s_{ \gamma/2 }}\Vert  f_{K, +}^l  \Vert_{L^2_{  \gamma/2}}  + C_l \Vert g \Vert_{L^\infty_l}  \Vert  f_{K, +}^l  \Vert_{L^2_{3-l}} \Vert  f_{K, +}^l  \Vert_{L^2_{2}} 
\\
\le & -(2c_1 -C_l \Vert  g \Vert_{L^\infty_{|\gamma|+9}} )\Vert f_{K, +}^l \Vert_{H^s_{\gamma/2}}^2 + C_{l}(1+\Vert g \Vert_{L^\infty_l}  ) \Vert f_{K,+}^l \Vert_{L^2_2}^2,
\end{aligned}
\end{equation*}
for some $c_1>0$ by taking $ \delta_2  =\frac {\gamma_0} {2 C_l}, \epsilon = \frac {\gamma_1 \delta_2} 2$. For the  $T_2$ term, using \eqref{f K l decomposition} we first split it into two terms
\begin{equation*}
\begin{aligned}
T_2  =& K \int_{\R^3} \int_{\R^3} \int_{\mathbb{S^2}} b (\cos \theta) |v-v_*|^\gamma G_* \frac 1 {\langle v \rangle^l} (f_{K, +}^l(v') \langle v' \rangle^l - f_{K, +}^l(v) \langle v \rangle^l) dv dv_* d\sigma
\\
=& K \int_{\R^3} \int_{\R^3} \int_{\mathbb{S^2}} b (\cos \theta) |v-v_*|^\gamma   G_* (f_{K, +}^l(v')  - f_{K, +}^l(v) ) dv dv_* d\sigma
\\
&+ K\int_{\R^3} \int_{\R^3} \int_{\mathbb{S^2}} b (\cos \theta) |v-v_*|^\gamma  G_* \frac 1 {\langle v \rangle^l} f_{K, +}^l(v')( \langle v' \rangle^l -  \langle v \rangle^l ) dv dv_* d\sigma
:= T_{2, 1} +T_{2, 2}.
\end{aligned}
\end{equation*}
For the term $T_{2, 1}$,  by cancellation lemma and Lemma \ref{L211}  we have
\[
T_{2, 1} \lesssim K\int_{\R^3} \int_{\R^3}  |v-v_*|^\gamma   G_* f_{K, +}^l(v) dv dv_* \lesssim K \Vert G \Vert_{L^\infty_4} \Vert f_{K, +}^l \Vert_{L^1} \lesssim K  (1 + \Vert g \Vert_{L^\infty_l})  \Vert f_{K, +}^l \Vert_{L^1} .
\]
For the term $T_{2, 2}$,  by \eqref{estimate Gamma 3} we have
\[
T_{2, 2} \lesssim K \Vert  G\Vert_{L^\infty_l} \Vert f_{K, +}^l \Vert_{L^1_2} \lesssim K  (1 + \Vert g \Vert_{L^\infty_l})  \Vert f_{K, +}^l \Vert_{L^1_2}.
\]
For the $T_3$ term, similarly we have
\begin{equation*}
\begin{aligned}
T_3  =&  \int_{\R^3} \int_{\R^3} \int_{\mathbb{S^2}} b (\cos \theta) |v-v_*|^\gamma G_* \mu (f_{K, +}^l(v') \langle v' \rangle^l - f_{K, +}^l(v) \langle v \rangle^l) dv dv_* d\sigma
\\
=&  \int_{\R^3} \int_{\R^3} \int_{\mathbb{S^2}} b (\cos \theta) |v-v_*|^\gamma   G_*  \mu \langle v \rangle^l (f_{K, +}^l(v')  - f_{K, +}^l(v) ) dv dv_* d\sigma
\\
&+ \int_{\R^3} \int_{\R^3} \int_{\mathbb{S^2}} b (\cos \theta) |v-v_*|^\gamma  G_* \mu f_{K, +}^l(v')( \langle v' \rangle^l -  \langle v \rangle^l ) dv dv_* d\sigma
:= T_{3, 1} +T_{3, 2}.
\end{aligned}
\end{equation*}
For the term $T_{3, 1}$ by cancellation lemma we have
\[
T_{3, 1} \lesssim \int_{\R^3} \int_{\R^3}  |v-v_*|^\gamma   G_* |f_{K, +}^l(v)| dv dv_* \lesssim \Vert G \Vert_{L^\infty_4} \Vert f_{K, +}^l \Vert_{L^1} \lesssim (1 + \Vert g \Vert_{L^\infty_l})  \Vert f_{K, +}^l \Vert_{L^1} .
\]
For the term $T_{3, 2}$,  by \eqref{estimate Gamma 3} we have
\[
T_{3, 2} \lesssim  \Vert  G\Vert_{L^\infty_l} \Vert f_{K, +}^l \Vert_{L^1_2} \lesssim   (1 + \Vert g \Vert_{L^\infty_l})  \Vert f_{K, +}^l \Vert_{L^1_2},
\]
gathering the terms together we have 
\[
\int_{\R^3} Q(G, F) f_{K, +}^l \langle v \rangle^l dv \le -(c_1 -C_l \Vert  g \Vert_{L^\infty_{|\gamma|+9}} )\Vert f_{K, +}^l \Vert_{H^s_{\gamma/2}}^2 + C_{l}(1 +  \Vert g \Vert_{L^\infty_l}  ) \Vert f_{K,+}^l \Vert_{L^2_2}^2 + C_l (1 + K)(l +  \Vert g \Vert_{L^\infty_l})  \Vert f_{K, +}^l \Vert_{L^1_2},
\]
by integrating in $x$, \eqref{estimate f l k} is proved . 
\end{proof}

To show $|f| \langle v \rangle^k \le K$ later, we will need not only the level set function $f_{K, +}^l$, but  also the one for  $-f \langle v \rangle^l \le K$. It is not surprising that the estimate for $(-f)_{K, +}^l$ follows a similar line as that for $f_{K, +}^l$. The equation for $h=-f$ is 
\[
\partial_t h + v \cdot \nabla_x h =-Q(G, \mu - h ) -\epsilon L_{\alpha} (u-h),
\]
for the function $h$ we have the following estimate

\begin{lem}\label{L42}
Suppose $G = \mu+g $ smooth and $\gamma \in (-3, 0], s \in (0, 1), \gamma +2s >-1$. Suppose in addition G satisfies that
\[
G \ge 0, \quad \inf_{t, x} \Vert G \Vert_{L^1_v} \ge A >0, \quad \sup_{t, x} (\Vert G \Vert_{L^1_2}  +\Vert G \Vert_{L \log L} ) < B< +\infty. 
\]
For some constant $A, B >0$. \\
(1) For any constant $l > 10, K \ge 0$ we have
\begin{align}
\nonumber
- \int_{\T^3} \int_{\R^3} Q(G, \mu-h) h_{K, +}^l \langle v \rangle^l dv dx \le& -(c_1 -C_l \sup_x \Vert  g \Vert_{L^\infty_{|\gamma|+9}} )\Vert h_{K, +}^l \Vert_{L^2_x H^s_{\gamma/2}}^2 + C_{l}(1 + \sup_x \Vert g \Vert_{L^\infty_l}  ) \Vert h_{K,+}^l \Vert_{L^2_x L^2_2}^2 
\\
&+ C_l (1 + K)(l + \sup_x \Vert g \Vert_{L^\infty_l})  \Vert h_{K, +}^l \Vert_{L^1_x L^1_2},
\end{align}
for some constant $c_1,C_l>0$. \\
(2) For the regularizing term, for any $l \ge 10, \alpha, K \ge 0$ we have
\[
-\int_{\T^3} \int_{\R^3  }  L_\alpha(\mu - h) h_{K, +}^l  \langle v \rangle^{l} dv dx  \le  - \frac 1 2 \Vert \langle v \rangle^{\alpha} h_{K, +}^l \Vert_{L^2_x H^1_v}^2  +C_{l ,\alpha} \Vert  h_{K, +}^l \Vert_{L^2_{x, v}}^2 + C_{l, \alpha} (1 + K)  \Vert h_{K, +}^l \Vert_{L^1_{x, v}},
\]
for some constant $C_{l, \alpha} >0$. 
\end{lem}
\begin{proof}
The regularizing term is proved in \cite{AMSY2} Proposition 3.4. We focus on the term later. First make the  decomposition
\begin{equation*}
\begin{aligned} 
-\int_{\R^3} Q(G, \mu - h) h_{K, +}^l \langle v \rangle^l dv  =& \int_{\R^3} Q(G, h-\frac K {\langle v \rangle^l}) h_{K, +}^l \langle v \rangle^l dv  +\int_{\R^3} Q(G, \frac {K} {\langle v \rangle^l}) h_{K, +}^l \langle v \rangle^l dv 
\\
&- \int_{\R^3} Q(G, \mu) f_{K, +}^l \langle v \rangle^l dv    :=K_1+K_2+K_3,
\end{aligned}
\end{equation*}
The $K_1$ $K_2$ term is the same as $T_1, T_2$ term in Lemma \ref{L41}, we have
\[
K_1 + K_2 \le -(c_1 -C_l \Vert  g \Vert_{L^\infty_{|\gamma|+9}} )\Vert h_{K, +}^l \Vert_{H^s_{\gamma/2}}^2 + C_{l}(1+\Vert g \Vert_{L^\infty_l}  ) \Vert h_{K,+}^l \Vert_{L^2_2}^2 
+ C_l K(l + \Vert g \Vert_{L^\infty_l})  \Vert h_{K, +}^l \Vert_{L^1_2},
\]
For the $K_3$ term, similarly as $T_3$ term in Lemma \ref{L41} we have
\begin{equation*}
\begin{aligned}
K_3  =&  \int_{\R^3} \int_{\R^3} \int_{\mathbb{S^2}} b (\cos \theta) |v-v_*|^\gamma G_* (-\mu) (f_{K, +}^l(v') \langle v' \rangle^l - f_{K, +}^l(v) \langle v \rangle^l) dv dv_* d\sigma
\\
=&  \int_{\R^3} \int_{\R^3} \int_{\mathbb{S^2}} b (\cos \theta) |v-v_*|^\gamma   G_*  (-\mu) \langle v \rangle^l (f_{K, +}^l(v')  - f_{K, +}^l(v) ) dv dv_* d\sigma
\\
&+ \int_{\R^3} \int_{\R^3} \int_{\mathbb{S^2}} b (\cos \theta) |v-v_*|^\gamma  G_* (-\mu) f_{K, +}^l(v')( \langle v' \rangle^l -  \langle v \rangle^l ) dv dv_* d\sigma:= K_{3, 1} +K_{3, 2}.
\end{aligned}
\end{equation*}
For the term $K_{3, 1}$, by \eqref{estimate Gamma 3}  we have
\[
K_{3, 1} \lesssim \int_{\R^3} \int_{\R^3}  |v-v_*|^\gamma   G_* |f_{K, +}^l(v)| dv dv_* \lesssim \Vert G \Vert_{L^\infty_4} \Vert f_{K, +}^l \Vert_{L^1} \lesssim (1 + \Vert g \Vert_{L^\infty_l})  \Vert f_{K, +}^l \Vert_{L^1}.
\]
For the term $K_{3, 2}$ by Lemma \ref{L23} we have
\[
K_{3, 2} \lesssim  \Vert  G\Vert_{L^\infty_l} \Vert f_{K, +}^l \Vert_{L^1_2} \lesssim   (1 + \Vert g \Vert_{L^\infty_l})  \Vert f_{K, +}^l \Vert_{L^1_2},
\]
the proof is ended by gathering all the terms and integrate in $x$.
\end{proof}

\begin{lem}\label{L43} Suppose $\gamma+2s>-1, s \in (0, 1), \gamma \in (-3, 0]$. For any smooth function $f, g$, suppose $G = \mu + g$ satisfies 
\[
G \ge 0,\quad  \Vert G \Vert_{L^1} \ge A, \quad \Vert G \Vert_{L^1_2} +\Vert G \Vert_{L \log L} \le B,
\]
for some constant $A, B>0$. Then for any $T>0$ and
\[
[T_1, T_2 ] \in [0, T] , \quad \epsilon \in [0, 1], \quad j \ge 0,  \quad \tau > 2,\quad  K \ge 0, \quad l \ge 10,
\]
we have
\begin{equation*}
\begin{aligned} 
&\int_{T_1}^{T_2} \int_{\T^3} \int_{\R^3} \left| \langle v \rangle^j (1-\Delta_v)^{-\tau/2} ( \tilde{Q}(G, F )  \langle v \rangle^l f_{K, +}^l  )           \right| dv dx dt
\\
\le& C \Vert \langle v \rangle^{j/2} f_{K, + }^l(T_1, \cdot, \cdot ) \Vert_{L^2_{x, v}}^2 + C (1+ \sup_{t, x}  \Vert g \Vert_{L^\infty_{9+|\gamma|}}) \Vert f_{K, +}^l \Vert_{L^2_{t, x} H^s_{\gamma/2}}^2 + C(1 + \sup_{t, x} \Vert g \Vert_{L^\infty_l}  )  \Vert f_{K, +}^l    \Vert_{L^2_{t, x} L^2_{j+2}}^2
\\
& + C (1 + K)   (1 +\sup_{t, x}  \Vert g \Vert_{L^\infty_l})   \Vert f_{K, +}^l \Vert_{L^1_{t, x} L^1_{j+2}},
\end{aligned}
\end{equation*}
where the constant $C$ is independent of $\epsilon$. Identical estimate holds for $\tilde{Q}(G, -\mu +h)$ with $f_{K, +}^l$ replaced by $h_{K, +}^l$
\end{lem}
\begin{proof}
First note that for any $\tau \ge 0$, by the theory of Bessel kernel
\[
(I - \Delta_v)^{-\tau/2} f = G_\tau * f, \quad G_\tau \in L^1(\R^d),
\]
which implies
\[
\int_{\R^3} \int_{\T^3} \langle v \rangle^j (1-\Delta_v)^{-\tau/2} (f_{K, +}^l)^2 (T_1, x, v) dx dv \le C \Vert \langle v \rangle^{j/2} f_{K, +}^l (T_1, \cdot, \cdot) \Vert_{L^2_{x, v}}^2.
\]
By Lemma \ref{L220} we have that for any $j \ge 0, l \ge 0, \tau \ge 0$ 
\begin{equation*}
\begin{aligned} 
&\int_{T_1}^{T_2} \int_{\T^3} \int_{\R^3} \left| \langle v \rangle^j (1-\Delta_v)^{-\tau/2} ( \tilde{Q}(G, F )  \langle v \rangle^l f_{K, +}^l  )           \right| dv dx dt -  C \Vert \langle v \rangle^{l/2} f_{K, +}^l (T_1, \cdot, \cdot) \Vert_{L^2_{x, v}}^2
\\
& \le 2 \int_{T_1}^{T_2} \int_{\T^3} \int_{\R^3} \left[ \langle v \rangle^j (1-\Delta_v)^{-\tau/2}  (\tilde{Q}(G, F )  \langle v \rangle^l f_{K, +}^l     )        \right]^+ dv dx dt
\\
& = 2 \int_{T_1}^{T_2} \int_{\T^3} \int_{\R^3} \langle v \rangle^j (1-\Delta_v)^{-\tau/2}  (\tilde{Q}(G, F )  \langle v \rangle^l f_{K, +}^l  )     1_{A_K} dv dx dt
\\
& = 2 \int_{T_1}^{T_2} \int_{\T^3} \int_{\R^3} \tilde{Q}(G, F )  \langle v \rangle^l f_{K, +}^l  (1-\Delta_v)^{-\tau/2}  (\langle v \rangle^j   1_{A_K}  )    dv dx dt,
\end{aligned}
\end{equation*}
where $A_K$ is the set given by 
\[
A_K = \{   (t, x, v) \in (T_1, T_2) \times \T^3 \times \R^3 |(1-\Delta_v)^{-\tau/2}  (\tilde{Q}(G, F )  \langle v \rangle^l f_{K, +}^l  )   \ge 0  \}.
\]
Denote 
\begin{equation}
\label{W k}
W_K(v) : =  (1-\Delta_v)^{-\tau/2}  (\langle v \rangle^j   1_{A_K}  )   \ge 0,
\end{equation}
since $\tau >2$, by (3.39) in \cite{AMSY2} we have
\[
|W_K(v) | + |\nabla_i W_K(v) | +| \nabla^2_{i, j} W_k(v) | \le C \langle v \rangle^{j} ,\quad i, j = 1,2,3,
\]
with $C>0$ is independent of $K$. Then we only need to estimate the term
\begin{equation*}
\begin{aligned} 
&\int_{T_1}^{T_2} \int_{\T^3} \int_{\R^3} \tilde{Q}(G, F )  \langle v \rangle^l f_{K, +}^l  W_k  dv dx dt
\\
=&\int_{T_1}^{T_2} \int_{\T^3} \int_{\R^3} {Q}(G, F )  \langle v \rangle^l f_{K, +}^l  W_k  dv dx dt  + \epsilon \int_{T_1}^{T_2} \int_{\T^3} \int_{\R^3}   L_\alpha (F)  \langle v \rangle^l f_{K, +}^l  W_k  dv dx dt,
\end{aligned}
\end{equation*}
For the second term, by \cite{AMSY2} Proposition 3.7 we have
\[
\int_{\T^3} \int_{\R^3  }    L_\alpha(\mu + f) f_{K, +}^l  \langle v \rangle^{l} W_k dv dx  \le C_{l ,\alpha} \Vert  f_{K, +}^l \Vert_{L^2_{x}L^2_{j/2} }^2 + C_{l, \alpha} (1 + K)  \Vert f_{K, +}^l \Vert_{L^1_{x}L^1_j }.
\]
We focus on the collision term later. First make the decomposition
\begin{equation*}
\begin{aligned} 
\int_{\R^3} Q(G, F) f_{K, +}^l \langle v \rangle^l W_K dv  =& \int_{\R^3} Q(G, f-\frac K {\langle v \rangle^l}) f_{K, +}^l \langle v \rangle^l  W_K dv  +\int_{\R^3} Q(G, \frac {K} {\langle v \rangle^l}) f_{K, +}^l \langle v \rangle^l W_K dv 
\\
&+\int_{\R^3} Q(G, \mu) f_{K, +}^l \langle v \rangle^l W_K dv    :=T_1'+T_2'+T_3'.
\end{aligned}
\end{equation*}
For the $T_1'$ term performing a similar argument we still have
\begin{equation*}
\begin{aligned} 
T_1' =& \int_{\R^3} \int_{\R^3} \int_{\mathbb{S}^2} b(\cos \theta) |v-v_*|^\gamma G_* ( f - \frac {K} {\langle v \rangle^l}  )(f_{K, +}^l(v') W_K(v') \langle v' \rangle^l  -f_{K, +}^l(v) \langle v \rangle^l W_K(v) )dv dv_* d \sigma
\\
\le &  \int_{\R^3} \int_{\R^3} \int_{\mathbb{S}^2} b(\cos \theta) |v-v_*|^\gamma G_*  f_{K, +}^l (v) \frac  1 {\langle v \rangle^l}  (f_{K, +}^l(v') W_K(v')  \langle v' \rangle^l  -f_{K, +}^l(v) \langle v \rangle^l W_K(v) )  dv dv_* d \sigma.
\end{aligned}
\end{equation*}
By the Cauchy-Schwarz inequality 
\begin{equation*}
\begin{aligned} 
&f_{K, +}^l(v) \frac 1 {\langle  v \rangle^l} (f_{K, +}^l (v')W_K(v') \langle v' \rangle^l  - f_{K, +}^l (v)  W_k(v ) \langle v \rangle^l  ) 
\\
 =& \frac 1 {\langle v \rangle^l} (f_{K,+}^l (v ) f_{K,+}^l (v')W_K(v')  \langle v \rangle^l  \cos^l \frac \theta 2 - |f_{K,+}^l (v )|^2  \langle v \rangle^l W_K) 
+ \frac 1 {\langle v \rangle^l}  f_{K,+}^l (v ) f_{K,+}^l (v') W_K(v' ) (\langle v' \rangle^l  - \langle v \rangle^l  \cos^l \frac \theta 2 )
\\
\le &\frac 1 2 \left( |f_{K,+}^l (v') |^2 W_K(v') \cos^{2l} \frac \theta 2  - |f_{K,+}^l (v )|^2 W_K(v) \right) 
+ \frac 1 {\langle v \rangle^l}  f_{K,+}^l (v ) f_{K,+}^l (v') W_K(v') (\langle v' \rangle^l  - \langle v \rangle^l  \cos^l \frac \theta 2 )  
\\
&+\frac 1 2 |f_{K, +}^l(v) |^2 (W_K(v')-W_K(v) ),
\end{aligned}
\end{equation*}
we split  $T_1'$ term into 3 terms $T_1' \le T_{1, 1}' + T_{1, 2}' +T_{1, 3}'$ respectively. For the  term $T_{1, 1}'$  by cancellation lemma and Lemma \ref{L212} we have
\begin{equation*}
\begin{aligned} 
T_{1,1}' =& \frac 1 2 \int_{\R^3} \int_{\R^3} \int_{\mathbb{S}^2} b(\cos \theta) |v-v_*|^\gamma G_*  \left( |f_{K,+}^l (v') |^2 W_k(v' ) \cos^{2l} \frac \theta 2-|f_{K,+}^l (v )|^2 W_k(v )  \right)  dv dv_* d\sigma
\\
= &  -\gamma_0 \int_{\R^3} \int_{\R^3}  |v-v_*|^\gamma (\mu + g ) (f_{K,+}^l (v ))^2 W_k(v )   dv dv_* \le -\gamma_0 (1- C_l \Vert  g \Vert_{L^\infty_{4+|\gamma| } }  )  \Vert f_{K, +}^l \sqrt{W_K} \Vert_{L^2_{\gamma/2}}.
\end{aligned}
\end{equation*}
where $\gamma_0$ is defined in \eqref{gamma 0}. For the $T_{1, 2}'$ term by \eqref{estimate Gamma 3} we have
\begin{equation*}
\begin{aligned} 
T_{1, 2}' = &\int_{\R^3} \int_{\R^3} \int_{\mathbb{S}^2}b(\cos \theta) |v-v_*|^\gamma G_*  \frac 1 {\langle v \rangle^l}  f_{K,+}^l (v ) f_{K,+}^l (v')W_k(v') (\langle v' \rangle^l  - \langle v \rangle^l  \cos^l \frac \theta 2 ) dv dv_* d\sigma
\\
\lesssim&\Vert G \Vert_{L^\infty_{|\gamma| + 9  }} \Vert f_{K, +}^l \Vert_{H^s_{ \gamma/2 }}\Vert f_{K, +}^l  W_k \Vert_{L^2_{  \gamma/2}}  +\Vert G \Vert_{L^\infty_l}  \Vert  f_{K, +}^l \Vert_{L^2_{3-l} } \Vert f_{K, +}^l  W_k  \Vert_{L^2_{2}} 
\\
\lesssim& (1+ \Vert g \Vert_{L^\infty_{|\gamma| + 9  }} ) \Vert f_{K, +}^l \Vert_{H^s_{ \gamma/2 }}^2  +(1 + \Vert g \Vert_{L^\infty_l}  )  \Vert f_{K, +}^l    \Vert_{L^2_{j+2}}^2.
\end{aligned}
\end{equation*}
For the $T_{1, 3}'$ term by Lemma \ref{L215} and Lemma \ref{L211} we have
\begin{equation*}
\begin{aligned} 
T_{1, 3}' =& \int_{\R^3} \int_{\R^3} \int_{\mathbb{S}^2} b(\cos \theta) |v-v_*|^\gamma G_*  |f_{K, +}^l(v)|^2  ( W_K' - W_K )dv dv_* d \sigma
\\
\lesssim& \int_{\R^3} \int_{\R^3}  |v-v_*|^{2+\gamma}  G_*  |f_{K, +}^l(v')|^2  \left(  \sup_{|u| \le |v_*| + |v|} |\nabla W_K(u)|  +  \sup_{|u| \le |v_*| + |v|} |\nabla^2 W_K (u)| \right) dv dv_* 
\\
\lesssim& \int_{\R^3} \int_{\R^3}  |v-v_*|^{\gamma}  G_*  \langle v_* \rangle^{j+2} |f_{K, +}^l(v')|^2  \langle v \rangle^{j+2} dv dv_* 
\\
\lesssim& (1 + \Vert g \Vert_{L^\infty_{j+6}} )\Vert f_{K, +}^l\Vert_{L^2_{j/2+1}}^2. 
\end{aligned}
\end{equation*}
The term  $T_2', T_3'$ is the same as the term $T_2, T_3$ in Lemma \ref{L41}, thus we have
\[
|T_2'| + |T_3'| \lesssim C(1 + K)   (1 + \Vert g \Vert_{L^\infty_l})  \Vert f_{K, +}^l W_K\Vert_{L^1_2} \lesssim C(1 + K)   (1 + \Vert g \Vert_{L^\infty_l})   \Vert f_{K, +}^l \Vert_{L^1_{j+2}},
\]
so the proof is ended by gathering all the terms together and integrate in $t $ and $x$. 
\end{proof}

For any $l \ge 0, p>1, 0< s''<\frac s {2(s+3)}, K \ge 0, 0 \le T_1 \le T_2$ fixed, we define the energy functional 
\begin{align}
\label{energy functional}
\nonumber
\mathcal{E}_{p, s''} (K, T_1, T_2) :=& \sup_{t \in [T_1, T_2] } \Vert f_{K, +}^l \Vert_{L^2_{x, v}}^2   + \int_{T_1}^{T_2} \int_{\T^3} \Vert \langle  v \rangle^{\gamma/2} f_{K, +}^l \Vert_{H^s_v}^2 dx d\tau 
\\ 
&+\frac 1 {C_0} \left( \int_{T_1}^{T_2} \Vert (1-\Delta_x)^{s''/2} ( \langle v \rangle^{-2+ \gamma/2} f_{K, +}^l)^2\Vert_{L^p_{x, v}}^p d\tau \right)^{1/p}.
\end{align}
where $C_0$ is a constant which will be determined later.

\begin{lem}\label{L44}
Let the parameters $T_1, T_2, s, s'', l, n$ be given such that
\[
0 \le T_1 <T_2 < T, \quad 0 <s'' < \frac s {2(s+3)} \in (0, 1), \quad  l \ge 0, \quad  n \ge 0. 
\]
There exist a constant $p^{c} >1$ depends on $s, s''$ such that for any $p \in (1, p^{c})$ fixed, there exist a $l_0>0$ large such that if we suppose
\[
\sup_{t}\Vert \langle v \rangle^{l+l_0} f \Vert_{L^2_{x, v}} \le   C_1,
\]
for some constant $C_1>0$. Then there exist a constant $q_*$ which is independent of $p$ and satisfies $1 < q_* <\frac {r(1)} 2$ such that the following holds: For any $1<q <q_*$ we can find a pair of parameter $(r_*, \xi_*)$ with the properties
\[
r_* >q_*>q>1, \quad \xi_* >2q_* >2q >2,
\] 
such that for any $0 \le M  < K $ and $0 \le T_1 \le T_2 \le T$ we have
\[ 
\Vert \langle v \rangle^{\frac n q} (f_{K, +}^l)^2 \Vert_{L^q((T_1, T_2  ) \times \T^3 \times \R^3)} \le C \frac { \mathcal{E}_{p, s''}  (M, T_1, T_2)^{\frac {r_*} q} } {(K-M)^{\frac {\xi_* - 2q} q }   } ,
\] 
where C  only depends on $(C_1, s, s'', q, p)$ and $q_*$ only depends on $(s, s'')$,  $r_*, \xi_*$ only depends on $(s, s'', p)$ and  $l_0 $ only depends on $(s, s'', p, n)$. In particular, all these parameters are independent of $K, M, T_1, T_2, l$ and $f$.
\end{lem}
\begin{proof}
For any $p >1$, define 
\[
r(1) : = \tilde{r } (s, s'', 1, 3) , \quad r(p ) := \tilde{r } (s, s'', p, 3),
\]
where $\tilde{r}(\cdot, \cdot, \cdot, \cdot)$ is defined in Lemma \ref{L217}. By the continuity of $r( \cdot)$, 
\[
\frac {r(z)} {2z} \frac {r(1) -2} {r(z) -2} \to \frac {r(1)} 2 >1,  \quad \hbox{as} \quad  z \to 1,
\]
there exist a $p^{c}  \in (1, 2)$ fixed and close enough to $1$ such that
\begin{equation}
\label{requirement p c}
\min_{[1, p^c]} \frac {r(p)} {2p} \frac {r(1) -2} {r(p) -2} >1. 
\end{equation}
Choose
\[
q_* : = 1+ \frac {r(1) -2} {2 r(p^c)}, \quad 1<  q_* < r(1)/2, 
\]
it is clear that $q_*$ only depends on $p^c, s, s''$(independent of $p$). Such choice guarantees that
\begin{equation}
\label{requirement q *}
\frac {r(p^{c})} 2 \frac {2 q_* -2} {r(1) -2}  =\frac 1 2 <1.
\end{equation}
For any $\xi, q$ satisfying $2<2q <\xi<r(1) <r(p)$,  define $\beta =\beta(\xi, p) \in (0, 1) $ by 
\[
\frac 1 {\xi} = \frac {1-\beta} {2} +\frac \beta {r(p)}, \quad \beta \in (0, 1), \quad \beta \xi = \frac {r(p) (\xi -2)} { r(p) -2 }. 
\]
By \eqref{requirement q *} we have
\[
\frac {\beta \xi} {2p} <\frac {\beta \xi} {2 }  =  \frac {r(p)} 2 \frac { \xi -2} { r(p) -2 } \le \frac {r(p^c)} 2 \frac { \xi -2} { r(1) -2 }<1 , \quad \hbox{as} \quad \xi \to 2 q_*, \quad \xi >2q_*, \quad \forall p \in(1, p^c),
\]
and by \eqref{requirement p c} we have
\[
\frac {\beta \xi} {2} > \frac {\beta \xi} {2p}  =  \frac {r(p)} {2p} \frac { \xi -2} { r(p) -2 } \to \frac {r(p)} {2p} \frac { r(1)  - 2 } { r(p) -2 }>1, \quad \hbox{as} \quad \xi \to r(1), \quad \xi < r(1), \quad \forall p \in(1, p^c).
\]
So for any $\zeta \in (0, 1)$, for any $p \in (1, p^c)$ we can find a pair $\beta_*, \xi_*$ such that
\[
\zeta \frac {\beta_* \xi_*} {2} +(1-\zeta) \frac {\beta_* \xi_*} {2p} =1, \quad \frac 1 {\xi_*} = \frac {1-\beta_*} {2} +\frac {\beta_*} {r(p)}. 
\]
We will take $\zeta = \tilde{\alpha}(s, s'', p, 3)$ later, where $\tilde{\alpha}(\cdot, \cdot, \cdot, \cdot)$ is defined in Lemma \ref{L217}. Wwith the parameters above we can start prove the theorem. For any $q \in (1, q_*)$, since $K > M > 0$ implies $f_{M, +}^l \ge f_{K, +}^l +(K-M) $ , by H\"older inequality  we have
\begin{equation*}
\begin{aligned} 
\Vert \langle v \rangle^{\frac n q} (f_{K, +}^l)^2 \Vert_{L^q_{t, x, v}}^q  =& \int_{T_1}^{T_2} \Vert \langle v \rangle^{\frac n {2q} } f_{K, +}^l  \Vert_{L^{2q}_{x, v}}^{2q} d\tau  \le\frac 1 { (K-M)^{\xi_* -2q} } \int_{T_1}^{T_2} \Vert \langle v \rangle^{n/\xi_*} f_{M, +}^l\Vert_{L^{\xi_*}_{x, v} }^{\xi_*} d\tau
\\
\le &\frac 1 { (K-M)^{\xi_* -2q} } \int_{T_1}^{T_2} \Vert\langle v \rangle^{a_0} f_{M, +}^l \Vert_{L^2_{x, v}}^{(1-\beta_*)\xi_*} \Vert\langle v \rangle^{-4+ \gamma/2} f_{M, +}^l \Vert_{L^{r(p)}_{x, v}}^{\beta_*\xi_*} d\tau ,
\end{aligned}
\end{equation*}
where $a_0 = \frac 1 {1-\beta_*} (\frac n {\xi_*} + (4-\gamma/2) \beta_* )$. By Lemma \ref{L217} with parameter $(r(p), s, s'' , p )$ and H\"older's inequality we have
\begin{equation*}
\begin{aligned}
&\int_{T_1}^{T_2} \Vert\langle v \rangle^{a_0} f_{M, +}^l \Vert_{L^2_{x, v}}^{(1-\beta_*)\xi_*} \Vert\langle v \rangle^{-4+ \gamma/2} f_{M, +}^l \Vert_{L^{r(p)}_{x, v}}^{\beta_*\xi_*} d\tau
\\
\le& C \int_{T_1}^{T_2}  \Vert\langle v \rangle^{a_0} f_{M, +}^l \Vert_{L^2_{x, v}}^{(1-\beta_*)\xi_*} \Vert (- \Delta_v )^{s/2}  (\langle v \rangle^{-4+\gamma/2} f_{M, +}^l )  \Vert_{L^2_{x, v}}^{\zeta \beta_* \xi_*}  \Vert  \langle v \rangle^{-8 + \gamma}  (1- \Delta_x )^{\frac {s''} 2} (f_{M, +}^ l )^2 \Vert_{L^1_v L^p_x}^{\frac{ 1-\zeta} 2 \xi_* \beta_* }  d\tau
\\
\le& C \int_{T_1}^{T_2}  \Vert\langle v \rangle^{a_0} f_{M, +}^l \Vert_{L^2_{x, v}}^{(1-\beta_*)\xi_*} \Vert (- \Delta_v )^{s/2}  (\langle v \rangle^{\gamma/2} f_{M, +}^l )  \Vert_{L^2_{x, v}}^{\zeta \beta_* \xi_*}  \Vert  \langle v \rangle^{-4+ \gamma}  (1- \Delta_x )^{\frac {s''} 2} (f_{M, +}^ l )^2 \Vert_{L^p_{x, v} }^{\frac{ 1-\zeta} 2 \xi_* \beta_* }  d\tau
\\
\le& C\left( \sup_{t}  \Vert\langle v \rangle^{a_0} f_{M, +}^l     \Vert_{L^2_{x, v}}^{(1-\beta_*)\xi_*}   \right)    \left( \int_{T_1}^{T_2} \Vert (- \Delta_v )^{s/2}  ( \langle v \rangle^{\gamma/2} f_{M, +}^l )  \Vert_{L^2_{x, v}}^{2}  d\tau \right)^{\frac {\zeta \beta_* \xi_*} 2 } 
\\
&\times \left( \int_{T_1}^{T_2} \Vert  \langle v \rangle^{-4+ \gamma}  (1- \Delta_x )^{\frac {s''} 2} (f_{M, +}^ l )^2 \Vert_{L^p_{x, v} }^{ p} d\tau  \right)^{\frac{ 1-\zeta} {2p} \xi_* \beta_* }
\\
 \le  &C\left( \sup_{t}  \Vert\langle v \rangle^{a_0} f_{M, +}^l     \Vert_{L^2_{x, v}}^{(1-\beta_*)\xi_*}   \right)   \mathcal{E}_{p, s''} (M, T_1. T_2)^{\frac {\zeta \beta_* \xi_*} 2 }  \mathcal{E}_{p, s''} (M, T_1. T_2)^{\frac{ 1-\zeta} {2} \xi_* \beta_* }. 
\end{aligned}
\end{equation*}
By H\"older inequality we easily have
\begin{equation*}
\begin{aligned}
&\sup_t \Vert\langle v \rangle^{a_0} f_{M, +}^l     \Vert_{L^2_{x, v}}^{(1-\beta_*)\xi_*}  \le  \sup_t\Vert  \langle v \rangle^{l_0} f_{M, +}^l     \Vert_{L^2_{x, v}}^{(1-\beta' )(1-\beta_*)\xi_*}   \sup_t \Vert  f_{M, +}^l     \Vert_{L^2_{x, v}}^{\beta' (1-\beta_*)\xi_*} 
\\
\le & C_1^{(1-\beta')(1-\beta_*)\xi_*   } \mathcal{E}_{p, s''}(M, T_1, T_2)^{\beta' (1-\beta_*) \frac {\xi_*} 2},
\end{aligned}
\end{equation*}
where $l_0 =\frac {a_0} {1-\beta'}$. Gathering all the terms we deduce
\[
\int_{T_1}^{T_2} \Vert\langle v \rangle^{a_0} f_{M, +}^l \Vert_{L^2_{x, v}}^{(1-\beta_*)\xi_*} \Vert\langle v \rangle^{-4+ \gamma/2} f \Vert_{L^{r(p)}_{x, v}}^{\beta_*\xi_*} d\tau \le C C_1^{(1-\beta')(1-\beta_*)\xi_*   } E_P(M, T_1, T_2)^{r_*},
\]
with 
\[
r_* = (1- (1 -\beta') (1-\beta_*)  ) \frac  {\xi_*} 2,
\]
since $\xi_* > 2q_* $ we can make $r_* >q_*$ by taking $l_0 \gg a_0$ which could implies that $\beta'$ is arbitrarily near $1$. We can see that $r_*, \xi_*$ only depends on $s, s'', p$ and $l_0$ only depends on $\beta_*, \xi_*, n, \beta'$ which only depends on $s, s'', p, n$.
\end{proof}

With the lemma above, we are ready to estimate regarding the energy functional 
\begin{lem}\label{L45}
Suppose $G = \mu+g, F =\mu +f$ smooth and $\gamma \in (-3, 0], s \in (0, 1), \gamma +2s >-1$. Let $T>0, \alpha \ge 0$ be fixed.  Assume that $f$ is a solution to \eqref{linearized equation}. Suppose in addition G satisfies that
\[
G \ge 0, \quad \inf_{t, x} \Vert G \Vert_{L^1_v} \ge A >0, \quad \sup_{t, x} (\Vert G \Vert_{L^1_2}  +\Vert G \Vert_{L \log L} ) < B < +\infty,
\]
for some constant $A, B>0$.  Then for any $\epsilon \in [0, 1] $,  $12 \le l \le k_0 $, suppose 
\begin{equation}
\label{assumption on g 1}
\sup_{t, x} \Vert  g  \Vert_{L^\infty_{9+|\gamma|}}  \le \delta_0, \quad \sup_{t, x} \Vert g \Vert_{L^\infty_{k_0}}  \le C,
\end{equation}
for some constant $\delta_0 >0$ small. Then there exists $s"> 0 $ and $p^a>0 $ such that for any  $p \in (1, p^a)$  there exist $l_0>0$ which depends on $s'', p$ such that, if we assume
\[ 
\sup_{t} \Vert \langle  v \rangle^{l_0+l }  f  \Vert_{L^2_{x, v} } \le C_1 < +  \infty,
\]
Then for any $0  \le T_1 \le T_2 \le T,  \epsilon \in (0, 1),  0 \le M <K$ we have
\begin{equation*}
\begin{aligned} 
&\Vert f_{K, +}^l (T_2) \Vert_{L^2_{x, v}}^2 + \int_{T_1}^{T_2} \Vert \langle v \rangle^{\gamma/2} (1-\Delta_v)^{s/2}  f_{K, +}^l \Vert_{L^2_{x, v}}^2 d \tau +\frac 1 {C_0} \left(       \int_{T_1}^{T_2} \Vert (1-\Delta_x)^{{s''}/2}  (\langle v \rangle^{-4+ \gamma} (f_{K, +}^l )^2) \Vert_{L^p_{x, v}}^p d \tau                                 \right)^{\frac 1 p}
\\
\le& C \Vert \langle v \rangle^2 f_{K, +}^l (T_1) \Vert_{L^2_{x, v}}^2 +  C \Vert \langle v \rangle^2 f_{K, +}^l (T_1)  \Vert_{L^{2p}_{x, v}}^2 +\frac {CK} {K - M} \sum_{i=1}^4 \frac {\mathcal{E}_{p, s''}(M, T_1, T_2)^{\beta_i}} {(K-M)^{a_i}},
\end{aligned}
\end{equation*}
for some constant  $\beta_i > 1, a_i > 0$ are defined later. The constant $C$ is independent of $\epsilon, K, M, f, T_1, T_2$. Furthermore, the same estimate holds for $h=-f$ with $f_{K, +}^l $ replaced by $h_{K, +}^l$.
\end{lem}
\begin{proof}
If $f$ is a solution to \eqref{linearized equation} we have $f_{K, +}^l$ satisfies \eqref{level set equation}. Choose $\sigma =1/4, \tau_1 > 2$, $f_{K, +}^l$ satisfies 
\[
\partial_t (f_{K, +}^l  \langle v \rangle^{-2+ \gamma/2})^2   + v \cdot \nabla_x (f_{K, +}^l  \langle v \rangle^{-2+ \gamma/2} )^2   = 2 \tilde{Q}(G, F) \langle v \rangle^{l-4+\gamma} f_{K, +}^l : = (1-\Delta_x-\partial_t^2)^{\sigma/2} (1-\Delta_v)^{\sigma/2 + \tau_1/2} G_K^l,
\]
where we define $G_K^l$ as
\[
G_K^l : = 2(1-\Delta_x-\partial_t^2)^{-\sigma/2} (1-\Delta_v)^{-\sigma/2 - \tau_1/2} ( \tilde{Q}(G, F) \langle v \rangle^{l-4+\gamma} f_{K, +}^l ). 
\]
Choose the parameters in  Lemma  \ref{L219} as
\[
m = \tau_1 + \sigma, \quad \beta \in (0, s), \quad s^b =s'' = \frac {(1-2\sigma) \beta_-} {2(1+\sigma +\tau_1 + \beta) } < \min \{\beta, \frac {(1-\sigma p) \beta_-} {p(1+\sigma +\tau_1 +\beta) }   \}, \quad r = \sigma, \quad \tau_1 > 2, \quad \sigma + \tau_1 <3,
\]
where $ 1 < p < 2$ is chosen to be close enough to 1 such that
\[
1< p < p^c, \quad \sigma p <1, \quad 1 < p <\frac {p} {2-p} <q_*, \quad \sigma p_* = \sigma \frac {p} {p-1} >6,
\]
where $p^c, q_*$ is defined in Lemma \ref{L44} and the last condition guarantees that
\[
H^{-\sigma, p} (\T^3 \times \R^3) \supseteq L^1(\T^3 \times \R^3) \quad \hbox{since} \quad H^{\sigma, p^*} (\T^3 \times \R^3) \subseteq L^\infty(\T^3 \times \R^3).
\]
With such choice, by Lemma \ref{L219} we have
\begin{equation*}
\begin{aligned} 
&\Vert (1-\Delta_x)^{s''/2 } (f_{K, +}^l \langle v \rangle^{-2+\gamma/2})^2\Vert_{L^p_{t, x, v}} 
\\
\lesssim&  \Vert \langle v \rangle^{\gamma} (f_{K, +}^l (T_1)  )^{2}\Vert_{L^p_{x, v}}  + \Vert \langle  v \rangle^4  (1-\Delta_v)^{-\tau_1/2}  (\langle v \rangle^{-2+ \gamma/2} f_{k, +}^l (T_2))^2\Vert_{H^{-\sigma, p}_{x, v}}
\\
& + \Vert \langle v \rangle^{-4+\gamma} (f_{K, +}^l)^2 \Vert_{L^p_{t, x, v}} + \Vert (-\Delta_v)^{\beta/2} (f_{K, +}^l  \langle v \rangle^{-2+ \gamma/2} )^2 \Vert_{L^p_{t, x, v}} + \Vert \langle v \rangle^{1+\sigma + \tau_1} G_K^l   \Vert_{L^p_{t, x, v}}
\\
\lesssim &\Vert \langle v \rangle^{2+\gamma/2} (f_{K, +}^l (T_1)  )\Vert_{L^{2p}_{x, v}}  + \Vert \langle  v \rangle^4  (1-\Delta_v)^{-\tau_1/2}  (\langle v \rangle^{-2+ \gamma/2} f_{k, +}^l (T_2))^2\Vert_{L^1_{x, v}}
\\
& + \Vert (f_{K, +}^l)^2 \Vert_{L^p_{t, x, v}} + \Vert (-\Delta_v)^{\beta/2} (f_{K, +}^l  \langle v \rangle^{-2+\gamma/2} )^2 \Vert_{L^p_{t, x, v}} + \Vert \langle v \rangle^{4} G_K^l   \Vert_{L^p_{t, x, v}}.
\end{aligned}
\end{equation*}
We will estimate the terms in the right hand side separately. For the third term by Lemma \ref{L44} we have
\[
\Vert  (f_{K, +}^l)^2 \Vert_{L^p_{t, x, v}} \lesssim \frac  { \mathcal{E}_{p, s''} (M, T_1, T_2)^{ \frac {r_*} p}} {(K-M)^{\frac {\xi_*-2p} p } }, \quad r_* >p, \quad \xi_* >2p.
\]
For the fourth term, using Lemma \ref{L218} with $p' = \frac {p} {2-p}$ and H\"older's inequality we have
\begin{equation*}
\begin{aligned} 
&\Vert (-\Delta_v)^{\beta/2} (f_{K, +}^l  \langle v \rangle^{-2+\gamma/2} )^2 \Vert_{L^p_{t, x, v}}^p   =\int_{T_1}^{T_2} \int_{\T^3} \Vert (-\Delta_v)^{\beta/2} (f_{K, +}^l  \langle v \rangle^{-2+\gamma/2} )^2 \Vert_{L^p_{v}}^p  dx d\tau
\\
\lesssim& \int_{T_1}^{T_2} \int_{\T^3} \Vert (I -\Delta_v)^{s/2}( f_{K, +}^l  \langle v \rangle^{-2+\gamma/2} )\Vert_{L^2_v}^p \Vert (f_{K, +}^l \langle v \rangle^{-2+\gamma/2} )^2 \Vert_{L^{p'}_v}^\frac p {2} +\Vert  (f_{K, +}^l  \langle v \rangle^{-2+\gamma/2} )^2\Vert_{L^{p}_v}^p dx d\tau
\\
\lesssim & \left(\int_{T_1}^{T_2} \Vert  (I -\Delta_v)^{s/2}( f_{K, +}^l  \langle v \rangle^{\gamma/2} ) \Vert_{L^2_{x, v}}^2 d\tau \right)^\frac p  2 \left(      \int_{T_1}^{T_2} \Vert (f_{K, +}^l )^2\Vert_{L^{p'}_{x, v}}^{p'}   d\tau \right)^{\frac {2-p} 2 } +    \int_{T_1}^{T_2} \Vert( f_{K, +}^l )^2\Vert_{L^{p}_{x, v}}^{p}   d\tau.
\end{aligned}
\end{equation*}
The first term is bounded by 
\[
\int_{T_1}^{T_2} \Vert  (I-\Delta_v)^{s/2}( f_{K, +}^l  \langle v \rangle^{\gamma/2} )\Vert_{L^2_{x, v}}^2 d\tau \lesssim \int_{T_1}^{T_2}  \Vert  f_{K, +}^l \Vert_{L^2_x H^s_{\gamma/2}}^2 d\tau \lesssim \int_{T_1}^{T_2}  \Vert  f_{M, +}^l \Vert_{L^2_x H^s_{\gamma/2}}^2 d\tau  \le \mathcal{E}_{p, s''} (M, T_1, T_2).
\]
Since $1< p < p' = \frac {p} {2-p}< q_*$ , by Lemma \ref{L44} we have
\[
 \int_{T_1}^{T_2} \Vert (f_{K, +}^l )^2\Vert_{L^{p'}_{x, v}}^{p'}   d\tau \lesssim \frac { \mathcal{E}_{p, s''}  (M, T_1, T_2)^{{r_*} } } {(K-M)^{{\xi_* - 2p'}  }   } , \quad   \int_{T_1}^{T_2} \Vert( f_{K, +}^l )^2\Vert_{L^{p}_{x, v}}^{p}  d \tau \lesssim \frac { \mathcal{E}_{p, s''}  (M, T_1, T_2)^{ {r_*}} } {(K-M)^{ {\xi_* - 2p}  }   } .
\]
Gathering all the terms we have
\[
\Vert (-\Delta_v)^{\beta/2} (f_{K, +}^l  \langle v \rangle^{\gamma/2} )^2 \Vert_{L^p_{t, x, v}}  \lesssim \frac { \mathcal{E}_{p, s''}  (M, T_1, T_2)^{\beta_1}} {(K-M)^{a_1 }   }+\frac { \mathcal{E}_{p, s''}  (M, T_1, T_2)^{\beta_2} }{(K-M)^{a_2 } },
\]
where the parameters satisfies
\[
\beta_1 = \frac 1 2 (1+r_*/p')>1, \quad a_1 = (\xi - 2p')/(2p')>0, \quad \beta_2 = r_*/p>1, \quad a_2 = (\xi_* - 2p)/p>0.
\]
Since $H^{-\sigma, p} (\T^3 \times \R^3) \supseteq L^1(\T^3 \times \R^3)$, by Lemma \ref{L43} with $j=4$ we have
\begin{equation*}
\begin{aligned} 
\Vert \langle v \rangle^{4} G_K^l   \Vert_{L^p_{t, x, v}} \le& 2 \left(\int_{T_1}^{T_2} \Vert \langle v \rangle^{4} (1-\Delta_x-\partial_t^2)^{-\sigma/2} (1-\Delta_v)^{-\sigma/2 - \tau_1/2} ( \tilde{Q}(G, F) \langle v \rangle^{l -4 + \gamma} f_{K, +}^l )\Vert_{L^p_{x, v}}^p d\tau  \right)^{1/p}
\\
\lesssim&\int_{T_1}^{T_2} \Vert  (1-\Delta_v)^{ - \tau_1/2} ( \tilde{Q}(G, F) \langle v \rangle^{l} f_{K, +}^l )\Vert_{L^1_{x, v}} d\tau 
\\
\lesssim&\Vert  \langle v \rangle^2 f_{K, +}^l (T_1) \Vert_{L^2_{x, v}}^2  +   \int_{T_1}^{T_2} \Vert f_{K, +}^l\Vert_{L^2_xH^s_{\gamma/2}}^2 d\tau + \int_{T_1}^{T_2}\Vert  \langle  v \rangle^6 f_{K, +}^l \Vert_{L^2_{x, v} }^2 d\tau  + (1 + K) \int_{T_1}^{T_2}\Vert  \langle  v \rangle^6 f_{K, +}^l \Vert_{L^1_{x, v} } d\tau,
\end{aligned}
\end{equation*}
by Lemma \ref{L44} we have
\begin{equation}
\label{estimate f K l E version}
 \int_{T_1}^{T_2}\Vert  \langle  v \rangle^6 f_{K, +}^l \Vert_{L^2_{x, v} }^2 d\tau  \le \frac {2^{2p-2}} {(K-M)^{2p-2}} \Vert  \langle  v \rangle^{\frac {12} {2p} } f_{\frac {K+M} 2, +}^l \Vert_{L^{2p}_{x, v} }^{2p} d\tau \lesssim \frac {\mathcal{E}_{p, s''}  (M, T_1, T_2)^{r_*}} { (K - M )^{\xi_*-2} },
\end{equation}
similarly 
\begin{equation}
\label{estimate f K l E version 2}
 \int_{T_1}^{T_2}\Vert  \langle  v \rangle^6 f_{K, +}^l \Vert_{L^1_{x, v} } d\tau  \le \frac {2^{2p-1}} {(K-M)^{2p-1}} \Vert  \langle  v \rangle^{\frac {6} {2p} } f_{\frac {K+M} 2, +}^l \Vert_{L^{2p}_{x, v} }^{2p} d\tau \lesssim \frac {\mathcal{E}_{p, s''}  (M, T_1, T_2)^{r_*}} { (K - M )^{\xi_* - 1} }.
\end{equation}
Since $\frac {K} {K-M} \ge 1$, we have
\begin{align} 
\label{estimate L p terms}
\nonumber
(1+K) \frac {\mathcal{E}_{p, s''}  (M, T_1, T_2)^{r_*}} { (K - M )^{\xi_* - 1} } \le \frac {\mathcal{E}_{p, s''}  (M, T_1, T_2)^{r_*}} { (K - M )^{\xi_* - 1} } + \frac {K} {K-M} \frac {\mathcal{E}_{p, s''}  (M, T_1, T_2)^{r_*}} { (K - M )^{\xi_* - 2} }
\\
\le \frac {K} {K-M} \left( \frac {\mathcal{E}_{p, s''}  (M, T_1, T_2)^{r_*}} { (K - M )^{\xi_* - 1} }  + \frac {\mathcal{E}_{p, s''}  (M, T_1, T_2)^{r_*}} { (K - M )^{\xi_* - 2} }  \right), \quad \xi_*>2.
\end{align}
Gathering all the terms we have
\begin{equation}
\label{estimate G k l}
\Vert \langle v \rangle^{4} G_K^l   \Vert_{L^p}   \lesssim \Vert  \langle v \rangle^2 f_{K, +}^l (T_1)  \Vert_{L^2_{x, v}}^2  +   \int_{T_1}^{T_2} \Vert f_{K, +}^l\Vert_{L^2_xH^s_{\gamma/2}}^2 d\tau + \frac {K} {K-M} \left( \frac {\mathcal{E}_{p, s''}  (M, T_1, T_2)^{r_*}} { (K - M )^{\xi_* - 1} }  + \frac {\mathcal{E}_{p, s''}  (M, T_1, T_2)^{r_*}} { (K - M )^{\xi_* - 2} }  \right),
\end{equation}
where the constants is independent of $\epsilon \in [0, 1]$. Finally we bound the second term, by Fubini's theorem we have
\[
\Vert (I -\Delta_v)^{-\tau_1/2} (f_{K, +}^l (T_2))^2 \Vert_{L^1_{x, v}}   = \int_{\R^3}  (I -\Delta_v)^{-\tau_1/2} \left( \int_{\T^3}  (f_{K, +}^l (T_2))^2  dx \right) dv.
\]
Integrate \eqref{level set equation} first in $x$ and then in $t, v$ gives
\begin{equation*}
\begin{aligned} 
&\int_{\R^3}  (I -\Delta_v)^{-\tau_1/2} \left( \int_{\T^3}  (f_{K, +}^l (T_2))^2  dx \right) dv
\\
\lesssim &\int_{\T^3} \int_{\R^3}  (I -\Delta_v)^{-\tau_1/2}   (f_{K, +}^l (T_1))^2  dxdv + \int_{T_1}^{T_2} \int_{\T^3} \int_{\R^3}  (I -\Delta_v)^{-\tau/2}  \left(   \tilde{Q}(Q, F)  \langle v \rangle^{l} f_{K, +}^l \right)  dv dx  d\tau
\\ 
\lesssim& \Vert \langle v \rangle^2 f_{K, +}^l (T_1)\Vert_{L^2_{x, v}}^2  +\int_{T_1}^{T_2} \Vert  (1-\Delta_v)^{ - \tau_1/2} ( \tilde{Q}(G, F) \langle v \rangle^{l} f_{K, +}^l )\Vert_{L^1_{x, v}} d\tau,
\end{aligned}
\end{equation*}
the last term can be estimated the same way as \eqref{estimate G k l}. Gathering all the terms we  have
\begin{align}
\label{estimate f K l s''}
\nonumber
\Vert (1-\Delta_x)^{s''/2 } (f_{K, +}^l \langle v \rangle^{-4+ \gamma/2})^2\Vert_{L^p_{t, x, v}} \le& C ( \Vert \langle v \rangle^2 f_{K, +}^l (T_1)\Vert_{L^2_{x, v}}^2 +  \Vert \langle v \rangle^2 f_{K, +}^l (T_1)\Vert_{L^{2p}_{x, v}}^2)
\\
&+C_l \int_{T_1}^{T_2} \Vert f_{K, +}^l\Vert_{L^2_xH^s_{\gamma/2}}^2 d\tau  + C \frac {K} {K-M}  \sum_{i=1}^4\frac { \mathcal{E}_{p, s''}  (M, T_1, T_2)^{\beta_i}} {(K-M)^{a_i }   },
\end{align}
where 
\[
\beta_1 = \frac 1 2 (1+r_*/p'),  \quad \beta_2 = r_*/p, \quad \beta_3 = \beta_4 = r_*
,\quad a_1 = (\xi - 2p')/(2p'),\quad a_2 = (\xi_* - 2p)/p, \quad a_3 =\xi_*-1, \quad a_4 = \xi_*-2,
\]
it's easily seen that $\beta_i >1, a_i >0$ for all $i=1, 2, 3, 4$. For any $\epsilon \in [0, 1]$, by Lemma \ref{L41} and \eqref{assumption on g 1} we have
\begin{equation*}
\begin{aligned}
\frac {d} {dt} \frac 1 2 \Vert  f_{K, +}^l \Vert_{L^2_{x, v}}^2 = &(Q(\mu +g, \mu+f ), f_{K, +}^l \langle v \rangle^{l})_{L^2_{x, v}} +  \epsilon (L_{\alpha}(\mu +f) ,f_{K, +}^l \langle v \rangle^{l})_{L^2_{x, v}}
\\
\le &-(2c_1 -C_l \sup_x \Vert  g \Vert_{L^\infty_{|\gamma|+9}} )\Vert f_{K, +}^l \Vert_{L^2_x H^s_{\gamma/2}}^2 + C_{l}(1 + \sup_x \Vert g \Vert_{L^\infty_l}  ) \Vert f_{K,+}^l \Vert_{L^2_x L^2_2}^2 
\\
&+ C_l (1 + K)(1 + \sup_x \Vert g \Vert_{L^\infty_l})  \Vert f_{K, +}^l \Vert_{L^1_x L^1_2} + \epsilon C_{l } \Vert  f_{K, +}^l \Vert_{L^2_{x, v}}^2 + \epsilon C_{l} (1 + K)  \Vert f_{K, +}^l \Vert_{L^1_{x, v}}
\\
\le &-c_1 \Vert f_{K, +}^l \Vert_{L^2_x H^s_{\gamma/2}}^2 + C_{l}\Vert f_{K,+}^l \Vert_{L^2_x L^2_2}^2 + C_l (1 + K)  \Vert f_{K, +}^l \Vert_{L^1_x L^1_2} .
\end{aligned}
\end{equation*}
Integrate in $[T_1, T_2]\times \T^3 \times \R^3$ and similarly as \eqref{estimate f K l E version}, \eqref{estimate f K l E version 2} and \eqref{estimate L p terms} we have
\begin{align} 
\label{estimate f K l v}
& \Vert f_{K, +}^l(T_2) \Vert_{L^2_{x, v}}^2 + c_1 \int_{T_1}^{T_2} \Vert f_{K, +}^l(\tau) \Vert_{L^2_x H^s_{\gamma/2}}^2  d\tau
\\
\le& \Vert f_{K, +}^l(T_1) \Vert_{L^2_{x, v}}^2 +  C_{l}  \int_{T_1}^{T_2}  \Vert  \langle v \rangle^2 f_{K,+}^l (\tau)\Vert_{L^2_{x, v}}^2  d\tau
+ C_l (1 + K) \int_{T_1}^{T_2}  \Vert  \langle v \rangle^2 f_{K,+}^l (\tau)\Vert_{L^1_{x, v}}  d\tau
\\
\le& \Vert f_{K, +}^l(T_1) \Vert_{L^2_{x, v}}^2 + C \frac {K} {K-M} \left( \frac {\mathcal{E}_{p, s''}  (M, T_1, T_2)^{r_*}} { (K - M )^{\xi_* - 1} }  + \frac {\mathcal{E}_{p, s''}  (M, T_1, T_2)^{r_*}} { (K - M )^{\xi_* - 2} }  \right).
\end{align}
So the proof is ended by adding the two terms \eqref{estimate f K l s''} and \eqref{estimate f K l v},  then  taking  $C_0$ large such that
\[
\frac {C_l } {C_0} \le \frac {c_1} 2.
\]
Since the estimates are all independent of $\epsilon$, the result is independent of $\epsilon$.  Since$ (-f)_{K, +}^l$ satisfies the same estimate as $f_{K, +}^l$, the proof is the same and thus omitted.
\end{proof}

Before proving the $L^\infty$ bound, we need a $L^2$ bound on the zeroth level energy $\mathcal{E}_{0, p, s''} $ defined by
\begin{align}
\label{definition E 0 p s''}
\mathcal{E}_{0, p, s''} = \mathcal{E}_{p, s''}(0, 0 , T) :=& \sup_{t \in [0, T] } \Vert f_{ +}^l \Vert_{L^2_{x, v}}^2   + \int_{0}^{T} \int_{\T^3} \Vert \langle  v \rangle^{\gamma/2} f_{+}^l \Vert_{H^s_v}^2 dx d\tau 
\\
&+\frac 1 {C_0} \left( \int_{0}^T \Vert (1-\Delta_x)^{s''/2}  (  \langle v \rangle^{-2+\gamma/2} f_+^l)^2\Vert_{L^p_{x, v}}^p d\tau \right)^{1/p}.
\end{align}
where $f_+$ denotes the positive part of $f$ and 
$f_+^l : = \langle v \rangle^l f_+$.

\begin{lem}\label{L46} Suppose $G = \mu+g, F =\mu +f$ smooth and $\gamma \in (-3, 0], s \in (0, 1), \gamma +2s >-1$. Let $T>0, \alpha \ge 0$ be fixed.  Assume that $f$ is a solution to \eqref{linearized equation}. Suppose in addition G satisfies that
\[
G \ge 0, \quad \inf_{t, x} \Vert G \Vert_{L^1_v} \ge A >0, \quad \sup_{t, x} (\Vert G \Vert_{L^1_2}  +\Vert G \Vert_{L \log L} ) < B < +\infty,
\]
for some constant $A, B>0$.  Then for any $\epsilon \in [0, 1] $,  $12 \le l \le k_0 -8, l \ge 3+2\alpha $, suppose 
\begin{equation}
\label{assumption on g 2}
\sup_{t, x} \Vert  g  \Vert_{L^\infty_{9+|\gamma|}}  \le \delta_0, \quad \sup_{t, x} \Vert g \Vert_{L^\infty_{k_0}}  \le C,
\end{equation}
for some constant $\delta_0 >0$ small.  Then for any $0 < s' <\frac {s} {2(s+3)}$, there exist $s" \in ( 0, s' \frac  4  {2l+r})$ and $p^b := p^b(l, \gamma, s, s') >1$ such that for any $1< p <p^b$, we have
\[
\mathcal{E}_{0, p, s''} \le C_l e^{C_l T} \max_{ j \in \{ 1/p, p'/p  \} } \left( \Vert \langle v \rangle^l f_0\Vert_{L^2_{x, v}}^{2j} + \sup_{t, x} \Vert g\Vert_{L^\infty_{k_0}}^{2j} T^j      +\epsilon^{2j} T^j  \right), \quad p' = p /(2- p). 
\]
The same estimate holds for $(-f)_{+}^l$ and its associated $\mathcal{E}_{0, p, s''} $.
\end{lem}
\begin{proof} By \eqref{assumption on g 2} we have $0 \le C(g) \le 1+C$, taking supremum on $[0, T]$ in Lemma \ref{L39} we have
\begin{equation*}
\begin{aligned}
&\sup_{[0, T]}\Vert f_+^l(t)\Vert_{L^2_{x, v}}^2 + c_0\int_0^t \Vert \langle v \rangle^{\gamma/2} f_+^l (\tau)\Vert_{L^2_xH^s_{v}}^2 d\tau
\\
\le& \sup_{[0, T]}\Vert \langle v \rangle^l f(t)\Vert_{L^2_{x, v}}^2 + c_0\int_0^t \Vert \langle v \rangle^l f (\tau)\Vert_{L^2_xH^s_{\gamma/2}}^2 d\tau 
\\
\le& C_l e^{C_l T} (\Vert \langle v \rangle^l f_0\Vert_{L^2_{x, v}}^2  +\sup_{t, x}\Vert g\Vert_{L^\infty_{k_0}}^2 T +\epsilon^2 T ) :=C_l  e^{C_l T}  D.
\end{aligned}
\end{equation*}
Let's concentrate on the last term, for $p = (1,2) , 0< s" < \beta<  s'$, using Lemma \ref{L218} with $p' = \frac {p} {2-p}$ we have
\begin{equation*}
\begin{aligned} 
&\int_{0}^{T} \int_{\R^3} \Vert (-\Delta_x)^{s''/2} (f_{+}^l  \langle v \rangle^{-2+\gamma/2} )^2 \Vert_{L^p_{x}}^p  dv d\tau
\\
\lesssim& \int_{0}^{T} \int_{\R^3} \Vert (I-\Delta_x)^{\beta/2}( f_{ +}^l  \langle v \rangle^{-2+\gamma/2} )\Vert_{L^2_x}^p \Vert (f_{+}^l \langle v \rangle^{-2+\gamma/2} )^2 \Vert_{L^{p'}_x}^\frac p {2} +\Vert  (f_{ +}^l  \langle v \rangle^{-2+\gamma/2} )^2\Vert_{L^{p}_x}^p dv d\tau
\\
\lesssim & \left(\int_{0}^{T} \Vert  (I-\Delta_x)^{\beta/2}( f_{+}^l  \langle v \rangle^{-2+ \gamma/2} )\Vert_{L^2_{x, v}}^2 d\tau \right)^\frac p  2 \left(      \int_{0}^{T} \Vert (f_{ +}^l \langle v \rangle^{-2+\gamma/2} )^2\Vert_{L^{p'}_{x, v}}^{p'}   d\tau \right)^{\frac {2-p} 2 } +    \int_{0}^{T} \Vert( f_{+}^l \langle v \rangle^{-2+\gamma/2} )^2\Vert_{L^{p}_{x, v}}^{p}   d\tau
\\
\lesssim & \int_{0}^{T} \Vert  (I-\Delta_x)^{\beta/2}( f_{+}^l  \langle v \rangle^{-2+ \gamma/2} )\Vert_{L^2_{x, v}}^2 d\tau +       \int_{0}^{T} \Vert (f_{+}^l \langle v \rangle^{-2+\gamma/2} )^2\Vert_{L^{p'}_{x, v}}^{p'}   d\tau +    \int_{0}^{T} \Vert( f_{+}^l \langle v \rangle^{-2+\gamma/2} )^2\Vert_{L^{p}_{x, v}}^{p}   d\tau.
\end{aligned}
\end{equation*}
The control of the $L^{2p}$ and $L^{2p'}$ norms of $f_+^l$  are both through interpolations. First we have
\[
\Vert f_{+}^l \langle v \rangle^{-2+ \gamma/2} \Vert_{L^{2p}_{x, v}} \le \Vert f_{+}^l \langle v \rangle^{-2+ \gamma/2}\Vert_{L^{2}_{x, v}}^{1-\beta_p} \Vert f_{+}^l  \langle v \rangle^{-2+ \gamma/2} \Vert_{L^{\xi(p)}_{x, v}}^{\beta_p} , \quad \xi(p) = \frac 2 {2-p } >2, \quad \beta_p =\frac 1 p. 
\]
For any $\beta>0$, by choosing $\xi(p) = r(s, \beta, 3)$ in Lemma \ref{L216} we have
\[
\Vert f_{+}^l \langle v \rangle^{-2+ \gamma/2} \Vert_{L^{\xi(p)}_{x, v}} \lesssim \left(\int_{\T^3} \Vert f_+^l(x, \cdot) \langle v \rangle^{-2+ \gamma/2} \Vert_{H^s_v}^2dx  \right)^{1/2} +  \left(\int_{\R^3} \Vert f_+^l( \cdot, v) \langle v \rangle^{-2+ \gamma/2} \Vert_{H^\beta_x}^2dv \right)^{1/2}.
\]
Consequently, we have
\begin{equation}
\label{estimate f l 2p}
\Vert f_{+}^l \langle v \rangle^{-2+ \gamma/2}\Vert_{L^{2p}_{x, v}}^{2p} \le \Vert f_{+}^l \langle v \rangle^{-2+ \gamma/2}\Vert_{L^{2}_{x, v}}^{2(p-1)} (\Vert (I-\Delta_v)^{s/2} ( f_+^l \langle v \rangle^{-2+ \gamma/2}) \Vert_{L^2_{x, v}}^2 + \Vert (I-\Delta_x)^{\beta/2}  f_+^l (\langle v \rangle^{-2+ \gamma/2} )\Vert_{L^2_{x, v}}^2 ).
\end{equation}
For any $q  >1 $ we have
\begin{equation*}
\begin{aligned}
 \Vert (I-\Delta_x)^{\beta/2}  f_+^l (\langle v \rangle^{-2+ \gamma/2} )\Vert_{L^2_{x, v}}^2 =& \int_{\R^3} \sum_{\eta \in \mathbb{Z}^3} \langle v \rangle^{2l-4 +\gamma} \langle \eta \rangle^{2\beta}  |\mathcal{F}_x (f_+)|^2 dv
 \\
\le & \int_{\R^3} \sum_{\eta \in \mathbb{Z}^3} \frac 1 q \langle v \rangle^{(2l-4 +\gamma)q}  + (1  - \frac 1 q) \langle \eta \rangle^{2\beta\frac {q} {q-1} }   |\mathcal{F}_x (f_+)|^2 dv.
\end{aligned}
\end{equation*}
Taking $q = \frac {2l+\gamma  } {2l+\gamma-4}$, if we assume $2 \beta \frac  {q} {q-1} < 2s'$ or equivalently $\beta < s'(1-\frac 1 q)  = s' \frac {4} {2l+\gamma} $, then we have
\[
 \Vert (I-\Delta_x)^{\beta/2} ( f_+^l \langle v \rangle^{-2+ \gamma/2} )\Vert_{L^2_{x, v}}^2 \le C_{l, \gamma} ( \Vert  f_+^l \langle v \rangle^{ \gamma/2} \Vert_{L^2_{x, v}}^2 + \Vert (I-\Delta_x)^{s'/2}  f_+ \Vert_{L^2_{x, v}}^2 ).
\]
Integrate in $[0, T]$ together with Lemma \ref{L39} we have
\[
 \int_{0}^T \Vert (I-\Delta_x)^{\beta/2} ( f_+^l \langle v \rangle^{-2+ \gamma/2} )\Vert_{L^2_{x, v}}^2  d\tau \le C_{l, \gamma}\int_0^T ( \Vert  f_+^l \langle v \rangle^{ \gamma/2} \Vert_{L^2_{x, v}}^2 + \Vert (I-\Delta_x)^{s'/2}  f_+ \Vert_{L^2_{x, v}}^2 ) d\tau \le C D.
\]
Since $\xi$ is an increasing function in $\beta$, we obtain the corresponding range for $\xi(p)$ and for $p$ as
\[
\xi(p) \in (2, r(s, s'\frac {4} {2l+\gamma}, 3) ) : = (2, r^b), \quad p \in (1, 2-2/r^b) : = (1, p^b),
\]
where $r(\cdot, \cdot, \cdot)$ is defined in Lemma \ref{L216}. It is clear by its definition that $p^b$ depends on $l,\gamma, s, s'$. Using such parameters and combining \eqref{estimate f l 2p}, we obtain that
\[
\Vert f_{+}^l \langle v \rangle^{-2+ \gamma/2}\Vert_{L^{2p}_{x, v}}^{2p} \le C\Vert f_{+}^l \Vert_{L^{2}_{x, v}}^{2(p-1)} (\Vert (I-\Delta_v)^{s/2} ( f_+^l \langle v \rangle^{\gamma/2}) \Vert_{L^2_{x, v}}^2 + \Vert (I-\Delta_x)^{s'/2}  f_+ \Vert_{L^2_{x, v}}^2 ),
\]
by Lemma  \ref{L39}  we have
\begin{equation}
\label{estimate f K}
\Vert f_{+}^l \langle v \rangle^{-2+ \gamma/2}\Vert_{L^{2p}_{x, v}}^{2p} \le C D^p, \quad p \in(1, p^b).
\end{equation}
Taking $p'$ close to 1 such that $p' \in (1, p^b)$, \eqref{estimate f K} still holds when $p$ is replaced by $p'$, which means
\[
\Vert f_{+}^l \langle v \rangle^{-2+ \gamma/2}\Vert_{L^{2p'}_{x, v}}^{2p'} \le C D^{p'}.
\]
The same estimate holds for $(-f)_+^l$ and its associated $E_0$ since Lemma \ref{L39} applies to the absolute value of $f$ which contains both positive and negative parts of $f$. So the theorem is thus proved. 
\end{proof}

\begin{thm}\label{T47}
 Suppose $G = \mu+g, F =\mu +f$ smooth and $\gamma \in (-3, 0], s \in (0, 1), \gamma +2s >-1$. Let $T>0, \alpha \ge 0$ be fixed.  Assume that $f$ is a solution to \eqref{linearized equation}. Suppose in addition G satisfies that
\[
G \ge 0, \quad \inf_{t, x} \Vert G \Vert_{L^1_v} \ge A >0, \quad \sup_{t, x} (\Vert G \Vert_{L^1_2}  +\Vert G \Vert_{L \log L} ) < B < +\infty,
\]
for some constant $A, B>0$.  Then for any $\epsilon \in [0, 1] $,  $12 \le l \le k_0 -8, l \ge 3+2\alpha $, suppose 
\[
\sup_{t, x} \Vert  g  \Vert_{L^\infty_{9+|\gamma|}}  \le \delta_0, \quad \sup_{t, x} \Vert g \Vert_{L^\infty_{k_0}}  \le C,
\]
for some constant $\delta_0 >0$ small. Assume that the initial data satisfies
\[
\Vert \langle  v \rangle^{l} f_0 \Vert_{L^2_{x, v}}  < + \infty, \quad\Vert \langle  v \rangle^{l} f_0 \Vert_{L^\infty_{x, v}}  <+ \infty. 
\]
Then there exist a constant $l_0 \ge 0$ which depends on $s, \gamma , l$ such that if we additionally suppose that
\[
\sup_{t}  \Vert \langle  v \rangle^{l+l_0} f_0  \Vert_{L^2_{x, v} } \le C_1 < +\infty,
\]
for some constant $C_1>0$.  Then it follows that
\[
\sup_{t \in [0, T]}\Vert  \langle v \rangle^l f\Vert_{L^\infty_{x, v}}  \le \max \{ 2\Vert  \langle v \rangle^l  f_0 \Vert_{L^\infty_{x, v}} , K_0^1 \},
\]
where 
\[
K_0^1 := C_l e^{C_l T} \max_{ 1 \le i \le 4} \max_{j \in \{ 1/p, p'/p \} }  \Vert \langle v \rangle^l f_0\Vert_{L^2_{x, v}}^{2j} + \sup_{t, x} \Vert g\Vert_{L^\infty_{k_0}}^{2j} T^j   + \epsilon^{2j} T^j )^{\frac {\beta_i -1} {a_i}} , \quad p' = \frac p {2-p} ,
\]
and $a_i, \beta_i$ is defined in Lemma \ref{L46}. 
\end{thm}

\begin{proof}
Choose $(p, s'')$ close enough to $(1, 0)$ such that
\[
s'' =\frac {1} {(2l+4)(s+3)}, \quad  p =\frac {\min \{p^a, p^b \}} 2 +\frac 1 2,
\]
thus $(p, s'')$ satisfies
\[
\quad 0< s''< s' \frac {4 } {2l+4} ,\quad0< s'< \frac {s} {2(s+3)} , \quad1<  p < \min \{p^a, p^b \},
\]
where $p^a$ and $p^b$ is defined in Lemma \ref{L45} and Lemma \ref{L46} respectively. Such choice $p, s''$ guarantees Lemma \ref{L45} and \ref{L46} hold. We use a classical iteration scheme to prove the estimate for the $L^\infty$ norm for the solution. Fix $K_0$ to be determined, we introduce an increasing level set $M_k$ as
\[
M_k : = K_0 (1-\frac 1 {2^k}),  \quad k =0, 1, 2....
\]
Take $T_2 \in [0, T]$ with $T>0$ fixed in the analysis. In order to simplify the notation, denote 
\[
f_k := f_{M_k, +}^l, \quad  \mathcal{E}_k := \mathcal{E}_{p, s''}(M_k, 0, T), \quad k =0, 1, 2...,
\]
choose $M= M_{k-1} <M_k =K$ and $T_1 = 0$ in Lemma \ref{L45} and by
\[
\mathcal{E}_{p, s''} (M_{k-1},  0, T_2)  \le \mathcal{E}_{p, s''} (M_{k-1},  0, T)  = \mathcal{E}_{k-1} ,  \quad k = 1, 2...,
\]
together with  Lemma \ref{L45}, for any $0 \le T_2\le T$ we have
\begin{equation*}
\begin{aligned} 
&\Vert f_{k} (T_2) \Vert_{L^2_{x, v}}^2 + \int_{0}^{T_2} \Vert \langle v \rangle^{\gamma/2} (1-\Delta_v)^{s/2}  f_{k} \Vert_{L^2_{x, v}}^2 d \tau +\frac 1 {C_0} \left(       \int_{0}^{T_2} \Vert (1-\Delta_x)^{{s''}/2}  (\langle v \rangle^{-4+ \gamma} (f_{k} )^2) \Vert_{L^p_{x, v}}^p d \tau                                 \right)^{\frac 1 p}
\\
\le& C \Vert \langle v \rangle^2 f_{k} (0) \Vert_{L^2_{x, v}}^2 +  C \Vert \langle v \rangle^2 f_{k} (0)  \Vert_{L^{2p}_{x, v}}^2 +\frac {CK} {K - M} \sum_{i=1}^4 \frac {\mathcal{E}_{p, s''}  (M, 0, T_2)^{\beta_i}} {(K-M)^{a_i}}
\\
\le& C \Vert \langle v \rangle^2 f_{k} (0) \Vert_{L^2_{x, v}}^2 +  C \Vert \langle v \rangle^2 f_{k} (0)  \Vert_{L^{2p}_{x, v}}^2 +C \sum_{i=1}^4 \frac {2^{k(a_i +1)}   \mathcal{E}_{k-1}^{\beta_i }  } { K_0^{a_i} },
\end{aligned}
\end{equation*}
where in the last line  we have used $K-M = M_k -M_{k-1} = K_0(\frac  1 {2^{k-1}} -\frac  1 {2^k}) =  K_0\frac  1 {2^k},  K \le K_0$. Taking supremum in $T_2 \in [0, T]$ we deduce
\[
\mathcal{E}_k \le C \Vert \langle v \rangle^2 f_k (0) \Vert_{L^2_{x, v}}^2 +  C \Vert \langle v \rangle^2 f_k (0)  \Vert_{L^{2p}_{x, v}}^2 +C \sum_{i=1}^4 \frac {2^{k(a_i +1)}   \mathcal{E}_{k-1}^{\beta_i }  } { K_0^{a_i} }.
\]
By taking $K_0 \ge 2 \Vert \langle v \rangle^{l} f_0\Vert_{L^\infty_{x, v}} $, we have $f_{k }(0)  = 0, k =1, 2, ...$,  hence
\[
\mathcal{E}_k \le C \sum_{i=1}^4 \frac {2^{k(a_i +1)}   \mathcal{E}_{k-1}^{\beta_i }  } { K_0^{a_i} }.
\]
Let
\[
Q_0 = \max_{1 \le i \le 4}( 2^{\frac {a_i +1} {\beta_i -1} } ), \quad \mathcal{E}_K^* = \mathcal{E}_0 (1/Q_0)^k, \quad k=0, 1, 2, ...,
\]
since $\beta_i>1, \alpha_i >0$ we have $Q_0 >1$.  Suppose
\[
K_0 \ge K_0(\mathcal{E}_0) := \max_{1 \le i \le 4 } \{ 4^{\frac 1 {a_i}}  C^{\frac 1 {a_i}}  \mathcal{E}_0^{\frac {\beta_i -1} {a_i}  }Q_0^{\frac {\beta_i} {a_i}} \}. 
\]
We easily compute
\[
\mathcal{E}_0^{*}  =\mathcal{E}_0, \quad  C \sum_{i=1}^4 \frac {2^{k(a_i +1)}   (\mathcal{E}_{k-1}^*)^{\beta_i}   } { K_0^{a_i} }  \le  \sum_{i=1}^4 \frac {2^{k(a_i +1)}   \mathcal{E}_0^{\beta_i} Q_0^{-\beta_i(k-1)}   } {4 \mathcal{E}_0^{\beta_i-1}Q_0^{\beta_i} }  \le \sum_{i=1}^4 \frac {Q_0^{k(\beta_i -1)}  \mathcal{E}_0   } {4 Q_0^{\beta_i k} }  \le   \mathcal{E}_0 (1/Q_0)^k =  \mathcal{E}_k^*.
\]
by comparison principle we have $\mathcal{E}_k \le \mathcal{E}_k^* \to 0$, as $k \to \infty$. In particular we  deduce 
\[
\sup_{t \in [0, T]} \Vert f_{K_0, +}^l (t, \cdot ,\cdot)\Vert_{L^2_{x, v}} = 0, \quad \hbox{for} \quad K_0 = \max \{ 2 \Vert \langle v \rangle^l f_0  \Vert_{L^\infty_{x, v}} , K_0 (\mathcal{E}_0)    \}.
\]
which implies that
\[
\sup_{t \in [0, T]} \Vert \langle v \rangle^l f_{ +} (t, \cdot ,\cdot)\Vert_{L^\infty_{x, v}} \le K_0,
\]
By Lemma \ref{L46}  we have
\[
K_0(\mathcal{E}_0)  \le 
 C_l e^{C_l T} \max_{ 1 \le i \le 4} \max_{j \in \{ 1/p, p'/p \} } ( \Vert \langle v \rangle^l f_0\Vert_{L^2_{x, v}}^{2j} + \sup_{t, x} \Vert g\Vert_{L^\infty_{k_0}}^{2j} T^j + \epsilon^{2j} T^j  )^{\frac {\beta_i -1} {a_i}}  := K_0^1  , \quad p' = \frac p {2-p}. 
\]
A similar bound is also valid for $-f$ since Lemma \ref{L45} and Lemma \ref{L46} have their corresponding $-f$ version. The proof is thus ended. \end{proof}

\begin{rmk}
For any $s, \gamma $ given, since $l_0$ depends on $l$, in the following we will use $l_0(l)$ to denote its dependence.
\end{rmk}

\section{Local existence for the linearized  equation} \label{section 5}
In this section we establish the local existence  of a modified linearized Boltzmann equation. The ambient space for contraction and  the subset $H_k$ is defined by
\[
X_k : = L^\infty(0, T, L^2_xL^2_k(\T^3 \times \R^3) ), \quad H_k := \{ g \in X_k| \mu+g \ge 0\},
\]
for some constant $k \ge 0$. The precise equation consider in this section is
\begin{equation}
\label{modified Boltzmann equation g}
\partial_t f +v \cdot \nabla_x f = \epsilon L_\alpha(\mu + f) + Q(\mu + g \chi(\langle v \rangle^{k_0}  g), f) + Q(g \chi(\langle v \rangle^{k_0} g), \mu ), \quad f(0, x, v) =f_0(x, v)
\end{equation}
where we recall $\chi \in C^\infty$ is defined in \eqref{cutoff function}  satisfies $0 \le  \chi \le 1$, $|\nabla\chi| \le 4/\delta_0$  and 
\begin{align*}
\mathcal \chi(x):=\left\{
\begin{aligned}
  & 1,  \quad |x| \le \delta_0\\
 & 0, \quad   |x| \ge 2\delta_0.
\end{aligned}
\right.
\end{align*} Note that since $g \in H_k$, we have $\mu + g \chi(\langle v \rangle^{k_0} g )  \ge 0$. Before going to the proof of the local existence, we first prove a lemma on the cutoff function $\chi$. 

\begin{lem}\label{L51}
For any smooth function $g, h$, for any constant $k_0 \ge 0$, there exist a constant $C$ independent of $\delta_0, k_0, g, h$ such that
\[
|g \chi(\langle v \rangle^{k_0} g) - h \chi(\langle v \rangle^{k_0} h)| \le C| g-h|.
\]
\end{lem}
\begin{proof}
Denote 
\[
p := \langle v \rangle^{k_0} g,  \quad  q := \langle v \rangle^{k_0} h.
\]
we only  need to prove that
\[
|p\chi(p)- q\chi(q) | \le C|p-q| , \quad \forall p, q \in \R.
\]
We split into three cases. If $|p| \ge 3\delta_0$ and $|q|  \ge 2 \delta_0$ we have
\[
|p\chi(p)- q\chi(q) |  = 0 \le |p-q|.
\]
If $|p| \ge 3\delta_0$ and $|q|  \le 2 \delta_0$ we have
\[
|p\chi(p)- q\chi(q) | = |q\chi(q)| \le |q| \le 2\delta_0 \le 2|p-q|. 
\]
It remains to prove the case $|p| \le 3\delta_0$. Since $|\chi| \le 1$ we have
\[
|p\chi(p)- q\chi(q) | =  | p(\chi(p ) -\chi(q) ) + \chi(q) (p-q)| \le |p| |\chi(p)-\chi(q)| + |\chi(q)| |p-q| \le  |p| |\chi(p)-\chi(q)| + |p-q|,
\]
since $|\nabla\chi| \le 4/\delta_0$ we deduce 
\[
|p| |\chi(p)-\chi(q)| + |p-q| \le 12|p-q| +|p-q| \le 13|p-q|. 
\]
Gathering all the cases the lemma is thus proved. 
\end{proof}

The main theorem for the linearized equation is 
\begin{lem}\label{L52}
Suppose $g, f$ smooth and $\gamma \in (-3, 0], s \in (0, 1), \gamma +2s >-1$. Assume that $f$ is a solution to \eqref{modified Boltzmann equation g}. Let $g \in H_k$ and let $\chi$ be the cutoff function defined in \eqref{cutoff function}. \\
(1) Let $T \ge 0$ be arbitrary but fixed. Suppose the initial data $f_0 \in H_k$ and assume that $k_0 \ge k+ 8   $ and $k \ge 12$, suppose that $\delta_0$ is small enough, the solution \eqref{modified Boltzmann equation g} has a unique solution $f \in H_k$.\\
(2) In addition, if we assume further that $k \ge l_0(14) + 14, \alpha = 5$, where $l_0$ is defined in Theorem \ref{T47}.  Then there exist constants $\delta_1,\epsilon_*>0$ and $T_{\delta_0} \in [0, 1]$ which is independent of $\epsilon$ such that if the initial data satisfies
\[
\Vert\langle v \rangle^{14}  f_0  \Vert_{L^\infty_{x, v} \cap L^2_{x, v}} \le \delta_1,
\]
then for any $ 0 <\epsilon \le \epsilon_*, T \in [0, T_{\delta_0} ] $, the solution obtained in (1) satisfies 
\[
\Vert\langle v \rangle^{14}  f  \Vert_{L^\infty ([0, T] \times \T^3 \times \R^3)} \le \delta_0.
\]
The choice of $\epsilon_*, \delta_*$ only depends on $r, s, k_0, \delta_0$.
\end{lem}

\begin{proof}
 Denote
\[
T h := -\partial_t h - v \cdot \nabla_x h -\epsilon L_\alpha h  - Q( \mu + g\chi (\langle v \rangle^{k_0} g ), h ).
\]
Let $S$ be the space of test function given by
\[
S = C^\infty_0( (-\infty, T], C^\infty(\T^3, C_c^\infty (\R^3 ) ) ).
\]
Let $h \in S$,  multiply $Th$ by $h \langle v \rangle^{2k}$ and integrate in $x, v$ we have
\begin{equation*}
\begin{aligned}
(Th,  h)_{L^2_xL^2_k} =& -\frac 1 2 \frac d {dt} \Vert h \Vert_{L^2_xL^2_k} + \epsilon \int_{\T^3} \int_{\R^3}\langle  v \rangle^{2k} (\langle v \rangle^{2\alpha} h - \nabla_v   \langle v \rangle^{2\alpha} \cdot \nabla_v  h )h dv dx 
\\
&  - \int_{\T^3} \int_{\R^3} Q(\mu + g \chi(\langle v \rangle^{k_0} g ), h ) h \langle v \rangle^{2k} dv  dx .
\end{aligned}
\end{equation*}
First we easily compute
\[
\epsilon \int_{\T^3} \int_{\R^3}\langle  v \rangle^{2k} (\langle v \rangle^{2\alpha} h - \nabla_v  \langle v \rangle^{2\alpha}  \cdot \nabla_v h )h dv dx \ge \frac {\epsilon} 2 \Vert f \Vert_{L^2_{x}L^2_{k+\alpha} }^2 + \frac {\epsilon}  2  \Vert \langle v \rangle^{\alpha+ k }  \nabla_v h \Vert_{L^2_{x, v} }^2 dx dv  -C_k \epsilon \Vert h \Vert_{L^2_x L^2_k}^2.
\]
Denote 
\[
T_0 = \int_{\T^3} \int_{\R^3} Q(\mu + g \chi(\langle v \rangle^{k_0} g ), h ) h \langle v \rangle^{2k}dv dx .
\]
It's easily seen that if $\delta_0>0$ is small, $G= \mu +g\chi(\langle v \rangle^{k_0} g)$ satisfies 
\[
G \ge 0, \quad \inf_{t, x} \Vert G \Vert_{L^1_v} \ge A >0, \quad \sup_{t, x} (\Vert G \Vert_{L^1_2}  +\Vert G \Vert_{L \log L} ) < B < +\infty,
\]
for some constant $A, B>0$. If $\delta_0 > 0$ is small enough, by Theorem \ref{T34} we have
\begin{equation*}
\begin{aligned}
(Q(G, f), f \langle v \rangle^{2k} )_{L^2_{x, v}}\le& - \gamma_2 \Vert f \Vert_{H^s_{k+\gamma/2}}^2  -\gamma_1 \Vert f \Vert_{L^2_{k+\gamma/2}}^2 + C_{k}  \Vert f \Vert_{L^2_{k+\gamma/2 -s' }}^2
\\
& + C_k \sup_x \Vert g\chi \Vert_{L^\infty_{ k + 5 + |\gamma|}}  \Vert f \Vert_{L^2_{ k + \gamma/2}}^2   +C_k\Vert g\chi \Vert_{L^2_{|\gamma| +7 }} \Vert f \Vert_{H^s_{ k + \gamma/2}}^2
\\
\le  & - (\gamma_2  - C_k \sup_x \Vert g \chi \Vert_{L^\infty_{k_0}}) \Vert f \Vert_{L^2_x H^s_{ k + \gamma/2}}^2 -(\gamma_1 - C_k \sup_x  \Vert g \chi \Vert_{L^\infty_{k_0}} )  \Vert f \Vert_{L^2_x L^2_{ k + \gamma/2}}^2+ C_{k}  \Vert f \Vert_{L^2_x L^2_{k + \gamma/2 - s' }}^2
\\ 
\le  &  C_{k}  \Vert f \Vert_{L^2_x L^2_{k + \gamma/2 - s' }}^2,
\end{aligned}
\end{equation*}
for some constant $\gamma_1,\gamma_2, C_k>0$. Gathering the terms we easily deduce that
\[
(Th, h)_{L^2_xL^2_k}  \ge -\frac 1 2 \frac {d} {dt} \Vert  h \Vert_{L^2_x L^2_k}^2 - C_k \Vert f \Vert_{L^2_x L^2_k}^2.
\]
By Gr\"onwall's inequality we have
\[
\int_t^T e^{ 2C_k \tau } (Th, h)_{L^2_xL^2_k}  d \tau \ge \frac 1 2 e^{2C_{k} t} \Vert h(t, \cdot, \cdot)  \Vert_{L^2_xL^2_k}^2, \quad \forall t \in [0, T],
\]
and we easily have
\[
\int_t^T e^{ 2 C_k \tau} (Th, h)_ {L^2_xL^2_k}  d \tau \le \sup_{t \in [0, T]} \Vert \langle v \rangle^k h \Vert_{L^2_{x, v}} \int_{0}^T e^{2 C_k \tau } \Vert Th \Vert_{L^2_x L^2_k} d \tau, \quad \forall t \in [0, T],
\]
which implies
\[
\sup_{t \in [0, T]} \Vert \langle v \rangle^k h \Vert_{L^2_{x, v}}  \le C_{T, k} \int_{0}^T \Vert Th \Vert_{L^2_x L^2_k} d \tau, \quad \forall t \in [0, T].
\]
Denote 
\[
W := TS = \{ w| w=Th, h \in S  \}, \quad Y = L^1([0, T] , L^2_xL^2_k(\T^3 \times \R^3)), \quad X= Y^* = L^\infty([0, T] , L^2_xL^2_k(\T^3 \times \R^3)),
\]
where the adjoint is taken in the weighted space $L^2([0, T], L^2_xL^2_k (\T^3\times \R^3) )$. Then $W$ is a subspace of $Y$ and we have already shown that 
\[
\Vert h \Vert_{X} \le C_{T, k} \Vert w \Vert_Y. 
\]
Denote
\[
R_\epsilon = - \epsilon (\langle v \rangle^{2\alpha} \mu - \nabla_v  \cdot ( \langle v \rangle^{2\alpha} \nabla_v)  \mu ),
\]
and define the linear mapping on $W$ by
\[
G(w) = (h_0, f_0)_{L^2_xL^2_k}  + \int_{0}^T ( Q(g \chi,\mu), h  )_{L^2_xL^2_k}  d\tau + \int_0^T (h, R_\epsilon)_{L^2_xL^2_k}  d \tau.
\]
It is  easily seen that
\[
(h_0, f_0)_{L^2_xL^2_k}   \le  \Vert f_0 \Vert_{L^2_xL^2_k} \Vert h_0 \Vert_{L^2_x L^2_k} \le C \Vert f_0 \Vert_{L^2_xL^2_k} \Vert h \Vert_{L^\infty( [0, T], L^2_x L^2_k) }  \le  C_{T, k} \Vert f_0 \Vert_{L^2_xL^2_k}   \Vert w\Vert_{Y}.
\]
For the second term we easily deduce
\[
\epsilon \int_{\T^3}\int_{ \R^3} \langle  v \rangle^{2k} (\langle v \rangle^{2\alpha} \mu - \nabla_v  \cdot ( \langle v \rangle^{2\alpha} \nabla_v)  \mu )h dv dx \le C_k \epsilon \Vert h \Vert_{L^2_x L^2_k} \le C_k  \Vert h \Vert_{L^2_x L^2_k},
\]
which implies
\[
\int_0^T ( h, R_\epsilon  )_{L^2_xL^2_k}  d \tau   \le C_{T, k} \Vert   h \Vert_{L^\infty([0, T], L^2_{x}L^2_k ) }  \le C_{T, k}\Vert w\Vert_{Y}.
\]
Since $k_0 -8 \ge k$, by Lemma \ref{L23} we have
\begin{equation*}
\begin{aligned}
 \int_0^T ( Q(g \chi,\mu), h  )_{L^2_xL^2_k}  d\tau  = \int_0^T \int_{\T^3}( Q(g \chi,\mu), h \langle v \rangle^{2k}) dx d\tau 
 \le &C_k \int_0^T \int_{\T^3}    \Vert g \Vert_{L^1_{ k+\gamma+2s} \cap L^2_3}  \Vert h  \Vert_{L^2_k}    dx d\tau 
 \\
 \le& C_{T, k} \sup_{t, x} \Vert g \Vert_{L^\infty_{k_0}}  \Vert h \Vert_{L^\infty([0, T], L^2_{x}L^2_k) }
  \\
 \le& C_{T, k} \sup_{t, x} \Vert g \Vert_{L^\infty_{k_0}}  \Vert w\Vert_{Y}.
\end{aligned}
\end{equation*}
This shows that $G$ is a bounded linear functional on $W$, thus can be extended to $Y$, by Hahn-Banach theorem, there exists $f \in X$ such that
\[
G(w) = \int_0^T (w, f)_{L^2_xL^2_k} d \tau, 
\]
with
\[
\Vert f \Vert_{X} \le C
_{T, k}  \Vert f_0\Vert_{L^2_xL^2_k} + C_{T, k} (1+ \sup_{t, x} \Vert g \chi \Vert_{L^\infty_{k_0} } ) .
\]
so the existence is thus proved. To show that $f\in H_k$, we need to prove that $\mu +f \ge 0$. Let $F = \mu + f$ and $G =\mu + g\chi$, then $G \ge 0$ and $F$ satisfies
\begin{equation}
\label{modified Boltzmann equation F}
\partial_t F+ v \cdot \nabla_x F = -\epsilon( \langle v \rangle^{2 \alpha} I -\nabla_v (\langle v \rangle^{2\alpha} \nabla_v ) ) F + Q(G, F).
\end{equation}
Denote $\eta(x) = \frac 1 2 (x_-)^2,  x_- = \min \{x, 0   \},   F_{\pm} = \pm \max \{\pm  F, 0 \}$.  Multiply \eqref{modified Boltzmann equation F} by $\langle v \rangle^{2k} F_-$ we have
\[
\frac 1 2 \langle v \rangle^{2k }(\partial_t (F_-)^2+ v \cdot \nabla_x (F_-)^2) = -\epsilon \langle v \rangle^{2 \alpha+2 k} (F F_-) +  \epsilon \langle v \rangle^{2k}  F_- \nabla_v (\langle v \rangle^{2\alpha} \nabla_v ) ) F + Q(G, F) F_- \langle v \rangle^{2k}.
\]
We easily compute
\[
\int_{\T ^3}\int_{\R^3} -\epsilon \langle v \rangle^{2 \alpha+2 k} (F F_-)  dv dx= -\epsilon \Vert \langle v \rangle^{\alpha+k}  F_-\Vert_{L^2_x L^2_v}^2,
\]
Since  $\eta'' \ge 0 $, we have
\begin{equation*}
\begin{aligned} 
&\epsilon \int_{\T^3} \int_{\R^3} \langle v \rangle^{2k}    F_- \nabla_v (\langle v \rangle^{2\alpha} \nabla_v F )  dv dx 
\\
=& -\epsilon \int_{\T^3}  \int_{\R^3} \nabla_v (\langle v \rangle^{2k} F_-) \cdot (\langle v \rangle^{2\alpha} \nabla_v F) dv dx
\\
=& - \epsilon  \int_{\T^3} \int_{\R^3} \langle v \rangle^{2\alpha +2k} \eta''(F) |\nabla_v F|^2 dv dx -\epsilon  \int_{\T^3} \int_{\R^3} F_- \nabla_v ( \langle v \rangle^{2k} ) \cdot (\langle v \rangle^{2\alpha} \nabla_v F) dv dx
\\
=& - \epsilon \int_{\T^3}  \int_{\R^3} \langle v \rangle^{2\alpha +2k} \eta''(F) |\nabla_v F|^2 dv dx  + \frac 1 2 \epsilon \int_{\T^3}  \int_{\R^3} F_-^2 \nabla_v( \langle v \rangle^{2\alpha} \nabla_v  \langle v \rangle^{2k} ) dvdx
\\
\le &C_k \epsilon \Vert \langle v \rangle^{\alpha +k-1}  F_- \Vert_{L^2_x L^2_v}^2 \le \frac \epsilon 2 \Vert \langle v \rangle^{\alpha +k}  F_- \Vert_{L^2_xL^2_v}^2  + C_k  \epsilon \Vert \langle v \rangle^{k} F_-  \Vert_{L^2_x L^2_v}^2 .
\end{aligned}
\end{equation*}
Since $F_+ F_- = 0, F_+ \ge 0, F_- \le 0, G \ge 0 $, by Theorem \ref{T34} we have
\begin{equation*}
\begin{aligned}
 &\int_{\T^3} \int_{\R^3} Q(G, F ) F_-\langle v \rangle^{2k} dv dx
 \\
 =&  \int_{\T^3} \int_{\R^3} Q(G, F_-) F_-\langle v \rangle^{2k} dv dx+ \int_{\R^3} Q(G, F_+) F_-\langle v \rangle^{2k} dv dx
\\
=&\int_{\T^3} \int_{\R^3} Q(G, F_-) F_-\langle v \rangle^{2k} dv dx +  \int_{\T^3} \int_{\R^3} \int_{\R^3} \int_{ \mathbb{S}^2 } B G_*' (F_+)' F_-\langle v \rangle^{2k} dv dx 
\\
&- \int_{\T^3} \int_{\R^3} \int_{\R^3} \int_{ \mathbb{S}^2 } B G_* F_+ F_-\langle v \rangle^{2k} dv dx
\\ 
\le &  \int_{\T^3} \int_{\R^3} Q(G, F_-) F_-\langle v \rangle^{2k} dv  dx
\\ 
\le & - (\gamma_2  - C_k \sup_x \Vert g \chi \Vert_{L^\infty_{k_0}}) \Vert F_- \Vert_{L^2_x H^s_{ k + \gamma/2}}^2 -(\gamma_1 - C_k \sup_x  \Vert g \chi \Vert_{L^\infty_{k_0}} )  \Vert   F_-  \Vert_{L^2_x L^2_{ k + \gamma/2}}^2+ C_{k}  \Vert   F_-  \Vert_{L^2_x L^2_{k + \gamma/2 - s' }}^2
\\ 
\le & C_{k}  \Vert   F_-  \Vert_{L^2_x L^2_{k  }}^2.
\end{aligned}
\end{equation*}
Gathering all the terms we have
\[
\frac d {dt} \Vert F_- \Vert_{L^2_x L^2_k }^2 \le C_k \Vert F_- \Vert_{L^2_xL^2_k}^2,  \quad F_-|_{t=0}=0,
\]
then by Gr\"onwall's Lemma we conclude that $F_- = 0$, which implies $f \in H_k$. Uniqueness also follows by Lemma  \ref{L39}, thus the proof for (1) is finished. By \eqref{computation estimate g} we have
\[
\frac d {dt} \Vert \langle v \rangle^{k} f  \Vert_{L^2_{x, v}}^2 \le C_{k} \Vert \langle v \rangle^{k} f  \Vert_{L^2_{x, v}}^2+ C_{k} ,
\]
so for $T \in [0, 1]$ we have
\[
\sup_{t \in [0, T]} \Vert \langle v \rangle^{k} f  \Vert_{L^2_{x, v}}^2 \le C_k (\Vert \langle v \rangle^{k} f_0  \Vert_{L^2_{x, v}}^2+ 1) < + \infty.
\]
Since $\alpha=5$, we have $3+2\alpha  \le 14$, together with $k -l_0(14)  \ge 14$, take $l=14$ in Theorem  \ref{T47} we have
\[
\sup_{t \in [0, T]}\Vert  \langle v \rangle^{14} f\Vert_{L^\infty_{x, v}}  \le \max \{ 2\Vert  \langle v \rangle^{14} f_0 \Vert_{L^\infty_{x, v}} , K_0^1 \},
\]
where 
\[
K_0^1 := C_l e^{C_l T} \max_{ 1 \le i \le 4} \max_{j \in \{ 1/p, p'/p \} }  \Vert \langle v \rangle^{14}  f_0\Vert_{L^2_{x, v}}^{2j} + \sup_{t, x} \Vert g\chi\Vert_{L^\infty_{k_0}}^{2j} T^j  +\epsilon^{2j} T^j )^{\frac {\beta_i -1} {a_i}} , \quad p' = \frac p {2-p},
\]
if  $T \le 1$, we  easily conclude by taking $\delta_1, T_{\delta_0}, \epsilon_*$ small. 
\end{proof}

\begin{equation*}
\begin{aligned} 
\end{aligned}
\end{equation*}

\section{Nonlinear local theory for weak singularity}\label{section 6}
In this section we are considering the 
\begin{equation}
\label{modified equation}
\partial_t f+ v \cdot \nabla_x f = \epsilon L_\alpha(\mu +f) + Q(\mu + f \chi(\langle v \rangle^{k_0} f), \mu + f), \quad f(0, x, v) =f_0(x, v)
\end{equation}
recall the cutoff function $\chi \in C^\infty$ is defined in \eqref{cutoff function}.
If the solution to \eqref{modified equation} satisfies
\[
\Vert f \Vert_{L^\infty ([0, T] \times \T^3 \times \R^3  )} \le \delta_0,
\]
the solution becomes a solution to 
\begin{equation}
\label{epsilon Boltzmann equation}
\partial_t f+ v \cdot \nabla_x f =\epsilon L_\alpha(\mu +f)  + Q(\mu + f , \mu + f), \quad f(0, x, v) =f_0(x, v)
\end{equation}

We first prove the existence of the localized equation.

\begin{lem}\label{L61} Suppose $g, f$ smooth and $\gamma \in (-3, 0], s \in (0, 1/2], \gamma +2s >-1, \alpha =5$. Assume that $f$ is a solution to \eqref{modified equation}. Suppose $k_0$ and the initial data $f_0$ satisfies 
\[
k_0 -l_0(14) \ge 14 + 8 ,  \quad \Vert \langle v \rangle^{14+l_0 (14 )} f_0 \Vert_{L^2_{x, v}} < +\infty,\quad  \mu +f_0 \ge 0,  \quad \Vert\langle v \rangle^{14}  f_0  \Vert_{L^\infty_{x, v} \cap L^2_{x, v}} \le \delta_1, 
\]
where $l_0$ is defined in Theorem \ref{T47}. For any $\delta_0>0$ small satisfies the same assumptions in the lemma above, there exist  $\delta_1,\epsilon_* >0$  such that  for any $\epsilon \in (0, \epsilon_*]$, there exist a time $T_{\epsilon}$ such that \eqref{modified equation} has a  solution $f \in L^\infty([0, T] , L^2_xL^2_{14} (\T^3 \times \R^3)) $. Moreover the solution  $f$ satisfies
\[
\Vert  \langle v \rangle^{14}  f \Vert_{L^\infty( [0, T_\epsilon ] \times \T^3\times \R^3) }  \le \delta_0.
\]
where $\delta_1$, $\epsilon$ is defined in Lemma \ref{L52}.
\end{lem}

\begin{proof}
In this proof we will set $k = 14$. For any $g \in H_k$, define the map 
\[
\Gamma :H_k \to H_k ,\quad \Gamma g = f,
\]
where $f$ is the unique solution to the equation 
\[
\partial_t f +v \cdot \nabla_x f = \epsilon L_\alpha(\mu + f) + Q(\mu + g \chi(\langle v \rangle^{k_0}  g), f) + Q(g \chi(\langle v \rangle^{k_0} g), \mu ) ,\quad f(0) =f_0.
\]
Since $f_0  \in H_{14}  = H_k $, Lemma \ref{L52} guarantees that $\Gamma$ is well-defined provided $\delta_0, \epsilon, T$ are small enough. Our goal is to show that $\Gamma$ is a contraction map on $L^\infty([0, T] , L^2_xL^2_k(\T^3 \times \R^3))$. Let $g, h \in H_k$ and let $f_g, g_h$ be the corresponding solution such that
\[
\partial_t f_g +v \cdot \nabla_x f_g = \epsilon L_\alpha(\mu + f_g) + Q(\mu + g \chi(\langle v \rangle^{k_0}  g), f_g) + Q(g \chi(\langle v \rangle^{k_0} g), \mu ),\quad f_g(0) =f_0,
\]
and 
\[
\partial_t f_h +v \cdot \nabla_x f_h = \epsilon L_\alpha(\mu + f_h) + Q(\mu + h \chi(\langle v \rangle^{k_0}  h), f_h) + Q(h \chi(\langle v \rangle^{k_0} h), \mu ), \quad f_g(0) =f_0.
\]
Since $f_0 \in H_{14 + l_0(14)}$, by Lemma \ref{L52} we have there exist a $T_{\delta_0}$ such that for all $T \in [0, T_{\delta_0} ] $ we have  
\[
\Vert\langle v \rangle^{14}  f_h  \Vert_{L^\infty ([0, T] \times \T^3 \times \R^3)} \le \delta_0.
\]
The difference of the two equations satisfies
\begin{align}
\label{f g f h}
\nonumber
 \partial_t (f_g - f_h) + v \cdot \nabla_x (f_g -f_h) =& \epsilon  L_\alpha(f_g -  f_h)  + Q(\mu + g \chi(\langle v \rangle^{k_0}  g), f_g -f_h) 
 \\
 & +  Q(g \chi(\langle v \rangle^{k_0} g) - h \chi(\langle v \rangle^{k_0} h),  f_h )  +  Q(g \chi(\langle v \rangle^{k_0} g) - h \chi(\langle v \rangle^{k_0} h),  \mu ) .
\end{align}
In the following we will multiply both side of \eqref{f g f h} by $(f_g -f_h) \langle v \rangle^{2k}$ and integrate. For the first term in the right hand side we  easily compute
\[
\int_{\T^3 }\int_{\R^3} \epsilon  L_\alpha(f_g -  f_h) (f_g -f_h)  \langle v \rangle^{2k} dvdx \le  - \frac {\epsilon}  2  \Vert \langle v \rangle^{\alpha+ k }  ( f_g- f_h) \Vert_{L^2_xH^1_v }^2  + C_k \epsilon \Vert f_g-f_h \Vert_{L^2_x L^2_k}^2.
\]
By Theorem \ref{T34} we have
\begin{equation*}
\begin{aligned} 
&\int_{\T^3 }\int_{\R^3} Q(\mu + g \chi(\langle v \rangle^{k_0}  g), f_g -f_h)  (f_g -f_h)  \langle v \rangle^{2k} dvdx  
\\ 
\le & - (\gamma_2  - C_k\sup_x \Vert g \chi \Vert_{L^\infty_{k_0}}) \Vert f_g -f_h \Vert_{L^2_x H^s_{ k + \gamma/2}}^2 -(\gamma_1 -  C_k \sup_x \Vert g \chi \Vert_{ L^\infty_{k_0}} )  \Vert f_g -f_h \Vert_{L^2_x L^2_{ k + \gamma/2}}^2+ C_{k}  \Vert f_g -f_h \Vert_{L^2_x L^2_{k+\gamma/2-s' }}^2
\\
\le&  C_{k}  \Vert f_g -f_h \Vert_{L^2_{k}}^2. 
\end{aligned}
\end{equation*}
Since $k=14, \alpha = 5, s \in (0, \frac 1 2]$ we have
\[
\gamma+2s +k -\alpha + 4   \le 14, \quad   \gamma+2s +k -\alpha +2   \le k,  \quad 2s \le 1.
\]
together with Lemma \ref{L23} and Lemma \ref{L51} we have
\begin{equation*}
\begin{aligned} 
&\int_{\T^3 }\int_{\R^3} Q(g \chi(\langle v \rangle^{k_0} g) - h \chi(\langle v \rangle^{k_0} h),  f_h ) (f_g -f_h)  \langle v \rangle^{2k} dvdx  
\\
\le& C \int_{\T^3 } \Vert g \chi(\langle v \rangle^{k_0} g) - h \chi(\langle v \rangle^{k_0} h) \Vert_{L^1_{\gamma+2s +k -\alpha} \cap L^2_2}  \Vert f_h  \Vert_{L^2_{\gamma +2s +k -\alpha }}  \Vert  (f_g -f_h) \Vert_{H^{2s}_{k+\alpha}} dx
\\
\le&C(\sup_x \Vert f_h\Vert_{L^\infty_{\gamma+2s +k -\alpha + 4 }} ) \Vert g- h\Vert_{L^2_xL^2_k}  \Vert  (f_g -f_h) \Vert_{H^{2s}_{k+\alpha}} dx
\\
\le & \frac {\epsilon} {8} \Vert \langle v \rangle^{\alpha+ k }  ( f_g- f_h) \Vert_{L^2_xH^1_v }^2  + C_{\epsilon} \Vert g- h\Vert_{L^2_xL^2_k}^2.
\end{aligned}
\end{equation*}
By Lemma \ref{L31} and Lemma \ref{L51} we have
\begin{equation*}
\begin{aligned} 
&\int_{\T^3 }\int_{\R^3}  Q(g \chi(\langle v \rangle^{k_0} g) - h \chi(\langle v \rangle^{k_0} h),  \mu )   (f_g -f_h)  \langle v \rangle^{2k} dvdx  
\\
\le &C_ k  \Vert g \chi(\langle v \rangle^{k_0} g) - h \chi(\langle v \rangle^{k_0} h) \Vert_{L^2_xL^2_k}    \Vert  (f_g -f_h) \Vert_{L^2_xL^2_k} 
\\
\le &C_ k  \Vert g- h \Vert_{L^2_xL^2_k}    \Vert  (f_g -f_h) \Vert_{L^2_xL^2_k} 
\\
\le & \frac {\epsilon} {8} \Vert \langle v \rangle^{\alpha+ k }  ( f_g- f_h) \Vert_{L^2_xH^1_v }^2  +  C_{k, \epsilon} \Vert g- h\Vert_{L^2_xL^2_k}^2. 
\end{aligned}
\end{equation*}
Combining the terms above we have
\[
\frac 1 2 \frac d {dt} \Vert f_g -f_h\Vert_{L^2_xL^2_k}^2  + \frac \epsilon 8\Vert \langle v \rangle^{\alpha+ k }  ( f_g- f_h) \Vert_{L^2_xH^1_v }^2  \le C_{k, \epsilon}  \Vert f_g -f_h\Vert_{L^2_xL^2_k}^2 +C_{k, \epsilon}  \Vert g - h\Vert_{L^2_xL^2_k}^2. 
\]
Since $f_g(0) =f_h(0) =f_0$, choosing $T$ small enough which may depend on $\epsilon$, by Gr\"onwall's inequality we have
\[
\Vert f_g -f_h\Vert_{L^\infty ([0, T],  L^2_xL^2_k)}^2    \le \frac 1 2  \Vert g - h\Vert_{L^\infty ([0, T],  L^2_xL^2_k)}^2,
\]
which implies $\Gamma$ is a contraction map.  We obtain a solution to \eqref{modified Boltzmann equation g} by fixed point theorem.  The $L^\infty$ bound is a direct consequence of Lemma \ref{L52}. 
\end{proof}

Now we briefly describe how the discussions in Section \ref{section 4} still works for \eqref{modified equation}.  First  replacing $g$ in  Lemma  \ref{L45}  by $f \chi$, we have

\begin{lem}\label{L62}
Suppose $ F =\mu +f$ smooth and $\gamma \in (-3, 0], s \in (0, 1), \gamma +2s >-1$. Let $T>0, \alpha \ge 0$ be fixed.  Assume that $f$ is a solution to \eqref{modified equation}. Suppose in addition $F$ satisfies that
\[
F \ge 0, \quad \inf_{t, x} \Vert F \Vert_{L^1_v} \ge A >0, \quad \sup_{t, x} (\Vert F \Vert_{L^1_2}  +\Vert F \Vert_{L \log L} ) < B < +\infty,
\]
for some constant $A, B>0$.  Then for any $\epsilon \in [0, 1] $,  $12 \le l \le k_0 $, suppose 
\[
\sup_{t, x} \Vert  f  \Vert_{L^\infty_{9+|\gamma|}}  \le \delta_0,
\]
for some constant $\delta_0 >0$ small. Then there exists $s"> 0 $ and $p^a>0 $ such that for any  $p \in (1, p^a)$  there exist $l_0>0$ which depends on $s'', p$ such that, if we assume
\[ 
\sup_{t} \Vert \langle  v \rangle^{l_0+l }  f  \Vert_{L^2_{x, v} } \le C_1 < +  \infty,
\]
Then for any $0  \le T_1 \le T_2 \le T,  \epsilon \in (0, 1),  0 \le M <K$ we have
\begin{equation*}
\begin{aligned} 
&\Vert f_{K, +}^l (T_2) \Vert_{L^2_{x, v}}^2 + \int_{T_1}^{T_2} \Vert \langle v \rangle^{\gamma/2} (1-\Delta_v)^{s/2}  f_{K, +}^l \Vert_{L^2_{x, v}}^2 d \tau +\frac 1 {C_0} \left(       \int_{T_1}^{T_2} \Vert (1-\Delta_x)^{{s''}/2}  (\langle v \rangle^{-4+ \gamma} (f_{K, +}^l )^2) \Vert_{L^p_{x, v}}^p d \tau                                 \right)^{\frac 1 p}
\\
\le& C \Vert \langle v \rangle^2 f_{K, +}^l (T_1) \Vert_{L^2_{x, v}}^2 +  C \Vert \langle v \rangle^2 f_{K, +}^l (T_1)  \Vert_{L^{2p}_{x, v}}^2 +\frac {CK} {K - M} \sum_{i=1}^4 \frac {\mathcal{E}_{p, s''}(M, T_1, T_2)^{\beta_i}} {(K-M)^{a_i}},
\end{aligned}
\end{equation*}
for some constant  $\beta_i > 1, a_i > 0$ are defined later. The constant $C$ is independent of $\epsilon, K, M, f, T_1, T_2$. Furthermore, the same estimate holds for $h=-f$ with $f_{K, +}^l $ replaced by $h_{K, +}^l$.
\end{lem}
For Lemma \ref{L46}, replacing the use of Lemma \ref{L39} by Corollary \ref{C310}, and repeating  the proof of Lemma \ref{L46} we have
\begin{lem}\label{L63} 
Suppose $f$ smooth and $\gamma \in (-3, 0], s \in (0, 1), \gamma +2s >-1$. Let $T>0, \alpha \ge 0$ be fixed.  Assume that $f$ is a solution to \eqref{modified equation}. Suppose in addition G satisfies that
\[
\mu +f  \ge 0, \quad \inf_{t, x} \Vert \mu + f\chi(\langle v \rangle^{k_0} f)  \Vert_{L^1_v} \ge A >0, \quad \sup_{t, x} (\Vert \mu + f\chi(\langle v \rangle^{k_0} f)  \Vert_{L^1_2}  +\Vert \mu + f\chi(\langle v \rangle^{k_0} f)  \Vert_{L \log L} ) \le B < +\infty,
\]
for some constant $A, B>0$.  Then for any $\epsilon \in [0, 1] $,  $12 \le l \le k_0 , l \ge 3+2\alpha $, suppose 
\[
\sup_{t, x} \Vert  f  \Vert_{L^\infty_{9+|\gamma|}}   \le \delta_0 ,
\]
for some constant $\delta_0 >0$ small.  Then for any $0 < s' <\frac {s} {2(s+3)}$, there exist $s" \in ( 0, s' \frac  4  {2l+r})$ and $p^b := p^b(l, \gamma, s, s') >1$ such that for any $1< p <p^b$, we have
\[
\mathcal{E}_{0, p, s''} \le C_l e^{C_l T} \max_{ j \in \{ 1/p, p'/p  \} } \left( \Vert \langle v \rangle^l f_0\Vert_{L^2_{x, v}}^{2j}     +\epsilon^{2j} T^j  \right), \quad p' = p /(2- p) ,
\]
where $\mathcal{E}_{0, p, s''} $ is defined in \eqref{definition E 0 p s''}. The same estimate holds for $(-f)_{+}^l$ and its associated $\mathcal{E}_{0, p, s''} $.
\end{lem}

To pass the limit in $\epsilon$ we need to show that the time interval of existence is independent of $\epsilon$. So we need to find the relation between the smallness of the initial data and the smallness of the solution.

\begin{lem}\label{L64}
 Suppose $f$ smooth and $\gamma \in (-3, 0], s \in (0, 1), \gamma +2s >-1$. Let $T \in [0, 1], \alpha =5$ be fixed.  Assume that $f$ is a solution to \eqref{modified equation}. Then for any $\epsilon \in [0, 1] $,  $14\le k_0 $, suppose 
\[
\Vert \langle  v \rangle^{k_0 + l_0(k_0)} f_0 \Vert_{L^2_{x, v}}  < +\infty, \quad \Vert\langle v \rangle^{14}  f\Vert_{L^\infty_{t, x, v }} \le \delta_0,
\]
for some constant $\delta_0 >0$ small, where $l_0$ is defined in Theorem \ref{T47}.  Then it follows that
\begin{equation}
\label{L infty estimate 1}
\sup_{t \in [0, T]}\Vert  \langle v \rangle^{k_0} f\Vert_{L^\infty_{x, v}}  \le \max \{ 2\Vert  \langle v \rangle^{k_0}  f_0 \Vert_{L^\infty_{x, v}} , K_0^1 \},
\end{equation}
where 
\[
K_0^1 := C_{k_0} e^{C_{k_0} T} \max_{ 1 \le i \le 4} \max_{j \in \{ 1/p, p'/p \} } ( \Vert \langle v \rangle^{k_0} f_0\Vert_{L^2_{x, v}}^{2j}  +\epsilon^{2j} T^{2j}  )^{\frac {\beta_i -1} {a_i}} , \quad p' = \frac p {2-p} ,
\]
Moreover, for any $0 \le T \le1$, there exist two constants $\delta_*, \epsilon_*>0$ small  enough which is independent of $T$ such that if we assume
\[
\Vert \langle  v \rangle^{k_0} f_0 \Vert_{L^2_{x, v} \cap L^\infty_{x, v} }  \le \delta_*, 
\]
then we have
\[
\Vert \langle v \rangle^{k_0} f\Vert_{L^\infty_{t, x, v }} \le \frac {\delta_1} 2 < \frac {\delta_0} 2,
\]
where $\delta_1$ is defined in Lemma \ref{L52}. 
\end{lem}
\begin{proof}
Take $l=k_0$ in Lemma \ref{L62} and \ref{L63} above,  $k_0$ satisfies the requirements in Lemma \ref{L62} and \ref{L63}. We just need to check that the assumptions in Lemma \ref{L62} and \ref{L63} holds. First  we have
\[
\sup_{t, x} \Vert f\Vert_{L^2_{7+|\gamma|}} \le C \sup_{t, x} \Vert f\Vert_{L^\infty_{14} } \le   \delta_0 \quad ,\sup_{t, x} \Vert f\Vert_{L^\infty_{9 + |\gamma|}} \le C\sup_{t, x} \Vert f\Vert_{L^\infty_{14} } \le   \delta_0,
\]
and 
\[
\inf_{t, x} \Vert \mu + f\chi \Vert_{L^1_v} \ge \Vert \mu \Vert_{L^1_v} - \Vert \langle v \rangle^{-4}  \Vert_{L^1_v} \Vert \langle v \rangle^4 f \chi \Vert_{L^\infty_{t, x, v}} \ge 1 - C\delta_0 \ge A>0,
\]
for some constant $A>0$, also we have
\[
\sup_{t, x} \Vert \mu + f\chi \Vert_{L^1_2 \cap L \log L} \le \sup_{t, x} \Vert \mu  \Vert_{L^1_2 \cap L \log L}+ \sup_{t, x} \Vert  f\chi \Vert_{L^1_2 \cap L \log L} \le C_0(1+\delta_0) \le B,
\]
for some constant $B>0$. By Corollary  \ref{C310} apply for $l=k_0+l_0(k_0)$ we have
\[
\sup_{t \in[0, T] } \Vert \langle v \rangle^{k_0+l_0(k_0)} f(t)  \Vert_{L^2_{x, v}}^2 \le Ce^{C_l T} (1 + \Vert \langle v \rangle^{k_0 +l_0(k_0)} f_0 \Vert_{L^2_{x, v}}^2 ) < +  \infty,
\]
Note that to make Corollary \ref{C310} works, we only need
\[
\sup_{t, x} \Vert f\Vert_{L^2_{7+|\gamma|}} \le \delta_0,
\]
this weight is independent of $k_0$. 
Then by the same proof as Theorem \ref{T47}, \eqref{L infty estimate 1} is thus proved.  Since $T \le 1, \epsilon_*\le 1, \delta_* <1, 2/p >1, 2p'/p >1$, we have
\[
K_0^1  \le C_{k_0} e^{C_{k_0} } \max_{1 \le i \le 4} (\delta_* + \epsilon_*)^{\frac {\beta_i -1} {a_i} }  = C_{k_0} e^{C_{k_0} }  (\delta_* +\epsilon_*)^{\eta_0}, \quad \eta_0 = \min_{1 \le i \le 4} \frac {\beta_i -1} {a_i} ,
\]
by setting
\[
\delta_*  =  \min \{\frac 1 2 \delta_1, \frac 1 {2 ( 2 C_{k_0} e^{C_{k_0}} )^{1/\eta_0}    }\delta_1^{\frac 1 {\eta_0} }     \},  \quad \epsilon_* =\frac 1 {2 (2C_{k_0} e^{C_{k_0}} )^{1/\eta_0}    }  \delta_1^{\frac 1 {\eta_0} } ,
\]
the lemma is thus proved. 
\end{proof}

 We summarize the above results to deduce the local existence for the regularized equation \eqref{epsilon Boltzmann equation}.

\begin{thm}\label{T65}
 Suppose $f$ smooth and $\gamma \in (-3, 0], s \in (0, 1/2], \gamma +2s >-1$. Let $T \in [0, 1], \alpha =5$ be fixed.  Suppose $k_0 \ge l_0(14) +  22 $, where $l_0$ is  defined in Theorem \ref{T47}. Suppose
\[
\Vert \langle  v \rangle^{k_0} f_0 \Vert_{L^2_{x, v} \cap L^\infty_{x, v} }  <\delta_*, \quad\Vert \langle  v \rangle^{k_0 + l_0(k_0)} f_0 \Vert_{L^2_{x, v}}  < +\infty, 
\]
suppose $\delta_*, \epsilon_*$ is defined in Lemma \ref{L64}.  then for any $0<\epsilon \le \epsilon_*$ the equation \eqref{modified equation} has a solution $f$ satisfying 
\begin{equation}
\label{L infty bound for modified}
\Vert \langle v \rangle^{k_0} f \Vert_{L^\infty([0, T] \times \T^3 \times \R^3)} \le \frac {\delta_1} 2 <\delta_0.
\end{equation}
In particular, the solution $f$ becomes a solution to \eqref{epsilon Boltzmann equation}.
\end{thm}

\begin{proof} Since $k_0 \ge  l_0(14) + 22 $, by Lemma \ref{L61} we have, there exist a $T_\epsilon >0$ such that  \eqref{epsilon Boltzmann equation}  has a solution satisfies 
\[
\Vert  \langle v \rangle^{14}  f \Vert_{L^\infty( [0, T_\epsilon ] \times \T^3\times \R^3) }  \le \delta_0,
\]
and Lemma \ref{L64} implies that
\[
\Vert  \langle v \rangle^{k_0}  f \Vert_{L^\infty( [0, T_\epsilon ] \times \T^3\times \R^3) }  \le \frac 1 2 \delta_1.
\]
We claim that such $T_\epsilon$ can be extended to $T$ independent of $\epsilon$. Indeed by Lemma \ref{L61} and the Lemma \ref{L64} we first extend $T_\epsilon$ to $\tilde{T}_{\epsilon}$, where $T_\epsilon$ is the largest interval such that
\[
\Vert \langle v \rangle^{k_0} f \Vert_{L^\infty([0, \tilde{T}_\epsilon] \times \T^3 \times \R^3)} \le \frac 1 2 \delta_1.
\]
Since $k_0 \ge  l_0(20) +28 $, so  for any $t <\tilde{T}_{\epsilon}$ we can apply Lemma \ref{L61} on $f_0 = f(t)$ to continue extension. Since the estimates in Corollary \ref{C310} and Lemma \ref{L64} are all independent of $\epsilon$, we have
\[
\Vert \langle v \rangle^{k_0} f \Vert_{L^\infty([0, \tilde{T}_\epsilon] \times \T^3 \times \R^3)} \le \frac 1 2 \delta_1,
\]
so $\tilde{T}_{\epsilon}$ can be continued to the maximum interval $[0, T]$ for any $T \le 1$. Since \eqref{L infty bound for modified} implies $\chi(\langle v\rangle^{k_0} f) =1$, the solution $f$ becomes a solution to \eqref{epsilon Boltzmann equation}.

\end{proof}
Since we have obtained a solution to \eqref{epsilon Boltzmann equation}, we are ready to pass the limit in $\epsilon$ and obtain a local solution the  Boltzmann equation \eqref{Boltzmann equation}. 
\begin{thm}
Suppose that $s \in (0, \frac 1 2 ]$ and $k_0 \ge l_0(14) + 22$, $l_0$ is defined Lemma \ref{L39}. Suppose that $f_0$ satisfies 
\[
\Vert \langle  v \rangle^{k_0} f_0 \Vert_{L^2_{x, v} \cap L^\infty_{x, v} }  <\delta_*, \quad\Vert \langle  v \rangle^{k_0 + l_0 (k_0) } f_0 \Vert_{L^2_{x, v}}  < +\infty, 
\]
where $\delta_*$ is defined in Lemma \ref{L64} and $\delta_0$ satisfies all the bound above. Then for any $T \le 1$, the nonlinear Boltzmann equation \eqref{Boltzmann equation} has a solution $f \in L^\infty([0, T], L^2_xL^2_{k_0 +l_0(k_0)} )$. Moreover, $f$ satisfies the bound 
\[
\Vert \langle v \rangle^{k_0} f \Vert_{L^\infty([0, T] \times \T^3 \times \R^3)} \le \frac {\delta_1} 2 <\delta_0.
\]
\end{thm}

\begin{proof}
Denote $f^\epsilon$ as the local solution to \eqref{epsilon Boltzmann equation}. By Corollary  \ref{L39}, Lemma \ref{L52} and Lemma \ref{L64} we obtain the uniform in $\epsilon$ bound of $f^\epsilon$ in the following space
\[
L^\infty_{k_0} ([0, T]  \times \T^3 \times\R^3  ) \cap L^\infty([0, T], L^2_xL^2_{k_0 +l_0(k_0)}(\T^3 \times \R^3) ) \cap H^{s'} ([0, T] \times \T^3, H^s_{k_0+l_0(k_0)}(\R^3) ), \quad s' < \frac {s} {2 (s +3)  }.
\]  
We can extract a subsequence, still denote $f^\epsilon$ such that
\[
\Vert f^\epsilon \Vert_{L^\infty_{k_0}([0, T] \times \T^3 \times \R^3)} \le \frac {\delta_1} 2,
\]
by the uniform polynomial decay and a diagonal argument, we have
\[
f^\epsilon \to f \quad \hbox {strongly in}  \quad L^2_{t, x }L^2_{l_0(k_0) + k_0} ([0, T] \times \T^3 \times \R^3),
\]
such strong convergence implies the convergence of $Q(f_\epsilon, f_\epsilon )$ to $Q(f, f)$ as distributions.  Indeed taking $\phi \in C_c^\infty(\R^3_v)$ as test functions, by Lemma \ref{L215} we have
\begin{equation*}
\begin{aligned}
&\left \Vert \int_{\R^3} Q(f^\epsilon, f^\epsilon) \phi(v) dv - \int_{\R^3} Q(f, f) \phi(v) dv \right\Vert_{L^2_{t, x}}
\\
=&\left \Vert \int_{\R^3} \int_{\R^3} \int_{\mathbb{S}^2 } b(\cos \theta) |v-v_*|^\gamma (f^\epsilon(v_*) f^\epsilon (v)   - f(v_*) f (v)    )  (\phi(v') -\phi(v) ) dv dv_* d\sigma \right\Vert_{L^2_{t, x}}
\\
\le&\left \Vert \int_{\R^3} \int_{\R^3}  |f^\epsilon(v_*) f^\epsilon (v)   - f(v_*) f (v)    |  |v-v_*|^\gamma \left| \int_{\mathbb{S}^2 } b(\cos \theta) (\phi(v') -\phi(v) ) d\sigma  \right|  dv dv_*  \right\Vert_{L^2_{t, x}}
\\
\le &\Vert \phi \Vert_{W^{2,\infty}}\left \Vert \int_{\R^3} \int_{\R^3}  |v-v_*|^{2+\gamma} |   f^\epsilon(v_*) f^\epsilon (v)   - f(v_*) f (v)  |    dv dv_* \right\Vert_{L^2_{t, x}}
\\
\le &\Vert \phi \Vert_{W^{2,\infty}}   \left \Vert \int_{\R^3} \int_{\R^3}  |v-v_*|^{2+\gamma} |   f^\epsilon(v_*)    - f(v_*)  |  |f^\epsilon (v)|  dv dv_* \right\Vert_{L^2_{t, x}}
\\
& +\Vert \phi \Vert_{W^{2,\infty}} \left \Vert \int_{\R^3} \int_{\R^3}  |v-v_*|^{2+ \gamma} |   f^\epsilon (v)   - f (v)  |  |f(v_*) |  dv dv_* \right\Vert_{L^2_{t, x}}
\\
\le &\Vert \phi \Vert_{W^{2,\infty}}    (\sup_{t, x} \Vert f_\epsilon \Vert_{L^\infty_6} ) \Vert f^\epsilon -f \Vert_{L^2_{t, x} L^2_4} \to 0 \quad \hbox{as}  \quad \epsilon \to 0.
\end{aligned}
\end{equation*}
Therefore we obtain a solution $f$ to the Boltzmann equation, where $f$ lives in the space
\[
L^\infty_{k_0} ([0, T]  \times \T^3 \times\R^3  ) \cap L^\infty([0, T], L^2_xL^2_{k_0 +l_0(k_0)}(\T^3 \times \R^3) ) \cap H^{s'} ([0, T] \times \T^3, H^s_{k_0+l_0 (k_0) }(\R^3) ), \quad s' < \frac {s} {2 (s +3)  } ,
\]  
so the theorem is ended.
\end{proof}

\section{Global existence}\label{section 7}
We recall that $L$ denotes for the linearized operator 
\[
L f = Q(\mu, f) +Q(f, \mu) - v \cdot \nabla_x f,
\]
and the nonlinear Boltzmann equation is recast as
\[
\partial_t f = L f +Q(f, f) , \quad (t, x, v) \in (0, T) \times\T^3  \times \R^3.
\]

\begin{lem}\label{L71} (\cite{CHJ}, Lemma 3.4 + Lemma \ref{L23}) For any $k \ge 12$, for any smooth function $f, g$, denote $X_k = L^2_k, Y_k =H^{s}_{k+\gamma/2}, Z_k = H^{-s}_{k-\gamma/2}$ which is the dual of $Y_k$ respect to $X_k$, then we have
\[
\Vert Q(f, g) \Vert_{Z_k} \lesssim \Vert f \Vert_{X_k} \Vert  g \Vert_{Y_{k+2}} + \Vert f \Vert_{Y_k} \Vert  g \Vert_{X_k},
\]
which implies
\[
\Vert Q(f, f) \Vert_{L^2_{x}H^{s}_{k-\gamma/2}} \lesssim (\sup_{x} \Vert f \Vert_{L^2_k}) \Vert  f \Vert_{L^2_xH^s_{k+2-\gamma}}.
\]
\end{lem}
\begin{lem}\label{L72} (\cite{CHJ}, Corollary 3.3, Lemma 5.2)
Let $f$ smooth be the solution to the linearized Boltzmann equation
\[
\partial_t f = L f,
\]
Then if $\gamma<0$, then for any $k \ge 12$ large we have
\[
\Vert S_L(t) f \Vert_{L^2_{x, v}} \lesssim \theta (t)  \Vert f \Vert_{L^2_x L^2_k}, \quad \Vert S_L(t) f \Vert_{L^2_{x, v}} \lesssim t^{-1/2}  \theta (t) \Vert f \Vert_{L^2_x H^{-s}_{k-\gamma/2}},
\]
where $ \theta(t) =e^{-\lambda t}$ for some $\lambda >0$ if $\gamma=0$,  $\theta(t) =  (1+t )^{-\frac {|k -12|}{|\gamma|}}$ if $\gamma \in (-3, 0)$.
\end{lem}

\begin{thm}\label{T73} Suppose $f$ smooth and $\gamma \in (-3, 0], s \in (0, 1), \gamma +2s >-1$. Suppose that $F = \mu +f \ge 0$ is a solution of the Boltzmann equation \eqref{Boltzmann equation}. For any $l  \ge 20$, assume that
\[
\sup_{t, x} \Vert f \Vert_{L^2_{7+|\gamma|}} \le \delta_0, \quad \Vert \langle v \rangle^l f_0\Vert_{L^2_{x, v}} <+\infty.
\]
for some small constant $\delta_0>0$. Denote $X = L^2_xL^2_l, Y= L^2_xH^s_{l+\gamma/2}$. Define the norm $||| f |||_X$ and the associate scalar product on $\Pi X$ by
\[
||| f |||_{X}^2  = \eta \Vert f \Vert_X^2 + \int_0^\infty \Vert S_L(\tau ) f \Vert_{L^2_{x, v}}^2 d \tau, \quad (( f, g ))_{X}  = \eta (f, g )_X + \int_0^\infty (S_L(\tau ) f,  S_L(\tau ) g)_{L^2_{x, v} } d \tau.
\]
Then there exist some $\eta>0$, such that the norm $|||\cdot |||$ is equivalent to $\Vert \cdot \Vert_X$ on $\Pi X$, moreover there exist  some constant $ K>0$ such that any smooth solution to the Boltzmann equation with $Pf_0 =0$ satisfies 
\begin{equation}
\label{estimate equivalent f}
\frac d {dt} ||| f |||_X^2 \le -K \Vert f \Vert_Y^2.
\end{equation}
As a consequence if $\gamma = 0$, we have
\[
\Vert \langle   v \rangle^{l} f (t)\Vert_{L^2_{x, v}} \le Ce^{-\lambda t} \Vert \langle v \rangle^l f_0\Vert_{L^2_{x, v}}, \quad t \in   [0, T],
\]
for some constant $C, \lambda>0$. If $\gamma \in (-3, 0)$ we have
\[
\Vert \langle   v \rangle^{l_1} f (t)\Vert_{L^2_{x, v}} \le Ct^{- \frac {|l - 14 |} { |\gamma| } } \Vert \langle v \rangle^l f_0\Vert_{L^2_{x, v}}, \quad t \in   [0, T], \quad \forall 14 \le l_1 < l .
\]
for some constant $C >0$. Moreover for any $0 < s' < \frac {s }  {2(s+3)}$, we have
\[
\sup_{t \in [0, T] } \Vert \langle v \rangle^{l} f\Vert_{L^2_{x, v}} + \int_{0}^T \Vert (I-\Delta_x)^{s'/2} f \Vert_{L^2_{x, v}}^2 +\int_{0}^T  \Vert \langle v \rangle^{l +\gamma/2} (1-\Delta_v)^{s/2} f  \Vert_{L^2_{x, v}}^2 \le C \Vert \langle v \rangle^{l} f_0 \Vert_{L^2_{x, v}}^2,
\]
where all the constants $C$ are all independent of $T$.
\end{thm}
\begin{proof}
First by Lemma \ref{L72}
\[
\int_0^\infty \Vert S_L(\tau ) f \Vert_{L^2_{x, v}}^2 d \tau \lesssim \Vert f \Vert_X^2 \int_0^\infty \theta^2(\tau) d\tau \lesssim \Vert f \Vert_X^2 . 
\]
The equivalence between two norms is thus proved. Then we easily compute
\[
\frac {d} {dt} ||| f(t) |||^2_X  = \eta (Q(\mu + f, \mu+f) ,f )_X + \int_0^\infty ( S_L(\tau)Lf, S_L(\tau) f )_{L^2_{x, v}} d\tau + \int_0^\infty ( S_L(\tau)Q(f, f ) , S_L(\tau) f )_{L^2_{x, v}} d\tau.
\]
We will estimate the terms separately, first by Theorem \ref{T36} we have
\begin{equation*}
\begin{aligned}
(Q(\mu + f, \mu+f), f)_X = \int_{\T^3} (Q(\mu + f, \mu+f), f \langle v \rangle^{2k}) dx  
\le& \int_{\T^3}  -2c_0 \Vert f \Vert_{H^s_{ l +\gamma/2}}^2 + C_{k}  \Vert f \Vert_{L^2}^2 + C_k\Vert f \Vert_{L^2_{|\gamma| +7 }} \Vert f \Vert_{H^s_{ l + \gamma/2}}^2 dx
\\
\le&  -2c_0 \Vert f \Vert_{Y}^2 + C_{k}  \Vert f \Vert_{L^2_{x, v}}^2 + C_k \sup_{t, x}\Vert f \Vert_{L^2_{7+|\gamma|}} \Vert f \Vert_{Y}^2
\\
\le& -c_0 \Vert f \Vert_{Y}^2 + C_{k}  \Vert f \Vert_{L^2_{x, v}}^2.
\end{aligned}
\end{equation*}
For the second term, by Lemma \ref{L72}, we have
\[
\int_0^\infty ( S_L(\tau)Lf, S_L(\tau) f )_{L^2_{x, v}} d\tau =\int_0^\infty \frac {d} {d\tau} \Vert S_L(\tau) f(t)\Vert_{L^2_{x, v}}^2 d\tau = \lim_{\tau \to \infty}\Vert S_L(\tau ) f(t)\Vert_{L^2_{x, v}}^2 - \Vert f(t) \Vert_{L^2_{x, v}}^2 = -\Vert f(t) \Vert_{L^2_{x, v}}^2.
\]
For the last term, by Lemma \ref{L71} and Lemma \ref{L72} we have
\begin{equation*}
\begin{aligned}
\int_0^\infty ( S_L(\tau)Q(f, f ) , S_L(\tau) f )_{L^2_{x, v}} d\tau \le &\int_0^\infty \Vert S_L(\tau)Q(f, f )\Vert_{L^2_{x, v}}   \Vert S_L(\tau) f \Vert_{L^2_{ x, v}} d\tau 
\\
\lesssim &\Vert Q(f, f )\Vert_{L^2_x H^{-s}_{10-\gamma/2}}    \Vert  f \Vert_{L^2_x L^2_{10}} \int_0^\infty \theta_1(\tau) \theta(\tau) d\tau
\\
\lesssim &\sup_{x} \Vert f \Vert_{L^2_{10}}  \Vert f\Vert_{L^2_xH^s_{10+2-\gamma/2}} \Vert  f \Vert_{Y} 
\lesssim \delta_0 \Vert  f \Vert_{Y}^2.
\end{aligned}
\end{equation*}
Combining all the terms and taking a suitable $\eta>0$ small, \eqref{estimate equivalent f} is thus proved. For the convergence rate, the case $\gamma \ge 0$ can be proved directly by using Gr\"onwall's Lemma and the case $\gamma \in (-3, 0)$ can be proved by interpolation, see \cite{C} Lemma 3.12 for exmaple.  By integrate on $[0, T]$ in \eqref{estimate equivalent f} we have
\[
\sup_{t \in [0, T] } \Vert \langle v \rangle^{l} f\Vert_{L^2_{x, v}} + K\int_{0}^T  \Vert \langle v \rangle^{l + \gamma/2} (1-\Delta_v)^{s/2} f  \Vert_{L^2_{x, v}}^2 d \tau \le C \Vert \langle v \rangle^{l} f \Vert_{L^2_{x, v}}^2,
\]
where the constant $C$ does not depend on $T$. By Lemma \ref{L23} we have
\[
\Vert\langle v \rangle^3 (I-\Delta_v)^{-1} ({Q}(F, F)) \Vert_{L^2_{v}}  \lesssim \Vert\langle v \rangle^3 (I-\Delta_v)^{-s} ({Q}(\mu+f, f) +Q(f, \mu) ) \Vert_{L^2_{v}}  \lesssim \Vert f\Vert_{L^2_{7}} \Vert f\Vert_{L^2_{7}} + \Vert f\Vert_{L^2_{7}},
\]
which implies
\[
\int_{0}^T \Vert\langle v \rangle^3 (I-\Delta_v)^{-1} ({Q}(F, F)) \Vert_{L^2_{x, v}}^2 d\tau  \lesssim (1 + \sup_{t, x} \Vert f \Vert_{L^2_7}) \int_{0}^T \Vert  \langle v \rangle^7 f \Vert_{L^2_{x, v}}^2 d\tau \lesssim \Vert \langle v \rangle^{l} f_0 \Vert_{L^2_{x, v}}^2,
\] 
similarly as \eqref{computation s'} we have
\begin{equation*}
\begin{aligned}
\int_{0}^t \Vert (I-\Delta_x)^{s'/2} f \Vert_{L^2_{x, v}}^2 d\tau  \lesssim& \Vert \langle v \rangle^3 f(0)  \Vert_{L^2_{x, v}}^2 + \Vert \langle v \rangle^3  f(t)  \Vert_{L^2_{x, v}}^2 + \int_{0}^t \Vert (I-\Delta_v)^{s/2} f \Vert_{L^2_{x, v}}^2 d\tau
\\
& + \int_{0}^t \Vert\langle v \rangle^3 (I-\Delta_v)^{-1} ({Q}(F, F)) \Vert_{L^2_{x, v}}^2 d\tau
\\
\lesssim& \Vert \langle v \rangle^{l_0} f_0 \Vert_{L^2_{x, v}}^2 + \int_{0}^t \Vert\langle v \rangle^3 (I-\Delta_v)^{-1} ({Q}(F, F)) \Vert_{L^2_{x, v}}^2 d\tau ,
\end{aligned} 
\end{equation*}
so the theorem is thus proved.
\end{proof}

\begin{thm}\label{T74} (Global Existence) Suppose $f$ smooth and $\gamma \in (-3, 0], s \in (0, 1), \gamma +2s >-1$. Suppose that $F = \mu +f \ge 0$ is a solution of the Boltzmann equation \eqref{Boltzmann equation}. Then for any $k_0 \ge 14$ there exist $\delta_*, \delta_0>0$ small such that if  the initial data$f_0$ satisfies 
\[
u +f_0 \ge 0, \quad \Vert  \langle v \rangle^{k_0} f_0\Vert_{L^\infty_{x, v} \cap L^2_{x, v}} \le \delta_*, \quad \Vert \langle v \rangle^{k_0+l_0(k_0)}f_0 \Vert_{L^2_{x, v}} < + \infty, \quad Pf_0 =0.
\] 
Then the Boltzmann equation has a global solution satisfies 
\[
\Vert  \langle v \rangle^{k_0} f\Vert_{L^\infty ( [0, \infty) \times \T^3\times \R^3 )} \le \frac {\delta_0} 2.
\] 
Moreover,  if $\gamma = 0$, we have
\[
\Vert \langle   v \rangle^{k_0} f (t)\Vert_{L^2_{x, v}} \le Ce^{-\lambda t} \Vert \langle v \rangle^{k_0} f_0\Vert_{L^2_{x, v}}, \quad \forall t \in   [0, +\infty),
\]
for some constant $C, \lambda>0$. If $\gamma \in (-3, 0)$ we have
\[
\Vert \langle   v \rangle^{k_1} f (t)\Vert_{L^2_{x, v}} \le Ct^{- \frac {|k_0 - 14 |} { |\gamma| } }  \Vert \langle v \rangle^{k_0} f_0\Vert_{L^2_{x, v}}, \quad \forall  t \in   [0, +\infty), \quad \forall 14 \le  k_1 < k_0.
\]
\end{thm}
\begin{proof}
The reason that Theorem \ref{T65} can only treat a short time existence is because that the bound of $\mathcal{E}_{0, p, s''} $ relies on $T$ (precisely $Ce^{CT}$), which will exceed if $T$ is large (see Lemma \ref{L63}). Now by Theorem \ref{T73}, we know that the term $Ce^{CT}$ can be replaced by a constant $C$ which does not depend on $t$, so the result in Lemma \ref{L63} can be improved to 
\[
\mathcal{E}_{0, p, s''} \le C_{k_0} \max_{ j \in \{ 1/p, p'/p  \} } \Vert \langle v \rangle^{k_0} f_0\Vert_{L^2_{x, v}}^{2j}  \le C_{k_0} \Vert \langle v \rangle^{k_0} f_0\Vert_{L^2_{x, v}}.
\]
As a consequence, there exists $C_{k_0}$ independent of $T$ such that
\[
K_0(\mathcal{E}_{0, p, s''}  ) \le C_{k_0} \Vert \langle v \rangle^{k_0} f_0\Vert_{L^2_{x, v}}^{\eta_0} ,\quad \eta_0 = \min_{1 \le i \le 4 } \frac {\beta_i-1}{a_i}.
\]
Taking $\delta_*$ small, we  conclude
\[
\Vert  \langle v \rangle^{k_0} f\Vert_{L^\infty ( [0, T) \times \T^3\times \R^3 )} \le \frac {\delta_0} 2 <\delta_0 , \quad \forall T>0,
\] 
thus for any $T>0$, the equation can be extended beyond $T$, the global existence is thus proved. The convergence rate follows by Theorem \ref{T73}. 
\end{proof}

\begin{thm}
Suppose $k_0$ and the initial data $f_0$ satisfies the same condition as Theorem \ref{T74} above. Then there exist $C_{k_0}, \eta_0$ such that for any $t>1$ the solution obtained Theorem \ref{T74} satisfies if $\gamma =0$
\[
\Vert \langle v \rangle^{k_0}    \Vert_{L^\infty_{x, v}} \le C_{k_0}  e^{-\lambda t} \Vert \langle v \rangle^{k_0} f_0 \Vert_{L^2_{x, v}}^{2\eta_0/p},
\]
for some constant $\lambda>0$ and if $\gamma \in (-3, 0)$
\[
\Vert \langle v \rangle^{k_0}    \Vert_{L^\infty_{x, v}} \le C_{k_0}  t^{-\alpha} \Vert \langle v \rangle^{k_0} f_0 \Vert_{L^2_{x, v}}^{2\eta_0/p},
\]
for some constant $\alpha>0$. 
\end{thm}
\begin{proof}
For any $K, t_1>0$, let $\mathcal{E}_{p, s''}(K, t_1, +\infty)$ be the energy  functional defined  in  \eqref{energy functional}
\begin{align*}
\mathcal{E}_{p, s''} (K, t_1, +\infty) :=& \sup_{t \ge t_1} \Vert f_{K, +}^l \Vert_{L^2_{x, v}}^2   + \int_{t_1}^{+\infty} \int_{\T^3} \Vert \langle  v \rangle^{\gamma/2} f_{K, +}^l \Vert_{H^s_v}^2 dx d\tau 
\\
&+\frac 1 {C_0} \left( \int_{t_1}^{+\infty} \Vert (1-\Delta_x)^{s''/2} ( \langle v \rangle^{-2+ \gamma/2} f_{K, +}^l)^2\Vert_{L^p_{x, v}}^p \right)^{1/p}.
\end{align*}
Our main goal is to remove the dependence of $L^\infty$ norm of $f_0$ so that we can have a decay in the weighted $L^\infty$ norm, define the levels
\[
M_k := K_0 (1-1/2^k) , \quad k =0, 1, 2,...,
\]
Setting $f_k = f_{M_k, +}^l$ and proceeding as Theorem \ref{T47}, we arrive at
\begin{equation}
\label{E p s'' global computation}
\mathcal{E}_{p, s''}(M_k, t_1, +\infty) \le C \Vert  \langle v \rangle^2 f_k(t_1) \Vert_{L^2_{x, v}}^2 + C \Vert \langle v \rangle^2 f_k(t_1) \Vert_{L^{2p}_{x, v}}^2 + C \sum_{i=1}^4 \frac {2^{k(a_i+1)}} {K_0^{a_i}}   \mathcal{E}_{p, s''}(M_{k-1}, t_1, + \infty)^{\beta_i},
\end{equation}
for $k =1, 2,...$. The parameters $a_i>0$, $\beta_i>0$ are the same as Theorem \ref{T47}. Fix $T>0$ and $T_k$ be the increasing time sequence and denote $E_k$ by
\[
T_{k-1} := T(1- \frac 1 {2^k}), \quad k =0, 1, 2..., \quad \mathcal{E}_k  := \mathcal{E}_{p, s''} (M_k, T_k, +\infty).
\]
Using  the monotonicity $\mathcal{E}_{p, s''}(\cdot, \cdot, \cdot)$ in its first and second variables and integrate \eqref{E p s'' global computation} in $t_1 \in [T_{k-1}, T_k]$,we deduce
\[
\mathcal{E}_k = \mathcal{E}_{p, s''}(M_k, T_k, +\infty) \le (T_k -T_{k-1} )^{-1} \left(\int_{T_{k-1} }^{T_k} \Vert \langle v \rangle^2 f_k(\tau) \Vert_{L^2_{x, v} }^2 d\tau  +  \int_{T_{k-1} }^{T_k} \Vert \langle v \rangle^2 f_k(\tau) \Vert_{L^{2p}_{x, v} }^2 d\tau      \right) +C \sum_{i=1}^{4} \frac {2^{k(a_i+1 )} \mathcal{E}_{K-1}^{\beta_i}} {K_0^{a_i}}.
\]
Similarly as \eqref{estimate f K l E version} we have
\[
\int_{T_{k-1}}^{T_k}     \Vert  \langle v \rangle^{2} f_k(\tau) \Vert_{L^2_{x, v}}^2 d\tau \le C \frac {\mathcal{E}_{p, s''}(M_{k-1}, T_{k-1}, T_k)^{r_*} } { (M_k -M_{k-1})^{\xi_*- 2 } } \le C_0 \frac   {2^{k(\xi_* -2) }  \mathcal{E}_{k-1}^{r*}} {K_0^{\xi_* - 2}}.
\]
By H\"older's inequality and Lemma \ref{L44} we have
\begin{equation*}
\begin{aligned} 
\int_{T_{k-1}}^{T_k}     \Vert  \langle v \rangle^{2} f_k(\tau) \Vert_{L^{2p}_{x, v}}^2 d\tau =&\int_{T_{k-1}}^{T_k}     \Vert  \langle v \rangle^{4} (f_k( \tau ))^2 \Vert_{L^p_{x, v}} d\tau 
\\
&\le (T_k -T_{k-1})^{\frac {p-1} p} \Vert  \langle v \rangle^4 (f_k)^2\Vert_{L^p_{t, x, v}} \le C_0 (T_k -T_{k-1})^{\frac {p-1} p} \frac   {2^{k(\frac {\xi_* -2p} {2p} ) }  \mathcal{E}_{k-1}^{ \frac {r*} {p}}} {K_0^{ \frac {\xi_* - 2p} {2p} }} .
\end{aligned}
\end{equation*}
Assuming $T \ge 1$, the following analogous estimate holds
\[
\mathcal{E}_k \le C_l \sum_{i=1}^4 \frac {2^{k(a_i+1)} \mathcal{E}_{k-1}^{\beta_i} } {K_0^{a_i}}, \quad k=1, 2, 3..., \quad T \ge 1.
\]
Apply the same De Giorgi iteration as Theorem \ref{T47} we conclude that
\[
\sup_{t \ge T} \Vert \langle v \rangle^l f_+(t, \cdot, \cdot)\Vert_{L^\infty_{x, v}} \le K_0 :=K_0(\mathcal{E}_0) = C_l \max_{1 \le i \le 4} \mathcal{E}_{0}^{\frac {\beta_i-1}  {a_i}}, \quad l \le k_0,
\]
where $\mathcal{E}_0 := \mathcal{E}_{p, s''}(0, T/2, +\infty)$. For any $T \ge 0$ and $l \le l_0(k_0) +k_0$, by Lemma \ref{L63} apply for $f(T)$  and Theorem \ref{T74} we have
\begin{equation*}
\begin{aligned}
 \mathcal{E}_P(0, T, +\infty) =& \sup_{t \ge T} \Vert \langle v \rangle^l  f_+\Vert_{L^2_{x, v}}^2   + \int_{T}^{+\infty} \int_{\T^3} \Vert \langle  v \rangle^{l+\gamma/2} f_+ \Vert_{H^s_v}^2 dx d\tau 
 \\
 &+\frac 1 {C_0} \left( \int_{T}^{\infty} \Vert (1-\Delta_x)^{s''/2} ( \langle v \rangle^{ l -2+ \gamma/2} f_+)^2\Vert_{L^p_{x, v}}^p \right)^{1/p}  \le C_l \Vert \langle v \rangle^{l} f(T) \Vert_{L^2_{x, v}}^{2/p}  
\end{aligned}
\end{equation*} 
so the convergence rate follows by the convergence rate in Theorem \ref{T74}. 
\end{proof}

\section{Strong singularity}\label{section 8}
This section is similar as \cite{AMSY2} Section 7. This section we prove the existence  of the solution in the case of strong singularity, the only reason that we have to restrict to the weak singularity is because that in  the construction of the solution, the regularization term $\epsilon L_\alpha$ need to be used to control the $H^{2s}$ norm, while all the priori estimates is performed for the full range $s \in (0, 1)$, we use an additional $\eta$ approximation in the collision kernel, recall that the original collision kernel satisfies
\[
b(\cos \theta)   \sim \frac 1 {\theta^{2+2s}}, \quad s \in (\frac 1 2, 1).
\]
Fix $s_* \in (0, \frac 1 2]$ such that
\[
2s - 2s_* < 1.
\]
For any $\eta \in (0, 1)$, let $Q_\eta$ be the approximate operator with the collision kernel 
\[
\frac {\alpha_0} {\theta^{2+2s_*} }  \le b_\eta(\cos\theta)  = \frac {b(\cos \theta)  \theta^{2+2s} } {\theta^{2+2s^*} (\theta +\eta)^{2s -2s_*} } \le b(\cos \theta), 
\]
for some constant $\alpha_0 >0$. We note that the coefficient $\alpha_0$ is independent of $\eta$ since
\[
\frac {1} {(\theta +\eta)^{2s -2s_*} }  \ge \frac 1 {(\pi +1)^{2s -2s_*}}, \quad \forall \theta \in (0, \pi), \quad \eta \in (0, 1). 
\]
For the $Q_\eta$ term, the upper bound of the collisional kernel does not change, since the proof of Lemma \ref{L23} only use the upper bound of $f$ (which can be seen in \cite{H}), we still have
\begin{lem}\label{L81}
Lemma \ref{L23} is still true when $Q$ is replaced by $Q_\eta$ for all $\eta \in [0, 1]$, moreover the constant is independent of $\eta$.
\end{lem}
In this section, we will consider the modified  equation 
\begin{equation}
\label{modified equation eta epsilon}
\partial_t f_\eta+ v \cdot \nabla_x f_\eta = \epsilon L_\alpha(\mu +f_\eta) + Q_\eta(\mu + f_\eta \chi(\langle v \rangle^{k_0} f_\eta), \mu + f_\eta), \quad f_\eta(0) =f_0,
\end{equation}
and its linearized version
\begin{equation}
\label{modified equation eta epsilon g}
\partial_t f_\eta+ v \cdot \nabla_x f_\eta = \epsilon L_\alpha(\mu +f_\eta) + Q_\eta(\mu + g \chi(\langle v \rangle^{k_0} f), \mu + f_\eta): =\tilde{Q} _\eta(\mu + g \chi, \mu + f_\eta), \quad f_\eta(0) =f_0.
\end{equation}
Choices of weights remain the same as in the previous sections. We will show the details for the basic energy estimates for the linearized equation to illustrate how to derive uniform-in-$\eta$ bounds. The rest of the steps are parallel to those before and their details will be either sketched or omitted. The regularization $\epsilon L_\alpha$ helps to simplify the estimates, since for each fixed $\epsilon$, the gain of velocity regularity (and subsequently the hypoellipticity) now comes from $\epsilon L_\alpha$ instead of $Q$.

\begin{lem}\label{L82} 
Suppose $\gamma+2s>-1, s \in (0, 1), \gamma \in (-3, 0]$. Fo any smooth function $f, g$, suppose $G= \mu +g\ge 0$ and $\delta_0>0$ in the cutoff function is small enough such that $G_\chi = u +g \chi$ satisfies 
\[
G_\chi  \ge 0,\quad \inf_{t, x} \Vert G_\chi  \Vert_{L^1_v} \ge A, \quad  \sup_{t, x}(\Vert G_\chi  \Vert_{L^1_2} +\Vert G \Vert_{L \log L} )\le B,
\]
for some constant $A< B>0$. If $f_\eta$ is a solution to \eqref{modified equation eta epsilon g}. Then for any $\epsilon \in [0, 1] $,  $12 \le l \le k_0 - 8,\alpha \ge 5$, for any $t \ge 0$  we have
\[
\Vert \langle v \rangle^l f_\eta (t)\Vert_{L^2_{x, v}}^2 + \frac \epsilon 4 \int_0^t \Vert \langle v \rangle^l f_\eta (\tau) \Vert_{L^2_xH^s_{\gamma/2}}^2 d\tau  + \le C_l e^{C_{l, \epsilon} t} (\Vert \langle v \rangle^l f_0\Vert_{L^2_{x, v}}^2  +  t ) ,
\]
for some constant $c_0, C_l>0$.
If in addition we assume $l \ge 3+2\alpha $, then for any $0 < s' < 1/8$, for any $t \ge 0$ we have
\[
\int_{0}^t \Vert (I-\Delta_x)^{s'/2} f_\eta (\tau) \Vert_{L^2_{x, v}} d \tau \le C_l e^{C_{l, \epsilon} t} (\Vert \langle v \rangle^l f_0\Vert_{L^2_{x, v}}^2  +  t),
\]
where the constant $C_{l, \epsilon}$ is independent of $\eta$ but depends on $\epsilon$ while $C_l$ is independent of $\eta$ and $\epsilon$.
\end{lem}
\begin{proof} By Lemma \ref{L38}, the regularization term $\epsilon L_\alpha $ gives
\[
\epsilon \int_{\T^3\times \R^3  }L_\alpha(\mu + f_\eta) f_\eta \langle v \rangle^{2l} dv dx  \le  -\frac \epsilon 2 \Vert \langle v \rangle^{l+\alpha} f_\eta \Vert_{L^2_x H^1_v }^2 + C_{l} \epsilon \Vert \langle v \rangle^l f_\eta \Vert_{L^2_{x, v} }^2  + C_{l} \epsilon \Vert \langle v \rangle^l f_\eta \Vert_{L^2_{x, v} }.
\]
Since $\epsilon L_\alpha$ will provide the dominating term in both the weight and the regularity, by Lemma \ref{L81} we have
\begin{equation*}
\begin{aligned}
&\int_{\T^3} \int_{\R^3} Q_{\eta} (\mu + g \chi (\langle v \rangle^{k_0} ), \mu +f_\eta ) f_\eta \langle v \rangle^{2l} dv dx
\\
\le& C_l \int_{\T^3}  \Vert \mu + g \chi (\langle v \rangle^{k_0})  \Vert_{L^1_{\gamma + 2s + l +2}\cap L^2_{2} } ( \Vert \langle v \rangle^{l + s+\gamma/2 } f_\eta \Vert_{L^2_xH^s_v}^2    +     \Vert \langle v \rangle^{l+2s+\gamma} \mu \Vert_{H^{2s}_v}   \Vert  \langle v \rangle^{l} f_\eta \Vert_{L^2_{x, v}}    )        dx
\\
\le& C_l \Vert \mu + g \chi (\langle v \rangle^{k_0})  \Vert_{L^\infty_{k_0} } ( \Vert  \langle v \rangle^{l} f_\eta \Vert_{L^2_{x, v}}   + \Vert \langle v \rangle^{l + s+\gamma/2 } f_\eta \Vert_{L^2_xH^s_v}^2 )
\\
\le&C_l \Vert  \langle v \rangle^{l} f_\eta \Vert_{L^2_{x, v}}   + C_l \Vert \langle v \rangle^{l + s+\gamma/2 } f_\eta \Vert_{L^2_xH^s_v}^2
\\
\le&C_l \Vert  \langle v \rangle^{l} f_\eta \Vert_{L^2_{x, v}}   + \frac \epsilon 4 \Vert \langle v \rangle^{l +\alpha  } f_\eta \Vert_{L^2_xH^1_v}^2 +C_{l, \epsilon } \Vert \langle v \rangle^l f_\eta \Vert_{L^2_{x, v}}^2, \quad \alpha \ge 5.
\end{aligned}
\end{equation*}
Gathering the two terms we deduce 
\begin{equation*}
\begin{aligned}
\frac {d} {dt} \Vert \langle v \rangle^l f_\eta \Vert_{L^2_{x, v}}^2 \le&  -\frac \epsilon 4 \Vert \langle v \rangle^{l+\alpha} f_\eta \Vert_{L^2_x H^1_v }^2 + C_{l, \epsilon} \Vert \langle v \rangle^l f_\eta \Vert_{L^2_{x, v} }^2  + C_{l, \epsilon} \Vert \langle v \rangle^l f_\eta \Vert_{L^2_{x, v} }
\\
\le& -\frac \epsilon 4 \Vert \langle v \rangle^{l+\alpha} f_\eta \Vert_{L^2_x H^1_v }^2 + C_{l, \epsilon} \Vert \langle v \rangle^l f_\eta \Vert_{L^2_{x, v} }^2  + 1,
\end{aligned}
\end{equation*}
thus the theorem can be proved same way as Lemma \ref{L39} (taking $s =1$ ). 
\end{proof}

The  $L^2$ level set estimate parallel to Lemma \ref{L41} is

\begin{lem}
Suppose $\gamma+2s>-1, s \in (\frac 1 2, 1), \gamma \in (-3, 0]$. Fo any smooth function $f, g$, suppose $G= \mu +g\ge 0$ and $\delta_0>0$ in the cutoff function is small enough such that $G_\chi = u +g \chi$ satisfies 
\[
G_\chi  \ge 0,\quad \inf_{t, x} \Vert G_\chi  \Vert_{L^1_v} \ge A, \quad  \sup_{t, x}(\Vert G_\chi  \Vert_{L^1_2} +\Vert G \Vert_{L \log L} )\le B,
\]
For some constant $A, B >0$. For any constant $12 \le l \le k_0 , K \ge 0, \alpha = 5$ we have
\begin{equation*}
\begin{aligned} 
&\int_{\T^3} \int_{\R^3} Q_\eta (\mu +g \chi, \mu + f) f_{K, +}^l \langle v \rangle^l dv dx  + \int_{\T^3} \int_{\R^3}  \epsilon L_{\alpha} (\mu +f) f_{K ,+}^l \langle v \rangle^l  dvdx
\\
\le& -\frac \epsilon 4 \Vert  \langle v \rangle^\alpha f_{K, +}^l \Vert_{L^2_x H^1_v}^2 +  C_{l, \epsilon} \Vert f_{K,+}^l \Vert_{L^2_{x, v}}^2 
+ C_l (1 + K) \Vert f_{K, +}^l \Vert_{L^1_x L^1_2},
\end{aligned}
\end{equation*}
for some constant $c_1, C_{l, \epsilon>}0$, where the constants are all independent of $\eta$.
\end{lem}
\begin{proof}
First we have
\begin{equation*}
\begin{aligned} 
\int_{\T^3}\int_{\R^3} Q_{\eta}(\mu +g \chi, \mu + f) f_{K, +}^l \langle v \rangle^l dv dx  =& \int_{\T^3}\int_{\R^3} Q_{\eta} (\mu +g \chi, f-\frac K {\langle v \rangle^l}) f_{K, +}^l \langle v \rangle^l dv dx
\\
&+ \int_{\T^3 } \int_{\R^3} Q_{\eta}(\mu +g \chi , \mu + \frac {K} {\langle v \rangle^l}) f_{K, +}^l \langle v \rangle^l dv dx:=H_1 +H_2.
\end{aligned}
\end{equation*}
By Lemma \ref{L81} we have
\begin{equation*}
\begin{aligned} 
&\int_{\T^3}\int_{\R^3} Q_{\eta}(\mu +g \chi, f-\frac K {\langle v \rangle^l}) f_{K, +}^l \langle v \rangle^l dv dx
\\
\le &  \int_{\T^3 } \int_{\R^3} \int_{\R^3} \int_{\mathbb{S}^2 } (\mu_* + g_*\chi_*) f_{K, +}^l \langle v \rangle^l (f_{k, +}^l(v') \langle v' \rangle^l - f_{K, +}^l \langle v \rangle^l ) b_\eta( \cos \theta) |v-v_*|^\gamma   dv dv_* d\sigma dx 
\\
\le& \int_{\T^3}\int_{\R^3} Q_{\eta}(\mu +g \chi, f_{K, +}^l \frac 1 {\langle v \rangle^l}) f_{K, +}^l \langle v \rangle^l dv dx
\\
\le & C \int_{\T^3} \Vert \mu +g\chi \Vert_{L^1_{\gamma/2 +s +l +2} \cap L^2_2} \Vert f_{K, +}^l \Vert_{H^s_{\gamma/2 + s }}^2 dx
\\
\le& C  \Vert f_{K, +}^l \Vert_{L^2_xH^s_{\gamma/2 + s }}^2
\\
\le& \frac \epsilon 8 \Vert \langle v \rangle^{\alpha} f_{K, +}^l \Vert_{L^2_x H^1_v}^2 + C_{l, \epsilon}\Vert  f_{K, +}^l \Vert_{L^2_{x, v}}^2 .
\end{aligned}
\end{equation*}
By term $T_2, T_3$ in Lemma \ref{L41} we have
\begin{equation*}
\begin{aligned} 
H_2 \le &  C_{l}(1 + \sup_x \Vert g \chi \Vert_{L^\infty_l}  ) \Vert f_{K,+}^l \Vert_{L^2_x L^2_2}^2 + C_l (1 + K)(l + \sup_x \Vert g \chi \Vert_{L^\infty_l})  \Vert f_{K, +}^l \Vert_{L^1_x L^1_2} 
\\
\le& C_l  \Vert f_{K, +}^l \Vert_{L^2_xH^s_{\gamma/2 + s }}^2+ C_l (1 + K) \Vert f_{K, +}^l \Vert_{L^1_x L^1_2} 
\\
\le& \frac \epsilon 8 \Vert \langle v \rangle^{\alpha} f_{K, +}^l \Vert_{L^2_x H^1_v}^2 + C_{l, \epsilon}\Vert  f_{K, +}^l \Vert_{L^2_{x, v}}^2 + C_l (1 + K) \Vert f_{K, +}^l \Vert_{L^1_x L^1_2} .
\end{aligned}
\end{equation*}
where the constant $C$ is independent of $\eta$, together with Lemma \ref{L41} (2)  the proof is finished.
\end{proof}

Similarly as Lemma \ref{L43}  we have

\begin{lem} 
Suppose $\gamma+2s>-1, s \in (\frac 1 2, 1), \gamma \in (-3, 0]$. For any smooth function $f, g$, suppose $G= \mu +g\ge 0$ and $\delta_0>0$ in the cutoff function is small enough such that $G_\chi = u +g \chi$ satisfies 
\[
G_\chi  \ge 0,\quad \inf_{t, x} \Vert G_\chi  \Vert_{L^1_v} \ge A, \quad  \sup_{t, x}(\Vert G_\chi  \Vert_{L^1_2} +\Vert G \Vert_{L \log L} )\le B,
\]
For some constant $A, B >0$. Then for any
\[
T \ge 0, \quad \epsilon \in [0, 1], \quad j \ge 0,  \quad\tau>2,\quad  K>0,  \quad 12 \le  l \le k_0 ,\quad \alpha \ge  j +2,
\]
we have
\begin{equation*}
\begin{aligned} 
&\int_{0}^{T} \int_{\T^3} \int_{\R^3} \left| \langle v \rangle^j (1-\Delta_v)^{-\tau/2} ( \tilde{Q}_\eta(\mu +g\chi,  u+f )  \langle v \rangle^l f_{K, +}^l  )           \right| dv dx dt
\\
\le& C \Vert \langle v \rangle^{j/2} f_{K, + }^l (0, \cdot, \cdot ) \Vert_{L^2_{x, v}}^2 + C \Vert  \langle v \rangle^{\alpha}  f_{K, +}^l \Vert_{L^2_x H^1_v}^2  + C (1 + K)    \Vert f_{K, +}^l \Vert_{L^1_xL^1_{j+2}},
\end{aligned}
\end{equation*}
where the constant $C$ is independent of $\eta, \epsilon$.  Similar estimate holds for $f_{K, +}^l$ replaced by $h_{K, +}^l$
\end{lem}

\begin{proof}
Recall $W_k$ is defined in \eqref{W k}, we have
\begin{equation*}
\begin{aligned} 
\int_{\T^3}\int_{\R^3} Q_{\eta}(\mu +g \chi, \mu + f) f_{K, +}^l \langle v \rangle^l W_k dv dx  =& \int_{\T^3}\int_{\R^3} Q_{\eta}(\mu +g \chi, f-\frac K {\langle v \rangle^l}) f_{K, +}^l \langle v \rangle^l W_k dv dx
\\
&+ \int_{\T^3 } \int_{\R^3} Q_{\eta} (\mu +g \chi , \mu + \frac {K} {\langle v \rangle^l}) f_{K, +}^l \langle v \rangle^l W_k dv dx := H_1 + H_2.
\end{aligned}
\end{equation*}
By Lemma \ref{L81} we have
\begin{equation*}
\begin{aligned} 
&\int_{\T^3}\int_{\R^3} Q_{\eta} (\mu +g \chi, f-\frac K {\langle v \rangle^l}) f_{K, +}^l W_k \langle v \rangle^l dv dx
\\
\le &  \int_{\T^3 } \int_{\R^3} \int_{\R^3} \int_{\mathbb{S}^2 } (\mu_* + g_*\chi_*) f_{K, +}^l \langle v \rangle^l (f_{k, +}^l(v') \langle v' \rangle^l W_k(v') - f_{K, +}^l \langle v \rangle^l W_k(v)) b_\eta( \cos \theta) |v-v_*|^\gamma dv dv_* d\sigma  dx 
\\
\le& \int_{\T^3}\int_{\R^3} Q_{\eta} (\mu +g \chi, f_{K, +}^l \frac 1 {\langle v \rangle^l}) f_{K, +}^l \langle v \rangle^l W_k dv dx
\\
\le & C \int_{\T^3} \Vert \mu +g\chi \Vert_{L^1_{\gamma/2 +s +l +2} \cap L^2_2} \Vert f_{K, +}^l \Vert_{H^s_{\gamma/2 + s }} \Vert f_{K, +}^l W_k \Vert_{H^s_{\gamma/2 + s }}  dx
\\
\le & C \int_{\T^3} \Vert \mu +g\chi \Vert_{L^\infty_{k_0}} \Vert f_{K, +}^l \Vert_{H^s_{\gamma/2 + s +j}}^2 dx
\\
\le& C  \Vert f_{K, +}^l \Vert_{L^2_xH^s_{\gamma/2 + s +j }}^2
\\
\le& C \Vert \langle v \rangle^{\alpha} f_{K, +}^l \Vert_{L^2_x H^1_v}^2 + C \Vert  f_{K, +}^l \Vert_{L^2_{x, v}}^2 .
\end{aligned}
\end{equation*}
The other terms can be proved by following the proof of Lemma \ref{L42}. 
\end{proof}


Replacing  the use of Lemma \ref{L39} by Lemma \ref{L82} in the proof \ref{L46} we have

\begin{lem}
\label{L85} Suppose $\gamma+2s>-1, s \in ( 1 /2, 1), \gamma \in (-3, 0]$. Fo any smooth function $f, g$, suppose $G= \mu +g\ge 0$ and $\delta_0>0$ in the cutoff function is small enough such that $G_\chi = u +g \chi$ satisfies 
\[
G_\chi  \ge 0,\quad \inf_{t, x} \Vert G_\chi  \Vert_{L^1_v} \ge A, \quad  \sup_{t, x}(\Vert G_\chi  \Vert_{L^1_2} +\Vert G \Vert_{L \log L} )\le B,
\]
for some constant $A< B>0$. If $F_\eta =  \mu +f_\eta$ be a solution to \eqref{modified equation eta epsilon g}. Then for any $\epsilon \in [0, 1] $,  $12 \le l \le k_0 - 8,\alpha = 5, l\ge3+2\alpha$, for any $T \ge 0, 0 < s' <1/8 $, there exist $s" \in 0, (s' \frac  4  {2l+r})$ and $p^b := p^b(l, \gamma, s, s') >1$ such that for any $1< p <p^b$, we have
\[
\mathcal{E}_{0, p, s''} \le C_l e^{C_{l, \epsilon} T} \max_{ j \in \{ 1/p, p'/p  \} } \left( \Vert \langle v \rangle^l f_0\Vert_{L^2_{x, v}}^{2j}  +  T^j  \right), \quad p' = p /(2- p) ,
\]
where $C_{l, \epsilon} $ is independent of $\eta$ and $C_l$ is independent of both $\eta$ and $\epsilon$. The same estimates holds for $(-f)_{+}^l$ and its associated $\mathcal{E}_{0, p, s''}$.
\end{lem}

Similarly as Theorem \ref{T47} we have

\begin{thm}\label{T86}
Suppose $\gamma+2s>-1, s \in ( 1 /2, 1), \gamma \in (-3, 0]$. Fo any smooth function $f, g$, suppose $G= \mu +g\ge 0$ and $\delta_0>0$ in the cutoff function is small enough such that $G_\chi = u +g \chi$ satisfies 
\[
G_\chi  \ge 0,\quad \inf_{t, x} \Vert G_\chi  \Vert_{L^1_v} \ge A, \quad  \sup_{t, x}(\Vert G_\chi  \Vert_{L^1_2} +\Vert G \Vert_{L \log L} )\le B,
\]
for some constant $A< B>0$. If $F_\eta =  \mu +f_\eta$ be a solution to \eqref{modified equation eta epsilon g}. Then for any $\epsilon \in [0, 1] $,  $12 \le l \le k_0 - 8,\alpha = 5, l\ge3+2\alpha$, for any $T \ge 0 $, assume that the initial data $f_0$ satisfies
\[
 \Vert \langle  v \rangle^{l} f_0 \Vert_{L^2_{x, v}}  < + \infty, \quad\Vert \langle  v \rangle^{l} f_0 \Vert_{L^\infty_{x, v}}  < +\infty,
\]
there exist a constant $l_0>0$ which depends on $l$ such that if we additionally suppose that
\[
\sup_{t}  \Vert \langle  v \rangle^{l+l_0} f  \Vert_{L^2_{x, v} } \le C.
\]
Then it follows that
\[
\sup_{t \in [0, T]}\Vert  \langle v \rangle^l f\Vert_{L^\infty_{x, v}}  \le \max \{ 2\Vert  \langle v \rangle^l  f_0 \Vert_{L^\infty_{x, v}} , K_0^1 \},
\]
where 
\[
K_0^1 := C_l e^{C_{l, \epsilon} T} \max_{ 1 \le i \le 4} \max_{j \in \{ 1/p, p'/p \} }  \Vert \langle v \rangle^l f_0\Vert_{L^2_{x, v}}^{2j} +  T^j )^{\frac {\beta_i -1} {a_i}} , \quad p' = \frac p {2-p},
\]
where $C_{l, \epsilon} $ is independent of $\eta$ and $C_l$ is independent of both $\eta$ and $\epsilon$.
\end{thm}

It is clear that form Theorem \ref{T86}, for each $\epsilon >0$ if we let $T$ be small enough (with smallness depend on $\epsilon $, $\delta_0$ only) and $\Vert \langle  v \rangle^{l} f_0 \Vert_{L^2_{x, v} \cap L^\infty_{x, v}}$ small enough (with smallness independent of both $\epsilon$ and $\eta$), then 
\[
\sup_{t \in [0, T] } \Vert \langle v \rangle^{l} f   \Vert_{L^\infty_{x, v} } \le \delta_0,
\]
we can  now move to the existence of local solutions.

\begin{lem}\label{L87}  Suppose $g, f$ smooth and $\gamma \in (-3, 0], s \in  (1/2, 1), \gamma +2s >-1, \alpha =5$. Assume that $f$ is a solution to \eqref{modified equation eta epsilon}. Suppose $k_0$ and the initial data $f_0$ satisfies 
\[
k_0 -l_0(14) \ge 14 + 8 ,  \quad \Vert \langle v \rangle^{14+l_0 (14 )} f_0 \Vert_{L^2_{x, v}} < +\infty,\quad  \mu +f_0 \ge 0,  \quad \Vert\langle v \rangle^{14}  f_0  \Vert_{L^\infty_{x, v} \cap L^2_{x, v}} \le \delta_1, 
\]
where $l_0$ is defined in Theorem \ref{T86}. For any $\delta_0>0$ small enough, there exist  $\delta_1,\epsilon_* >0$  such that  for any $\epsilon \in (0, \epsilon_*]$, there exist a time $T_{\epsilon}$ (independent of $\eta$) such \eqref{modified equation} has a solution $f \in L^\infty([0, T] , L^2_xL^2_{14} (\T^3 \times \R^3)) $. Moreover the solution  $f$ satisfies
\[
\Vert  \langle v \rangle^{14}  f \Vert_{L^\infty( [0, T_\epsilon ] \times \T^3\times \R^3) }  \le \delta_0.
\]
\end{lem}

\begin{proof}
The proof is similar as  Lemma \ref{L52}.  when applying the fixed point argument, the coefficient obtained will depend on $\eta$, specifically, we will have
\begin{equation*}
\begin{aligned} 
&\int_{\T^3 }\int_{\R^3} Q_\eta (g \chi(\langle v \rangle^{k_0} g) - h \chi(\langle v \rangle^{k_0} h),  f_h ) (f_g -f_h)  \langle v \rangle^{2k} dvdx  
\\
\le& \frac C {\eta^{2s-2s_*}} \int_{\T^3 } \Vert g \chi(\langle v \rangle^{k_0} g) - h \chi(\langle v \rangle^{k_0} h) \Vert_{L^1_{\gamma+2s_* +k -\alpha} \cap L^2_2}  \Vert f_h  \Vert_{L^2_{\gamma +2s_* +k -\alpha }}  \Vert  (f_g -f_h) \Vert_{H^{2s_*}_{k+\alpha}} dx
\\
\le&C_\eta(\sup_x \Vert f_h\Vert_{L^2_{\gamma+2s_* +k -\alpha }} ) \Vert g- h\Vert_{L^2_xL^2_k}  \Vert  (f_g -f_h) \Vert_{L^2_x  H^{2s_*}_{k+\alpha}}
\\
\le & \frac {\epsilon} {8} \Vert \langle v \rangle^{\alpha+ k }  ( f_g- f_h) \Vert_{L^2_xH^1_v }^2  + C_{\epsilon, \eta} \Vert g- h\Vert_{L^2_xL^2_k}^2,
\end{aligned}
\end{equation*}
and other terms can be estimated the same way as Lemma \ref{L52}. First  we can deuce the existence in $T_{\epsilon, \eta}$ depends on both $\eta$ and $\epsilon$.  However, since all the priori estimates are independent of $\eta$, such a solution can be extend to $T_{\epsilon}$ which is independent of $\eta$.

\end{proof}

Once the existence of $f_\eta$ is shown, we can pass to the limit and return to the original operator $Q$ with $\chi$

\begin{lem} Suppose $g, f$ smooth and $\gamma \in (-3, 0], s \in  (1/2, 1), \gamma +2s >-1, \alpha =5$. Assume that $f$ is a solution to \eqref{modified equation}. Suppose $k_0$ and the initial data $f_0$ satisfies 
\[
k_0 -l_0(14) \ge 14 + 8 ,  \quad \Vert \langle v \rangle^{14+l_0 (14 )} f_0 \Vert_{L^2_{x, v}} < +\infty,\quad  \mu +f_0 \ge 0,  \quad \Vert\langle v \rangle^{14}  f_0  \Vert_{L^\infty_{x, v} \cap L^2_{x, v}} \le \delta_1, 
\]
where $l_0$ is defined in Theorem \ref{T86}.  For any $\delta_0>0$ small enough, there exist  $\delta_1,\epsilon_* >0$  such that  for any $\epsilon \in (0, \epsilon_*]$, there exist a time $T_{\epsilon}$ (independent of $\eta$) such \eqref{modified equation} has a solution $f \in L^\infty([0, T] , L^2_xL^2_{14} (\T^3 \times \R^3)) $. Moreover the solution  $f$ satisfies
\[
\Vert  \langle v \rangle^{14}  f \Vert_{L^\infty( [0, T_\epsilon ] \times \T^3\times \R^3) }  \le \delta_0.
\]
\end{lem}

\begin{proof}
By Lemma \ref{L87}, we know there exist a solution $f_\eta$ to \eqref{modified equation eta epsilon} satisfies
\[
\Vert  \langle v \rangle^{14}  f_\eta \Vert_{L^\infty_{t, x ,v} }  \le \delta_0, \quad \Vert f_\eta \Vert_{H^{s'}_{t, x} H^1_{k + \alpha }} \le C_0 < +\infty, \quad s' <1/8.
\]
Given the uniform polynomial decay and a diagonal argument, we can extract a subsequence, still denote as $f_\eta$ such that
\[
f_\eta \to f, \quad \hbox{strongly in } L^2_{t, x, v}([0, T] \times \T^3 \times \R^3).
\]
Our goal is to show that 
\[
Q_\eta (\mu + f_\eta \chi(\langle v \rangle^{k_0}  f_\eta), \mu + f_\eta)  \to Q (\mu + f \chi(\langle v \rangle^{k_0}  f), \mu + f),
\]
is distribution.  Using a test function $\phi$ we consider the difference 
\begin{equation*}
\begin{aligned} 
(Q_\eta(f_\eta \chi_\eta, f_\eta ) -Q(f\chi, f), \phi)
= & \int_{\R^3} \int_{\R^3} \int_{\mathbb{S}^2 } b_\eta(\cos \theta) |v-v_*|^\gamma f_{\eta, *} \chi_{\eta, *} f_\eta    (\phi(v') -\phi(v) ) dv dv_* d\sigma 
\\
&- \int_{\R^3} \int_{\R^3} \int_{\mathbb{S}^2 } b(\cos \theta) |v-v_*|^\gamma f_{ *} \chi_{ *} f    (\phi(v') -\phi(v) ) dv dv_* d\sigma 
\\
 =& \int_{\R^3} \int_{\R^3} \int_{\mathbb{S}^2 } b_\eta(\cos \theta) |v-v_*|^\gamma( f_{\eta, *} \chi_{\eta, *} f_\eta   -f_{ *} \chi_{ *} f    )   (\phi(v') -\phi(v) ) dv dv_* d\sigma 
 \\
& +\int_{\R^3} \int_{\R^3} \int_{\mathbb{S}^2 } (b_\eta (\cos \theta) - b(\cos \theta)) |v-v_*|^\gamma f_{ *} \chi_{ *} f      (\phi(v') -\phi(v) ) dv dv_* d\sigma 
\\
: =& E_1+E_2.
\end{aligned}
\end{equation*}
For the $E_1$ term by Lemma \ref{L215} we have
\begin{equation*}
\begin{aligned} 
E_1 \le& \left|\int_{\R^3} \int_{\R^3} \int_{\mathbb{S}^2 } b_\eta(\cos \theta) |v-v_*|^\gamma( f_{\eta, *} \chi_{\eta, *} f_\eta   -f_{ *} \chi_{ *} f    )   (\phi(v') -\phi(v) ) dv dv_* d\sigma  \right|
\\
\le&\int_{\R^3} \int_{\R^3}  |v-v_*|^\gamma  | f_{\eta, *} \chi_{\eta, *} f_\eta   -f_{ *} \chi_{ *} f    |  \left| \int_{\mathbb{S}^2 }  b_\eta(\cos \theta)(\phi(v') -\phi(v) )d\sigma \right|   dv dv_* 
\\
\le& C \Vert \phi \Vert_{W^{2, \infty}} \int_{\R^3} \int_{\R^3}  |v-v_*|^{2+\gamma}  | f_{\eta, *} \chi_{\eta, *} f_\eta   -f_{ *} \chi_{ *} f  | dv dv_* 
\\
\le &  C \Vert \phi \Vert_{W^{2, \infty}} \int_{\R^3} \int_{\R^3}  |v-v_*|^{2+\gamma} ( | f_{\eta, *} \chi_{\eta, *}  -f_{ *} \chi_{ *}  |  |f_\eta|  + |f_* \chi_*| |f_\eta - f | )dv dv_* 
\\
\le &  C \Vert \phi \Vert_{W^{2, \infty}} \int_{\R^3} \int_{\R^3}  |v-v_*|^{2+\gamma} ( | f_{\eta, *} -f_{ *}  |  |f_\eta|  + |f_* | |f_\eta - f | )dv dv_*,
\end{aligned}
\end{equation*}
which implies
\[
\Vert E_1\Vert_{L^2_{t, x}} \le  C \Vert \phi \Vert_{W^{2, \infty}} \Vert f_\eta \Vert_{L^2_{x}L^2_4} \Vert f_\eta  - f \Vert_{L^2_{x}L^2_4}  \to 0, \quad \eta \to 0.
\]
To estimate $E_2$, by symmetry (precisely antisymmetry) and Taylor expansion we have
\begin{equation*}
\begin{aligned}  
& \left| \int_{\R^3} \int_{\R^3} \int_{\mathbb{S}^2 } (b (\cos \theta) - b_\eta(\cos \theta)) |v-v_*|^\gamma f_{ *} \chi_{ *} f      (\phi(v') -\phi(v) ) dv dv_*  d\sigma\right|
\\
\le& \left| \int_{\R^3} \int_{\R^3} \int_{\mathbb{S}^2 } (b (\cos \theta) - b_\eta(\cos \theta)) |v-v_*|^\gamma f_{ *} \chi_{ *} f (v- v') \cdot \nabla_v \phi(v) dv dv_* d\sigma  \right|
\\
& + \frac 1 2  \left| \int_{\R^3} \int_{\R^3} \int_{\mathbb{S}^2 } (b (\cos \theta) - b_\eta(\cos \theta)) |v-v_*|^\gamma f_{ *} \chi_{ *} f (v-v') \otimes     ( v-v' ) \nabla_v^2 \phi(\bar{v})  dv dv_* d\sigma  \right|
\\
\le& \left| \int_{\R^3} \int_{\R^3} \int_{\mathbb{S}^2 }(1- \cos \theta) | b (\cos \theta) - b_\eta(\cos \theta) |  |v-v_*|^{1+\gamma} |f_{ *} \chi_{ *} | |f| |\nabla_v \phi(v) |  dv dv_* dv\sigma \right|
\\
& + \frac 1 2  \left| \int_{\R^3} \int_{\R^3} \int_{\mathbb{S}^2 } \sin^2 \theta  | b (\cos \theta) - b_\eta(\cos \theta)  | |v-v_*|^{2+\gamma} |f_{ *} \chi_{ *} |  |f|  |\nabla_v^2 \phi(\bar{v})|  dv dv_* d\sigma  \right|.
\end{aligned}
\end{equation*}
The integrands of the two last terms satisfies the uniform in $\eta$ bound
\[
(1- \cos \theta) |b (\cos \theta) - b_\eta(\cos \theta)| |v-v_*|^{1+\gamma} |f_{ *} \chi_{ *} | |f| |\nabla_v \phi(v) | \le 2\Vert \phi \Vert_{W^{1,\infty}} (1-\cos \theta) b (\cos \theta) |f_*| |f| |v-v_*|^{1+\gamma},
\]
and
\[
\sin^2 \theta |b (\cos \theta) - b_\eta(\cos \theta)| |v-v_*|^{2+\gamma} |f_{ *} \chi_{ *} |  |f|  |\nabla_v^2 \phi(\bar{v})|  \le 2\Vert \phi \Vert_{W^{2,\infty}}   \sin^2 \theta b (\cos \theta) |f_*| |f| |v-v_*|^{2 + \gamma}.
\]
Since the right hand of the inequalities above are integrable, we can apply the Lebesgue dominated convergence theorem and obtain that $E_2 \to 0$ as $\eta \to 0$, hence the theorem is proved.

\end{proof}

Recall that the only place that the restriction of a weak singularity enters is when we apply the fixed-point argument to obtain an approximate solution to equation. Once such restriction is bypassed, the rest of the results in Section \ref{section 6} and Section \ref{section 7}  all hold, since they are all proved for $s \in (0, 1)$. Theorem \ref{T11} is thus proved.

\end{document}